\documentclass[oneside,american]{amsart}
\usepackage[T1]{fontenc}
\usepackage[latin9]{inputenc}
\usepackage{enumitem}
\usepackage{amstext}
\usepackage{amsthm}
\usepackage{amssymb}

\makeatletter
\numberwithin{equation}{section}
\numberwithin{figure}{section}
\theoremstyle{plain}
\newtheorem{thm}{\protect\theoremname}[section]
\theoremstyle{definition}
\newtheorem{defn}[thm]{\protect\definitionname}
\theoremstyle{remark}
\newtheorem*{rem*}{\protect\remarkname}
\theoremstyle{plain}
\newtheorem{cor}[thm]{\protect\corollaryname}
\theoremstyle{plain}
\newtheorem{lem}[thm]{\protect\lemmaname}
\theoremstyle{plain}
\newtheorem{prop}[thm]{\protect\propositionname}
\theoremstyle{plain}
\newtheorem*{prop*}{\protect\propositionname}

\usepackage[pdftex,pdfpagelabels,bookmarks,hyperindex,hyperfigures]{hyperref}

\usepackage{thmtools}

\makeatother

\usepackage{babel}
\providecommand{\corollaryname}{Corollary}
\providecommand{\definitionname}{Definition}
\providecommand{\lemmaname}{Lemma}
\providecommand{\propositionname}{Proposition}
\providecommand{\remarkname}{Remark}
\providecommand{\theoremname}{Theorem}

\begin{document}

\title{Dimension of self-conformal measures associated to an exponentially
separated holomorphic IFS}

\author{\noindent Zhou Feng and Ariel Rapaport}



\thanks{Funded by the European Union (ERC, DIM-FRACTAL, 101217774). Views
and opinions expressed are however those of the authors only and do
not necessarily reflect those of the European Union or the European
Research Council Executive Agency. Neither the European Union nor
the granting authority can be held responsible for them. This research
was also supported by the Israel Science Foundation (grant No. 619/22).
ZF was supported in part by a Technion fellowship.}
\begin{abstract}
Let $\Phi$ be a holomorphic IFS on a bounded domain in $\mathbb{C}$.
Suppose that the following conditions hold: (1) the maps in $\Phi$
do not have a common fixed point; (2) there does not exist a regular
real-analytic curve which is invariant under all of the maps in $\Phi$;
(3) $\Phi$ is not holomorphically conjugate to a homothetic IFS;
(4) $\Phi$ is exponentially separated. Under these assumptions, we
show that the dimensions of the self-conformal measures associated
to $\Phi$, as well as the Hausdorff dimension of the associated self-conformal
set, attain their natural upper bounds. The proof combines recently
developed methods from the dimension theory of stationary fractal
measures with complex-analytic arguments.
\end{abstract}

\maketitle

\section{Introduction}

\subsection{Background}

Let $\Omega$ be a bounded domain in $\mathbb{R}^{d}$. A function
$f:\Omega\rightarrow\mathbb{R}^{d}$ is called conformal if it is
smooth and, for each $x\in\Omega$, there exist $U_{x}\in\mathrm{O}(d)$
and $r_{x}>0$ such that $f'(x)=r_{x}U_{x}$, where $\mathrm{O}(d)$
denotes the orthogonal group of $\mathbb{R}^{d}$. For $d=1$, this
is equivalent to $f$ being smooth with nonvanishing derivative. For
$d=2$, upon identifying $\mathbb{R}^{2}$ with $\mathbb{C}$, it
is equivalent to $f$ being holomorphic or antiholomorphic with nonvanishing
derivative. For $d\geq3$, Liouville\textquoteright s theorem implies
that $f$ is conformal if and only if it is the restriction of a Möbius
transformation.

Let $\Lambda$ be a finite index set, and for each $i\in\Lambda$
let $\varphi_{i}:\overline{\Omega}\rightarrow\Omega$ be given. Assume
that each $\varphi_{i}$ extends to a conformal injection on a domain
containing $\overline{\Omega}$ and that $0<\Vert\varphi_{i}'(x)\Vert_{\mathrm{op}}<1$
for all $i\in\Lambda$ and $x\in\overline{\Omega}$, where $\Vert\cdot\Vert_{\mathrm{op}}$
denotes the operator norm. Setting $\Phi:=\left\{ \varphi_{i}\right\} _{i\in\Lambda}$,
the collection $\Phi$ is called a conformal iterated function system
(IFS) on $\Omega$.

For a word $u=i_{1}...i_{n}\in\Lambda^{n}$, set $\varphi_{u}:=\varphi_{i_{1}}\circ...\circ\varphi_{i_{n}}$.
For a sufficiently large $n\geq1$, every map $\varphi_{u}$ with
$u\in\Lambda^{n}$ is a strict contraction of $\overline{\Omega}$.
Hence, by \cite{Hut}, there exists a unique nonempty compact set
$K_{\Phi}\subset\Omega$ satisfying $K_{\Phi}=\cup_{i\in\Lambda}\varphi_{i}(K_{\Phi})$.
This set is called the self-conformal set, or the attractor, associated
with $\Phi$.

Certain natural measures are also associated with $\Phi$. Let $p=(p_{i})_{i\in\Lambda}$
be a positive probability vector. By \cite{Hut}, there exists a unique
Borel probability measure $\mu$ on $\Omega$ satisfying $\mu=\sum_{i\in\Lambda}p_{i}\cdot\varphi_{i}\mu$,
where $\varphi_{i}\mu$ denotes the pushforward of $\mu$ under $\varphi_{i}$.
The measure $\mu$, which is supported on $K_{\Phi}$, is called the
self-conformal measure corresponding to $\Phi$ and $p$.

By the work of Feng and Hu \cite{FH-dimension}, the measure $\mu$
is exact dimensional. That is, there exists a number $\dim\mu$, called
the dimension of $\mu$, such that
\[
\underset{\delta\downarrow0}{\lim}\frac{\log\mu\left(B(x,\delta)\right)}{\log\delta}=\dim\mu\text{ for }\mu\text{-a.e. }x\in\Omega.
\]
Here, $B(x,\delta)$ denotes the closed ball in $\mathbb{R}^{d}$
with centre $x$ and radius $\delta$. Computing the dimension of
self-conformal measures is a natural and important problem in fractal
geometry.

Write $H(p)$ for the entropy of $p$ and $\chi:=\chi\left(\Phi,p\right)$
for the Lyapunov exponent associated to $\Phi$ and $p$. That is,
\[
H(p):=-\sum_{i\in\Lambda}p_{i}\log p_{i}\text{ and }\chi:=-\sum_{i\in\Lambda}p_{i}\int\log\Vert\varphi_{i}'(x)\Vert_{\mathrm{op}}\:d\mu(x),
\]
where we always use $2$ as the base of the logarithm. It is easy
to show that $H(p)/\chi$ is always an upper bound for $\dim\mu$.
Moreover, in the absence of certain obvious obstructions, one generally
expects that
\begin{equation}
\dim\mu=\min\left\{ d,H(p)/\chi\right\} ,\label{eq:dim mu =00003D min=00007Bd,H(p)/chi=00007D}
\end{equation}
although this is often difficult to verify.

The IFS $\Phi$ is said to satisfy the strong separation condition
(SSC) if the sets $\left\{ \varphi_{i}\left(K_{\Phi}\right)\right\} _{i\in\Lambda}$
are disjoint, in which case it is easy to establish (\ref{eq:dim mu =00003D min=00007Bd,H(p)/chi=00007D}).
Additionally, in various settings, this equality can be shown to hold
almost surely under a natural randomization of the parameters (see,
e.g., \cite[Section 7]{MR1852098} and \cite[Chapter 14]{MR4661364}).
It is therefore desirable to find explicit conditions ensuring (\ref{eq:dim mu =00003D min=00007Bd,H(p)/chi=00007D})
that are much milder than the restrictive SSC.

The IFS $\Phi$ is called self-similar if each $\varphi_{i}$ is a
contracting similarity. That is, for each $i\in\Lambda$, there exist
$0<r_{i}<1$, $U_{i}\in\mathrm{O}(d)$, and $a_{i}\in\mathbb{R}^{d}$
such that $\varphi_{i}(x)=r_{i}U_{i}x+a_{i}$ for all $x\in\Omega$.
In recent years, major progress has been made on the above problem
in the self-similar case, initiated by Hochman's seminal work \cite{Ho1}
on the real line. Let $\Vert\cdot\Vert_{\Omega}$ denote the supremum
norm on $\Omega$. That is, $\Vert f\Vert_{\Omega}=\sup_{x\in\Omega}\left|f(x)\right|$
for a bounded function $f:\Omega\rightarrow\mathbb{R}^{d}$.
\begin{defn}
\label{def:exp sep}We say that $\Phi$ is exponentially separated
if there exists $c>0$ such that for infinitely many $n\in\mathbb{Z}_{>0}$,
\[
\Vert\varphi_{u_{1}}-\varphi_{u_{2}}\Vert_{\Omega}\ge c^{n}\text{ for all distinct }u_{1},u_{2}\in\Lambda^{n}.
\]
\end{defn}

The exponential separation assumption is significantly weaker than
both the SSC and the well-known open set condition, which we do not
state here.

The main result of \cite{Ho1} states that (\ref{eq:dim mu =00003D min=00007Bd,H(p)/chi=00007D})
holds whenever $d=1$ and $\Phi$ is self-similar and exponentially
separated. Later, in \cite{Ho}, Hochman extended this result to higher
dimensions. Kittle and Kogler \cite{kittle2025dimension} subsequently
extended Hochman\textquoteright s higher-dimensional result to the
contracting-on-average setting. In the uniformly contracting setting
considered here, it follows directly from \cite[Theorem 1.5]{Ho}
that (\ref{eq:dim mu =00003D min=00007Bd,H(p)/chi=00007D}) holds
whenever the following conditions are all satisfied:
\begin{enumerate}[label={(A\arabic*)}]
\item \label{enu:hoch R^d cond SS}$\Phi$ is self-similar;
\item \label{enu:hoch R^d cond affine irred}there is no proper affine subspace
of $\mathbb{R}^{d}$ that is invariant under all the maps in $\Phi$;
\item \label{enu:hoch R^d cond linear irred}there is no nontrivial linear
subspace of $\mathbb{R}^{d}$ that is invariant under all the linear
parts of the maps in $\Phi$;
\item \label{enu:hoch R^d cond exp sep}$\Phi$ is exponentially separated.
\end{enumerate}
It is a challenging and important open problem to relax these assumptions
as much as possible.

Note that conditions \ref{enu:hoch R^d cond affine irred} and \ref{enu:hoch R^d cond linear irred}
are necessary. Indeed, suppose that there exists a proper affine subspace
$V$ of $\mathbb{R}^{d}$ that is invariant under all the maps in
$\Phi$. Then $K_{\Phi}\subset V$ and hence $\mu(V)=1$, which implies
that $\dim\mu\leq\dim V$. Consequently, (\ref{eq:dim mu =00003D min=00007Bd,H(p)/chi=00007D})
fails whenever $H(p)/\chi>\dim V$. For the case $d=1$, note that
the condition \ref{enu:hoch R^d cond affine irred} is equivalent
to requiring that the maps in $\Phi$ do not all share a common fixed
point. This assumption is not stated explicitly in \cite{Ho1}, since
in the one-dimensional self-similar setting it follows from exponential
separation.

Concerning the necessity of condition \ref{enu:hoch R^d cond linear irred},
as demonstrated in \cite[Example 1.2]{Ho}, one can construct exponentially
separated self-similar measures $\mu$ for which condition \ref{enu:hoch R^d cond affine irred}
holds but (\ref{eq:dim mu =00003D min=00007Bd,H(p)/chi=00007D}) fails.
In these examples, $\mu$ is saturated, in a sense made precise in
\cite{Ho}, along translates of a nontrivial linear subspace of $\mathbb{R}^{d}$
which is invariant under the linear parts of the maps in the IFS.
Note that when $d=2$ and the maps in $\Phi$ are orientation-preserving,
the condition \ref{enu:hoch R^d cond linear irred} fails if and only
if the maps in $\Phi$ are all homotheties. That is, for each $i\in\Lambda$,
the map $\varphi_{i}$ is of the form $\varphi_{i}(x)=r_{i}x+a_{i}$
for some $0\neq r_{i}\in(-1,1)$ and $a_{i}\in\mathbb{R}^{d}$.

Regarding the exponential separation assumption, note that it is necessary
to assume some form of separation. Indeed, if $\Phi$ does not generate
a free semigroup, in which case it is said to have exact overlaps,
then (\ref{eq:dim mu =00003D min=00007Bd,H(p)/chi=00007D}) necessarily
fails whenever $\dim\mu<d$. One of the most important open problems
in fractal geometry, called the exact overlaps conjecture, asserts
that in the self-similar setting on the real line, (\ref{eq:dim mu =00003D min=00007Bd,H(p)/chi=00007D})
can only fail in the presence of exact overlaps (see \cite{MR3966837,Var_ICM}
for further discussion). This conjecture has recently been verified
in various situations (see \cite{MR4929167,Rap-EO,rapaport20203maps,Var-Bernoulli}).
In higher dimensions, given the above result from \cite{Ho}, it is
reasonable to expect that, when conditions \ref{enu:hoch R^d cond SS},
\ref{enu:hoch R^d cond affine irred} and \ref{enu:hoch R^d cond linear irred}
hold, (\ref{eq:dim mu =00003D min=00007Bd,H(p)/chi=00007D}) should
hold whenever $\Phi$ generates a free semigroup.

It is easy to verify that $\Phi$ has exact overlaps if and only if
there exist $n\ge1$ and distinct $u_{1},u_{2}\in\Lambda^{n}$ such
that $\varphi_{u_{1}}=\varphi_{u_{2}}$. Thus, exponential separation
is a stronger quantitative version of the assumption that there are
no exact overlaps.

Especially relevant to us is the extension of the above results beyond
the self-similar setting. Hochman and Solomyak \cite{HS} proved that
(\ref{eq:dim mu =00003D min=00007Bd,H(p)/chi=00007D}) holds for $d=1$
when $\Phi$ is exponentially separated, its maps are restrictions
of Möbius transformations, and the semigroup generated by $\Phi$,
viewed as a subsemigroup of $\mathrm{SL}(2,\mathbb{R})$, is strongly
irreducible and proximal. Rapaport and Ren \cite{Rap_Ren_Fur_on_CP1}
obtained the analogous result for $d=2$, with $\mathrm{SL}(2,\mathbb{C})$
in place of $\mathrm{SL}(2,\mathbb{R})$, under the additional requirement
that there be no generalized circle invariant under all the maps in
$\Phi$\footnote{The results of \cite{HS} and \cite{Rap_Ren_Fur_on_CP1} are actually
more general, as they do not require $\mu$ to be compactly supported.}.

In all the above results, the semigroup generated by $\Phi$ naturally
embeds into a finite-dimensional Lie group, thereby avoiding significant
difficulties that arise when this is not the case. In \cite{Rap_analytic_on_R},
Rapaport made substantial progress beyond this situation by proving
that (\ref{eq:dim mu =00003D min=00007Bd,H(p)/chi=00007D}) holds
for $d=1$ when all the maps in $\Phi$ are real-analytic, have no
common fixed point, and $\Phi$ is exponentially separated.

It is natural to try to extend the preceding result to higher dimensions.
As pointed out above, when $d\geq3$, conformal maps are restrictions
of Möbius transformations. Hence, in this case, self-conformal IFSs
may be viewed naturally as finite subsets of the projective orthogonal
group $\mathrm{PO}(d+1,1)$ (see, e.g., \cite[Theorem 2.9]{conf-book}).
Consequently, proving the desired equality for $d\geq3$ should not
require extending the arguments from \cite{Rap_analytic_on_R} that
were developed to handle the absence of an ambient finite-dimensional
Lie group.

In the present paper, we extend the result of \cite{Rap_analytic_on_R}
to the case $d=2$. Given a holomorphic IFS on a domain $\Omega\subset\mathbb{C}$,
we establish (\ref{eq:dim mu =00003D min=00007Bd,H(p)/chi=00007D})
under mild assumptions that, apart from exponential separation, are
necessary. This also extends the planar case of Hochman\textquoteright s
self-similar result to the holomorphic setting.

\subsection{\label{subsec:Setup-and-the-main-result}Setup and the main result}

Throughout the rest of the paper, let $\Omega$ be a bounded domain
in $\mathbb{C}$, let $\Lambda$ be a finite index set, and for each
$i\in\Lambda$ let $\varphi_{i}:\overline{\Omega}\rightarrow\mathbb{C}$
be given. We shall always assume that
\begin{enumerate}[label={(\roman*)}]
\item \label{enu:holo IFS cond 1}there exists a domain $\Omega'\subset\mathbb{C}$
containing $\overline{\Omega}$ such that, for each $i\in\Lambda$,
the map $\varphi_{i}$ extends to a holomorphic injection on $\Omega'$;
\item \label{enu:holo IFS cond 2}for each $i\in\Lambda$, $\varphi_{i}\left(\overline{\Omega}\right)\subset\Omega$
and $0<\left|\varphi_{i}'(z)\right|<1$ for all $z\in\overline{\Omega}$.
\end{enumerate}
Setting $\Phi:=\left\{ \varphi_{i}\right\} _{i\in\Lambda}$, the collection
$\Phi$ is called a holomorphic IFS on $\Omega$. Thus, viewing $\mathbb{C}$
as $\mathbb{R}^{2}$, every holomorphic IFS is in particular a conformal
IFS.

As before, let $K_{\Phi}$ denote the self-conformal set associated
with $\Phi$. Fix a positive probability vector $p=(p_{i})_{i\in\Lambda}$,
and let $\mu$ denote the self-conformal measure associated with $\Phi$
and $p$. We continue to write $H(p)$ for the entropy of $p$ and
$\chi:=\chi(\Phi,p)$ for the Lyapunov exponent associated with $\Phi$
and $p$.

To state our results, we need the following definitions.
\begin{defn}
Given an interval $J\subset\mathbb{R}$, a function $\gamma:J\rightarrow\mathbb{C}$
is said to be real-analytic if, for each $t_{0}\in J$, there exist
$\epsilon>0$ and numbers $a_{0},a_{1},\ldots\in\mathbb{C}$ such
that $\gamma(t)=\sum_{n\geq0}a_{n}(t-t_{0})^{n}$ for all $t\in J\cap(t_{0}-\epsilon,t_{0}+\epsilon)$.
\end{defn}

\begin{defn}
\label{def:reg real-anal curve}A subset $\Gamma\subset\mathbb{C}$
is said to be a regular real-analytic curve if it is an embedded $1$-dimensional
real-analytic submanifold of $\mathbb{C}$. That is, for every $z_{0}\in\Gamma$
there exist an open subset $U\subset\mathbb{C}$ containing $z_{0}$,
an open interval $J\subset\mathbb{R}$, and a real-analytic map $\gamma:J\rightarrow\mathbb{C}$
such that $\gamma(J)=\Gamma\cap U$, $\gamma$ is a homeomorphism
onto $\Gamma\cap U$, and $\gamma'(t)\ne0$ for all $t\in J$. When
$\Gamma\subset\Omega$, we say that $\Gamma$ is $\Phi$-invariant
if $\varphi_{i}(\Gamma)\subset\Gamma$ for each $i\in\Lambda$.
\end{defn}

\begin{rem*}
Note that a regular real-analytic curve need not be connected.
\end{rem*}
\begin{defn}
Given an open set $U\subset\mathbb{C}$, a function $f:U\rightarrow\mathbb{C}$
is said to be homothetic if there exist $0\neq r\in\mathbb{R}$ and
$a\in\mathbb{C}$ such that $f(z)=rz+a$ for all $z\in U$. We say
that $\Phi$ is holomorphically conjugate to a homothetic IFS if there
exists an injective holomorphic map\footnote{Recall that an injective holomorphic map $\psi:\Omega\rightarrow\mathbb{C}$
is biholomorphic onto its image. That is, $\psi(\Omega)$ is open
and $\psi^{-1}:\psi(\Omega)\rightarrow\Omega$ is also holomorphic;
see, e.g., \cite{MR1976398}.} $\psi:\Omega\rightarrow\mathbb{C}$ such that $\psi\circ\varphi_{i}\circ\psi^{-1}:\psi(\Omega)\rightarrow\mathbb{C}$
is homothetic for each $i\in\Lambda$.
\end{defn}

Throughout the paper, we use exponential separation in the sense of
Definition \ref{def:exp sep}, identifying $\mathbb{C}$ with $\mathbb{R}^{2}$.
The following theorem is our main result.
\begin{thm}
\label{thm:main thm}Let $\Omega\subset\mathbb{C}$ be a bounded domain
and let $\Phi=\left\{ \varphi_{i}\right\} _{i\in\Lambda}$ be a holomorphic
IFS on $\Omega$ (so that conditions \ref{enu:holo IFS cond 1} and
\ref{enu:holo IFS cond 2} hold). Suppose further that
\begin{enumerate}[label={(M\arabic*)}]
\item \label{enu:main thm cond no common fix}the maps in $\Phi$ do not
have a common fixed point;
\item \label{enu:main thm cond no inv curve}there does not exist a regular
real-analytic curve $\Gamma\subset\Omega$ which is $\Phi$-invariant
and closed in $\Omega$;
\item \label{enu:main thm cond not holo conj}$\Phi$ is not holomorphically
conjugate to a homothetic IFS;
\item $\Phi$ is exponentially separated.
\end{enumerate}
Then for every positive probability vector $p=(p_{i})_{i\in\Lambda}$,
\begin{equation}
\dim\mu=\min\left\{ 2,H(p)/\chi(\Phi,p)\right\} ,\label{eq:dim formula in main thm}
\end{equation}
where $\mu$ is the self-conformal measure associated to $\Phi$ and
$p$.
\end{thm}

Let us make a few remarks on the assumptions appearing in the theorem.
Conditions \ref{enu:holo IFS cond 1} and \ref{enu:holo IFS cond 2}
are standard in the context of conformal IFSs. Although condition
\ref{enu:holo IFS cond 1} requires the maps in $\Phi$ to be holomorphic,
our arguments also apply, with only minor modifications, when some
of them are anti-holomorphic. For simplicity of exposition, we restrict
to the holomorphic case. The injectivity assumption in condition \ref{enu:holo IFS cond 1}
is needed to obtain certain bi-Lipschitz estimates (see Section \ref{subsec:A-bi-Lipchitz-estimate})
and to carry out various compactness arguments (see Lemma \ref{lem:injective or constant}).
Condition \ref{enu:holo IFS cond 2} ensures that sufficiently long
compositions of maps in $\Phi$ are uniformly contracting and that
the bounded distortion property holds (see Lemma \ref{lem:bounded distortion}).

Assumption \ref{enu:main thm cond no common fix} is equivalent to
$K_{\Phi}$ not being a singleton. This is clearly necessary for the
validity of (\ref{eq:dim formula in main thm}), unless $\Phi$ consists
of a single map.

Assumption \ref{enu:main thm cond no inv curve} also cannot be omitted
in general. Indeed, if there exists a regular real-analytic curve
$\Gamma\subset\Omega$ that is $\Phi$-invariant and closed in $\Omega$,
then $K_{\Phi}\subset\Gamma$, and hence $\dim\mu\leq1$. Consequently,
(\ref{eq:dim formula in main thm}) fails whenever $H(p)/\chi>1$.
In the self-similar case, it is not difficult to see that assumptions
\ref{enu:main thm cond no common fix} and \ref{enu:main thm cond no inv curve}
together are equivalent to Hochman\textquoteright s affine irreducibility
condition \ref{enu:hoch R^d cond affine irred}. Notice that assumption
\ref{enu:main thm cond no inv curve} excludes only $\Phi$-invariant
curves that are closed in $\Omega$. This slight strengthening of
the theorem will be useful in subsequent applications of our main
result.

Assumption \ref{enu:main thm cond not holo conj} is again necessary
in general. Indeed, the example in \cite[Example 1.2]{Ho} mentioned
above shows that the homothetic case must be excluded. Since holomorphic
conjugacy leaves both $\dim\mu$ and $\chi$ unchanged, one must also
exclude IFSs that are holomorphically conjugate to a homothetic IFS.
In the holomorphic self-similar case, it is not difficult to see that
assumption \ref{enu:main thm cond not holo conj} is equivalent to
Hochman\textquoteright s linear irreducibility condition \ref{enu:hoch R^d cond linear irred}.

Although assumptions \ref{enu:main thm cond no inv curve} and \ref{enu:main thm cond not holo conj}
are necessary, they may be difficult to verify in concrete situations.
In Corollary \ref{cor:ver suff criterion} below, we replace these
two assumptions with a single condition that implies both and is often
easier to verify.

Finally, as in the self-similar case, although exponential separation
is a rather mild condition, it is desirable to weaken it to the necessary
assumption that there are no exact overlaps. Since the exact overlaps
conjecture remains open in full generality even for self-similar IFSs
on the real line, this is currently well beyond our reach.

\subsection{Dimension of self-conformal sets}

Using Theorem \ref{thm:main thm}, we can compute the dimension of
the self-conformal set $K_{\Phi}$. In what follows, $\dim_{H}$ denotes
Hausdorff dimension.

For each $t\ge0$ set
\[
P_{\Phi}(t):=\underset{n\rightarrow\infty}{\lim}\frac{1}{n}\log\left(\sum_{u\in\Lambda^{n}}\Vert\varphi_{u}'\Vert_{\Omega}^{t}\right),
\]
where $\Vert\cdot\Vert_{\Omega}$ denotes the supremum norm on $\Omega$
and the limit exists by sub-additivity. It is easy to verify that
$P_{\Phi}(0)=\log|\Lambda|$, $P_{\Phi}(t)\rightarrow-\infty$ as
$t\rightarrow\infty$, and $P_{\Phi}$ is strictly decreasing and
continuous. Thus, there exists a unique $s(\Phi)\ge0$ such that $P_{\Phi}\left(s(\Phi)\right)=0$.
Following \cite{MR4661364}, we call the number $s(\Phi)$ the conformal
similarity dimension associated to $\Phi$. It is well known and easy
to show that $s(\Phi)$ is always an upper bound of $\dim_{H}K_{\Phi}$.
\begin{cor}
\label{cor:main cor for sets}Let $\Omega\subset\mathbb{C}$ be a
bounded domain, and let $\Phi=\left\{ \varphi_{i}\right\} _{i\in\Lambda}$
be a holomorphic IFS on $\Omega$ satisfying the assumptions of Theorem
\ref{thm:main thm}. Then
\begin{equation}
\dim_{H}K_{\Phi}=\min\left\{ 2,s(\Phi)\right\} ,\label{eq:exp dim formula for sets}
\end{equation}
where $K_{\Phi}$ denotes the self-conformal set associated to $\Phi$.
\end{cor}

The proof of Corollary \ref{cor:main cor for sets} is given in Section
\ref{sec:Proof-of-main Corollary for sets}.

\subsection{A verifiable sufficient criterion}

The following corollary provides a convenient sufficient criterion
for applying Theorem \ref{thm:main thm} and Corollary \ref{cor:main cor for sets}.
Let $\Lambda^{*}$ denote the set of finite words over $\Lambda$.
\begin{cor}
\label{cor:ver suff criterion}Let $\Omega\subset\mathbb{C}$ be a
bounded domain and let $\Phi=\left\{ \varphi_{i}\right\} _{i\in\Lambda}$
be a holomorphic IFS on $\Omega$. Suppose that
\begin{enumerate}
\item the maps in $\Phi$ do not have a common fixed point;
\item \label{enu:exists fix pt with non triv rot}there exists $u\in\Lambda^{*}$
such that $\varphi_{u}'(z_{u})\notin\mathbb{R}$, where $z_{u}$ denotes
the unique fixed point of $\varphi_{u}$ in $\Omega$;
\item $\Phi$ is exponentially separated.
\end{enumerate}
Then $\dim_{H}K_{\Phi}=\min\left\{ 2,s(\Phi)\right\} $, where $K_{\Phi}$
is the self-conformal set associated to $\Phi$. Moreover, for any
positive probability vector $p=(p_{i})_{i\in\Lambda}$, we have $\dim\mu=\min\left\{ 2,\frac{H(p)}{\chi(\Phi,p)}\right\} $,
where $\mu$ is the self-conformal measure associated to $\Phi$ and
$p$.
\end{cor}

\begin{proof}
By Theorem \ref{thm:main thm} and Corollary \ref{cor:main cor for sets},
it suffices to verify assumptions \ref{enu:main thm cond no inv curve}
and \ref{enu:main thm cond not holo conj} from Theorem \ref{thm:main thm}.

Assume by contradiction that there exists a regular real-analytic
curve $\Gamma\subset\Omega$ which is $\Phi$-invariant and closed
in $\Omega$. Being $\Phi$-invariant and closed in $\Omega$, it
is clear that $\Gamma$ contains $K_{\Phi}$. Moreover, it is clear
that $z_{u}\in K_{\Phi}$, and so $z_{u}\in\Gamma$. Write $V$ for
the tangent space of $\Gamma$ at $z_{u}$, so that $V$ is a $1$-dimensional
real subspace of $\mathbb{C}$. Since $\varphi_{u}$ is smooth on
$\Omega$ and $\Gamma$ is $\Phi$-invariant and embedded in $\Omega$,
the restriction of $\varphi_{u}$ to $\Gamma$ is smooth as a map
from the manifold $\Gamma$ into itself. Hence, from $\varphi_{u}(z_{u})=z_{u}$,
we obtain that $\varphi_{u}'(z_{u})V=V$. But this contradicts $\varphi_{u}'(z_{u})\notin\mathbb{R}$,
and so there cannot exist a regular real-analytic curve $\Gamma\subset\Omega$
which is $\Phi$-invariant and closed in $\Omega$.

Next, assume by contradiction that there exists a holomorphic and
injective $\psi:\Omega\rightarrow\mathbb{C}$ such that $\psi\circ\varphi_{i}\circ\psi^{-1}:\psi(\Omega)\rightarrow\mathbb{C}$
is homothetic for each $i\in\Lambda$. It is clear that $h:=\psi\circ\varphi_{u}\circ\psi^{-1}$
is homothetic, and so $h'\left(\psi(z_{u})\right)\in\mathbb{R}$.
On the other hand, by the chain rule and since $\varphi_{u}(z_{u})=z_{u}$,
we have $h'\left(\psi(z_{u})\right)=\varphi_{u}'(z_{u})\notin\mathbb{R}.$
This contradiction shows that $\Phi$ is not holomorphically conjugate
to a homothetic IFS, which completes the proof of the corollary.
\end{proof}
\begin{rem*}
In the proof above, we showed that condition (\ref{enu:exists fix pt with non triv rot})
appearing in Corollary \ref{cor:ver suff criterion} implies assumptions
\ref{enu:main thm cond no inv curve} and \ref{enu:main thm cond not holo conj}.
It is unclear to us whether the converse holds, which seems to be
an interesting question.
\end{rem*}
In the next subsection, we use Corollary \ref{cor:ver suff criterion}
to show that the expected dimension formulas (\ref{eq:dim formula in main thm})
and (\ref{eq:exp dim formula for sets}) hold for all parameters outside
a set of Hausdorff dimension zero in the context of one-parameter
families of holomorphic IFSs. By contrast, we do not pursue here the
interesting problem of applying our results to obtain explicit non-self-similar
examples for which these formulas hold. Nevertheless, we briefly comment
on this direction.

In \cite{BaranyKolossvaryTroscheit}, Bárány, Kolossváry and Troscheit
established a sufficient condition for an analytic IFS on $\mathbb{R}$
to be exponentially separated. Combining their result with the result
of Rapaport \cite{Rap_analytic_on_R} mentioned above yields explicit
examples of overlapping, non-self-similar analytic IFSs on $\mathbb{R}$
whose associated self-conformal sets and measures satisfy the expected
dimension formulas. It seems likely that the results of \cite{BaranyKolossvaryTroscheit}
could be extended to the holomorphic setting in a form that, when
combined with Corollary \ref{cor:ver suff criterion}, would yield
explicit examples in our setting. We leave this as a problem for further
research.

\subsection{One-parameter families of holomorphic IFSs}

In this subsection, we combine Corollary \ref{cor:ver suff criterion}
with a result of Solomyak and Takahashi \cite{MR4300232} to deduce
a statement about one-parameter families of holomorphic IFSs.

Let $\Omega'\subset\mathbb{C}$ be a domain containing $\overline{\Omega}$,
let $J$ be a compact interval in $\mathbb{R}$, and suppose that
the index set $\Lambda$ has cardinality at least $2$. For each $i\in\Lambda$,
let $\{\varphi_{i,t}\}_{t\in J}$ be a one-parameter family of maps
from $\Omega'$ into $\mathbb{C}$. Suppose that
\begin{enumerate}
\item $\varphi_{i,t}:\Omega'\rightarrow\mathbb{C}$ is holomorphic and injective
for all $i\in\Lambda$ and $t\in J$;
\item \label{enu:parametric cond uni cont}$\varphi_{i,t}\left(\overline{\Omega}\right)\subset\Omega$
and $0<\left|\varphi_{i,t}'(z)\right|<1$ for all $i\in\Lambda$,
$t\in J$ and $z\in\overline{\Omega}$;
\item for each $i\in\Lambda$, the map $(t,z)\mapsto\varphi_{i,t}(z)$ extends
to a real-analytic map on an open neighborhood of $J\times\Omega'$.
\end{enumerate}
For $t\in J$, set $\Phi_{t}:=\left\{ \varphi_{i,t}\right\} _{i\in\Lambda}$
and let $K_{\Phi_{t}}$ be the self-conformal set corresponding to
$\Phi_{t}$. Given a probability vector $p=(p_{i})_{i\in\Lambda}$,
write $\mu_{t,p}$ for the self-conformal measure associated to $\Phi_{t}$
and $p$.

Fix some $z_{0}\in\Omega$. Given $\omega,\eta\in\Lambda^{\mathbb{N}}$,
let $F_{\omega,\eta}:J\rightarrow\mathbb{C}$ be defined by
\[
F_{\omega,\eta}(t):=\underset{n\rightarrow\infty}{\lim}\left(\varphi_{\omega_{0},t}\circ...\circ\varphi_{\omega_{n},t}(z_{0})-\varphi_{\eta_{0},t}\circ...\circ\varphi_{\eta_{n},t}(z_{0})\right)\text{ for }t\in J,
\]
where the existence of the limit follows easily from assumption (\ref{enu:parametric cond uni cont}).
We say that the family of IFSs $\left\{ \Phi_{t}\right\} _{t\in J}$
is real (resp. imaginary) non-degenerate if, for all distinct $\omega,\eta\in\Lambda^{\mathbb{N}}$,
the function $\mathrm{Re}\circ F_{\omega,\eta}$ (resp. $\mathrm{Im}\circ F_{\omega,\eta}$)
is not identically zero.

Given $i_{1}...i_{n}=u\in\Lambda^{*}$ and $t\in J$, set $\varphi_{u,t}:=\varphi_{i_{1},t}\circ...\circ\varphi_{i_{n},t}$
and write $z_{u,t}$ for the unique fixed point of $\varphi_{u,t}$.
\begin{cor}
Suppose that $\left\{ \Phi_{t}\right\} _{t\in J}$ is either real
non-degenerate or imaginary non-degenerate. Assume moreover that there
exist $u\in\Lambda^{*}$ and $t_{0}\in J$ such that $\varphi_{u,t_{0}}'\left(z_{u,t_{0}}\right)\notin\mathbb{R}$.
Then there exists $E\subset J$ with $\dim_{H}E=0$ such that, for
all $t\in J\setminus E$, we have $\dim_{H}K_{\Phi_{t}}=\min\left\{ 2,s(\Phi_{t})\right\} $
and $\dim\mu_{t,p}=\min\left\{ 2,H(p)/\chi\left(\Phi_{t},p\right)\right\} $
for every positive probability vector $p=(p_{i})_{i\in\Lambda}$.
\end{cor}

\begin{proof}
Let $E_{1}$ be the set of $t\in J$ for which the maps in $\Phi_{t}$
have a common fixed point. Let $u\in\Lambda^{*}$ and $t_{0}\in J$
be as in the statement of the corollary, and denote by $E_{2}$ the
set of $t\in J$ for which $\varphi_{u,t}'\left(z_{u,t}\right)\in\mathbb{R}$.
Additionally, write $E_{3}$ for the set of $t\in J$ for which $\Phi_{t}$
is not exponentially separated. By Corollary \ref{cor:ver suff criterion},
in order to prove the corollary, it suffices to show that the sets
$E_{1}$, $E_{2}$ and $E_{3}$ are all of Hausdorff dimension zero.

Since $\left\{ \Phi_{t}\right\} _{t\in J}$ is real or imaginary non-degenerate,
it follows from \cite[Theorem 2.10]{MR4300232} that $\dim_{H}E_{3}=0$.
We next complete the proof by showing that $E_{1}$ and $E_{2}$ are
finite.

Let $i_{1},i_{2}\in\Lambda$ be distinct. For $j=1,2$, let $i_{j}^{\infty}$
denote the infinite word in $\Lambda^{\mathbb{N}}$ consisting only
of the letter $i_{j}$, and set $F:=F_{i_{1}^{\infty},i_{2}^{\infty}}$.
Given $t\in J$, note that $\varphi_{i_{1},t}$ and $\varphi_{i_{2},t}$
share a common fixed point if and only if $F(t)=0$. Consequently,
we have $E_{1}\subset F^{-1}\{0\}$. Since $\left\{ \Phi_{t}\right\} _{t\in J}$
is real or imaginary non-degenerate, $F$ is not identically zero.
Moreover, by \cite[Lemma 2.6]{MR4300232}, $F$ is real analytic on
$J$. Together, these two facts imply that $F^{-1}\{0\}$ is finite,
which shows that $E_{1}$ is finite.

Let $f:J\rightarrow\mathbb{R}$ be defined by $f(t)=\mathrm{Im}\left(\varphi_{u,t}'\left(z_{u,t}\right)\right)$
for $t\in J$. Since $\varphi_{u,t_{0}}'\left(z_{u,t_{0}}\right)\notin\mathbb{R}$,
the function $f$ is not identically zero. Note that for each $t\in J$,
we have $z_{u,t}=\underset{n\rightarrow\infty}{\lim}\varphi_{u,t}^{n}(z_{0})$,
where $\varphi_{u,t}^{n}$ denotes the composition of $\varphi_{u,t}$
with itself $n$ times. Thus, by \cite[Lemma 2.6]{MR4300232}, the
function $t\mapsto z_{u,t}$ is real-analytic on $J$. Additionally,
by our assumptions on $\left\{ \Phi_{t}\right\} _{t\in J}$, the function
$(t,z)\mapsto\varphi_{u,t}'(z)$ is real analytic on $J\times\Omega$.
Combining these facts, we obtain that $f$ is real analytic on $J$.
Since $f$ is not identically zero and $E_{2}=f^{-1}\{0\}$, it follows
that $E_{2}$ is finite, which completes the proof.
\end{proof}

\subsection{\label{subsec:An-overview-of-the-proof}An overview of the proof}

In this subsection, we provide a brief overview of the proof of Theorem
\ref{thm:main thm}. For the sake of presentation, we allow ourselves
to be rather imprecise and informal. All the arguments will be carried
out rigorously and in detail in later sections of the paper.

We continue to use the notation introduced in Section \ref{subsec:Setup-and-the-main-result},
and suppose that $\Phi$ satisfies the assumptions of Theorem \ref{thm:main thm}.
Denote by $\mathcal{O}(\Omega)$ the vector space of holomorphic functions
from $\Omega$ to $\mathbb{C}$. For $k\ge1$, denote by $\mathcal{P}_{k}$
the vector space of polynomials $q\in\mathbb{C}[X]$ with $\deg q\le k$.
Given a set $Y$, write $\mathcal{M}_{\mathrm{fin}}(Y)$ for the collection
of finitely supported probability measures on $Y$.

As in many other recent developments in the dimension theory of stationary
fractal measures, the basic strategy of our proof is to extend Hochman\textquoteright s
argument from \cite{Ho1}, developed there in the self-similar setting
on $\mathbb{R}$, to the setting studied here. In carrying out this
strategy, we assume for contradiction that
\[
\Delta:=\min\left\{ 2,H(p)/\chi\right\} -\dim\mu>0
\]
and encounter convolutions of the form $\nu.\mu$, where $\nu\in\mathcal{M}_{\mathrm{fin}}\left(\mathcal{O}(\Omega)\right)$
has non-negligible Shannon entropy in a way that depends on $\Delta$\footnote{During the proof, we actually need to consider more general convolutions
of the form $\nu.(\psi\mu)$, where $\psi:\Omega\rightarrow\mathbb{C}$
is a suitably controlled biholomorphism onto its image.}. Here, $\nu.\mu$ denotes the pushforward of $\nu\times\mu$ via
the map $(f,z)\mapsto f(z)$. In order to successfully carry out the
aforementioned extension of the argument from \cite{Ho1}, we need
to show that, in a certain sense to be made precise, the entropy of
$\nu.\mu$ is significantly larger than that of $\mu$ alone.

At this point, we encounter major difficulties caused by the infinite-dimensionality
of $\mathcal{O}(\Omega)$. In order to overcome these difficulties,
we extend to our setting a method developed in \cite{Rap_analytic_on_R}
in the context of real-analytic systems on $\mathbb{R}$. This method
enables us to reduce the above problem to a finite-dimensional situation.
More precisely, it enables us to approximate the entropy of convolutions
of the form $\nu.\mu$, with $\nu\in\mathcal{M}_{\mathrm{fin}}\left(\mathcal{O}(\Omega)\right)$,
by the entropy of convolutions of the form $\nu'.\mu$, with $\nu'\in\mathcal{M}_{\mathrm{fin}}\left(\mathcal{P}_{k}\right)$.
Here, $k\geq1$ is an integer which is assumed to be large with respect
to $\Delta$ and the exponential separation constant $c$ appearing
in Definition \ref{def:exp sep}.

In order to proceed with the overview of the proof, we introduce some
additional notation and terminology. For $n\geq0$, denote by $\mathcal{D}_{n}$
the dyadic partition of $\mathbb{C}$. Given $D\in\mathcal{D}_{n}$
with $\mu(D)>0$, we call a measure of the form $\mu_{D}$ a component
of $\mu$. Very roughly speaking, given a line $\ell\in\mathbb{RP}^{1}$,
the component $\mu_{D}$ is said to be saturated in the direction
$\ell^{\perp}$ if the entropy of $\pi_{\ell}\mu_{D}$, where $\pi_{\ell}$
denotes the orthogonal projection onto $\ell$, is close to its trivial
lower bound, namely, the entropy of $\mu_{D}$ minus $1$.

Given the above reduction, we shall need to establish an entropy increase
result for convolutions $\nu'.\mu$, where $\nu'\in\mathcal{M}_{\mathrm{fin}}\left(\mathcal{P}_{k}\right)$
has non-negligible entropy\footnote{We shall actually prove such a statement for compactly supported probability
measures $\nu'$ on $\mathcal{P}_{k}$, rather than just for finitely
supported probability measures.}. The proof of this result is based on Hochman\textquoteright s inverse
theorem \cite{Ho} for the entropy growth of convolutions of measures
on $\mathbb{R}^{d}$. In order to apply this theorem, we shall need
to show that most components of $\mu$ are not saturated in any direction,
which is the main challenge in our proof of the entropy increase result.

To address this challenge, we shall extend to our setting a method
from the work of Rapaport and Ren \cite{Rap_Ren_Fur_on_CP1}, which
was used there to establish non-saturation of components in the context
of Furstenberg measures on $\mathbb{CP}^{1}$. This method relies
on an ergodic-theoretic argument involving a certain cocycle, which,
following \cite{Rap_Ren_Fur_on_CP1}, we call the direction cocycle.
A key step is to show that the direction cocycle is not a coboundary.
The proof of this fact relies on extending ideas from \cite{Rap_Ren_Fur_on_CP1}
and combining them with complex-analytic techniques.

\medskip{}

\noindent \textbf{Structure of the paper.} In Section \ref{sec:Preliminaries},
we introduce the necessary notation and establish some basic facts
that will be used throughout the paper. In Section \ref{sec:-measure-of-real-analytic-curves},
we show that $\mu$ vanishes on real-analytic curves, which is used
in several places in our arguments. In Section \ref{sec:non-saturation-of-components},
we establish the aforementioned non-saturation property. In Section
\ref{sec:An-entropy-increase-result}, we prove the aforementioned
entropy increase result. In Section \ref{sec:Proof-of-the-main-result},
we prove our main result, Theorem \ref{thm:main thm}. Finally, in
Section \ref{sec:Proof-of-main Corollary for sets}, we prove Corollary
\ref{cor:main cor for sets}, concerning the dimension of self-conformal
sets.

\section{\label{sec:Preliminaries}Preliminaries}

\subsection{\label{subsec:Basic-notations}Basic notation}

Throughout the paper, the base of the logarithm is $2$.

For a metric space $X$, denote by $\mathcal{M}(X)$ the collection
of all compactly supported Borel probability measures on $X$. Given
another metric space $Y$, a Borel map $f:X\rightarrow Y$, and a
measure $\nu\in\mathcal{M}(X)$, we write $f\nu:=\nu\circ f^{-1}$
for the pushforward of $\nu$ via $f$. For a Borel set $E\subset X$
with $\nu(E)>0$, we denote by $\nu_{E}$ the conditioning of $\nu$
on $E$. That is, $\nu_{E}:=\frac{1}{\nu(E)}\nu|_{E}$, where $\nu|_{E}$
is the restriction of $\nu$ to $E$.

Given a partition $\mathcal{D}$ of a set $X$, for $x\in X$ we denote
by $\mathcal{D}(x)$ the unique $D\in\mathcal{D}$ containing $x$.

Given an integer $n\ge1$, let $\mathcal{N}_{n}:=\left\{ 1,...,n\right\} $,
and denote the normalized counting measure on $\mathcal{N}_{n}$ by
$\lambda_{n}$; that is, $\lambda_{n}\{i\}=1/n$ for each $1\le i\le n$.

For a complex vector space $V$, a scalar $c\in\mathbb{C}$, and a
vector $v\in V$, define $S_{c}:V\rightarrow V$ and $T_{v}:V\rightarrow V$
by $S_{c}w=cw$ and $T_{v}w=v+w$ for $w\in V$. We will often use
this notation with $V=\mathbb{C}$.

Given $z\in\mathbb{C}$ and $r>0$, denote by $D_{r}(z)$, $\overline{D}_{r}(z)$
and $C_{r}(z)$ the open disc, closed disc and circle in $\mathbb{C}$
with centre $z$ and radius $r$. Write $\mathbb{D}:=D_{1}(0)$ for
the open unit disc. For a nonempty subset $E\subset\mathbb{C}$, we
denote by $E^{(r)}$ the $r$-thickening of $E$; that is,
\[
E^{(r)}:=\left\{ z\in\mathbb{C}\::\:|z-w|\le r\text{ for some }w\in E\right\} .
\]

We denote by $\mathbb{RP}^{1}$ the set of real lines through the
origin in $\mathbb{C}$; that is, $\mathbb{RP}^{1}:=\left\{ z\mathbb{R}\::\:0\ne z\in\mathbb{C}\right\} $.
For $z\mathbb{R},w\mathbb{R}\in\mathbb{RP}^{1}$, we set $z\mathbb{R}w\mathbb{R}:=zw\mathbb{R}$,
which makes $\mathbb{RP}^{1}$ into a multiplicative group whose identity
element is $\mathbb{R}$. We denote by $\left(z\mathbb{R}\right)^{\perp}\in\mathbb{RP}^{1}$
the line perpendicular to $z\mathbb{R}$. 

For $z,w\in\mathbb{C}$ with $|z|=|w|=1$, write
\[
d\left(z\mathbb{R},w\mathbb{R}\right):=\left(1-\left(\mathrm{Re}\left(z\overline{w}\right)\right)^{2}\right)^{1/2},
\]
which defines a metric on $\mathbb{RP}^{1}$. In fact, $d\left(z\mathbb{R},w\mathbb{R}\right)$
equals the modulus of the sine of the angle between $z\mathbb{R}$
and $w\mathbb{R}$. Given $r>0$, we write $B\left(z\mathbb{R},r\right)$
for the closed ball in $\left(\mathbb{RP}^{1},d\right)$ with center
$z\mathbb{R}$ and radius $r$.

For $z\mathbb{R}\in\mathbb{RP}^{1}$, we denote by $\pi_{z\mathbb{R}}:\mathbb{C}\rightarrow\mathbb{C}$
the orthogonal projection onto $z\mathbb{R}$, where $\mathbb{C}$
is identified with $\mathbb{R}^{2}$; that is,
\[
\pi_{z\mathbb{R}}(w)=|z|^{-2}\mathrm{Re}\left(w\overline{z}\right)z\;\text{ for }w\in\mathbb{C}.
\]

\subsubsection*{Relations between parameters}

Given $R_{1},R_{2}\in\mathbb{R}$ with $R_{1},R_{2}\ge1$, we write
$R_{1}\ll R_{2}$ in order to indicate that $R_{2}$ is large with
respect to $R_{1}$. Formally, this means that $R_{2}\ge f(R_{1})$,
where $f$ is an unspecified function from $[1,\infty)$ into itself.
The values attained by $f$ are assumed to be sufficiently large in
a manner depending on the specific context.

Similarly, given $0<\epsilon_{1},\epsilon_{2}<1$ we write $R_{1}\ll\epsilon_{1}^{-1}$,
$\epsilon_{2}^{-1}\ll R_{2}$ and $\epsilon_{1}^{-1}\ll\epsilon_{2}^{-1}$
in order to respectively indicate that $\epsilon_{1}$ is small with
respect to $R_{1}$, $R_{2}$ is large with respect to $\epsilon_{2}$,
and $\epsilon_{2}$ is small with respect to $\epsilon_{1}$.

The relation $\ll$ is clearly transitive. That is, if $R_{1}\ll R_{2}$
and for $R_{3}\ge1$ we have $R_{2}\ll R_{3}$, then also $R_{1}\ll R_{3}$.
For instance, the sentence `Let $m\ge1$, $k\ge K(m)\ge1$ and $n\ge N(m,k)\ge1$
be given' is equivalent to `Let $m,k,n\ge1$ be with $m\ll k\ll n$'.

\subsection{\label{subsec:The-setup}The setup}

As in Section \ref{subsec:Setup-and-the-main-result}, let $\Omega\subset\mathbb{C}$
be a bounded domain, let $\Lambda$ be a finite index set, and let
$\Phi=\left\{ \varphi_{i}\right\} _{i\in\Lambda}$ be a holomorphic
IFS on $\Omega$. In what follows, we always assume that $\Phi$ satisfies
all the assumptions in Theorem \ref{thm:main thm}, except the exponential
separation condition, which is assumed only in Sections \ref{sec:Proof-of-the-main-result}
and \ref{sec:Proof-of-main Corollary for sets}, where we prove our
main results. Without loss of generality, suppose that $0\in\Omega$.

Throughout the paper, we fix a bounded domain $\Omega_{0}\subset\mathbb{C}$,
containing $\overline{\Omega}$, such that the maps $\varphi_{i}$
are defined on $\overline{\Omega_{0}}$ and conditions \ref{enu:holo IFS cond 1}
and \ref{enu:holo IFS cond 2} hold with $\Omega_{0}$ in place of
$\Omega$.

Let $K_{\Phi}\subset\Omega$ denote the self-conformal set associated
to $\Phi$. That is, $K_{\Phi}$ is the unique nonempty compact subset
of $\Omega$ such that $K_{\Phi}=\cup_{i\in\Lambda}\varphi_{i}(K_{\Phi})$.

Fix a positive probability vector $p=(p_{i})_{i\in\Lambda}$, and
let $\mu$ denote the self-conformal measure associated to $\Phi$
and $p$. That is, $\mu$ is the unique element in $\mathcal{M}(\Omega)$
satisfying the relation
\begin{equation}
\mu=\sum_{i\in\Lambda}p_{i}\cdot\varphi_{i}\mu.\label{eq:def rel of mu}
\end{equation}
Note that the support of $\mu$ is $K_{\Phi}$.

As above, let $\chi:=\chi\left(\Phi,p\right)$ denote the Lyapunov
exponent associated to $\Phi$ and $p$. That is,
\begin{equation}
\chi:=-\sum_{i\in\Lambda}p_{i}\int\log\left|\varphi_{i}'(z)\right|\:d\mu(z).\label{eq:def of Lyap expo}
\end{equation}

\subsection{\label{subsec:A-bi-Lipchitz-estimate}A bi-Lipschitz estimate}

Given $0<\epsilon\le1$, let $\mathcal{F}_{\epsilon}$ denote the
set of holomorphic injective maps $\psi:\Omega_{0}\rightarrow\mathbb{C}$
satisfying $\epsilon\le\left|\psi'(z)\right|\le\epsilon^{-1}$ for
all $z\in\Omega_{0}$.
\begin{lem}
\label{lem:lip prop for psi in F_epsilon}There exists a constant
$C>1$ such that for all $0<\epsilon\le1$ and $\psi\in\mathcal{F}_{\epsilon}$,
\[
C^{-1}\epsilon|z-w|\le\left|\psi(z)-\psi(w)\right|\le C\epsilon^{-1}|z-w|\;\text{ for all }z,w\in\Omega.
\]
\end{lem}

\begin{proof}
Let $C>1$ be large with respect to $\Omega$ and $\Omega_{0}$, and
let $0<\epsilon\le1$ and $\psi\in\mathcal{F}_{\epsilon}$ be given.

We first prove the upper bound. For $z,w\in\Omega$, write $d\left(z,w\right)$
for the infimum of the lengths of smooth paths $\gamma:[0,1]\rightarrow\Omega_{0}$
with $\gamma(0)=z$ and $\gamma(1)=w$. As shown in the proof of \cite[Lemma 2.1]{MR1479016},
assuming $C$ is sufficiently large we have
\[
d\left(z,w\right)\le C|z-w|\text{ for all }z,w\in\Omega.
\]
Given $z,w\in\Omega$ and $\delta>0$, there exists a smooth path
$\gamma:[0,1]\rightarrow\Omega_{0}$ such that $\gamma(0)=z$, $\gamma(1)=w$,
and $\mathrm{length}(\gamma)\le d\left(z,w\right)+\delta$. Note that
$\psi\circ\gamma$ is a smooth path from $\psi(z)$ to $\psi(w)$.
Hence,
\begin{multline*}
\left|\psi(z)-\psi(w)\right|\le\mathrm{length}\left(\psi\circ\gamma\right)=\int_{0}^{1}\left|\psi'\left(\gamma(t)\right)\right|\cdot\left|\gamma'(t)\right|\:dt\le\epsilon^{-1}\int_{0}^{1}\left|\gamma'(t)\right|\:dt\\
=\epsilon^{-1}\mathrm{length}(\gamma)\le\epsilon^{-1}\left(d\left(z,w\right)+\delta\right)\le\epsilon^{-1}\left(C|z-w|+\delta\right).
\end{multline*}
Since this holds for all $\delta>0$, the upper bound follows.

We turn to the proof of the lower bound. We may assume that $D_{1/C}(z)\subset\Omega_{0}$
for all $z\in\Omega$. Let $z,w\in\Omega$, and set $r:=\min\left\{ 1/C,|z-w|\right\} $.
By the Koebe quarter theorem (see, e.g., \cite[Theorem 14.14]{MR924157}),
\[
D\left(\psi(z),\frac{1}{4}\left|\psi'(z)\right|r\right)\subset\psi\left(D\left(z,r\right)\right).
\]
Thus, since $w\notin D\left(z,r\right)$ and $\psi$ is injective
on $\Omega_{0}$,
\[
\left|\psi(z)-\psi(w)\right|\ge\frac{1}{4}\left|\psi'(z)\right|r\ge\frac{\epsilon r}{4}.
\]
Since $\Omega$ is bounded, we may assume that $r\ge|z-w|/C^{2}$.
Hence,
\[
\left|\psi(z)-\psi(w)\right|\ge\frac{\epsilon}{4C^{2}}|z-w|,
\]
which completes the proof of the lemma.
\end{proof}

\subsection{Bounded distortion and related auxiliary results}

Recall that $\Omega_{0}$ is a bounded domain containing $\overline{\Omega}$,
the maps $\varphi_{i}$ are defined on $\overline{\Omega_{0}}$, and
conditions \ref{enu:holo IFS cond 1} and \ref{enu:holo IFS cond 2}
hold with $\Omega_{0}$ in place of $\Omega$. This implies that $\Phi$
satisfies the so-called bounded distortion property on $\Omega_{0}$,
as stated in the following lemma. Its proof can be found, for example,
in \cite[Lemma 2.1]{MR1479016}. Recall that $\Lambda^{*}$ denotes
the set of finite words over $\Lambda$.
\begin{lem}[Bounded distortion property]
\label{lem:bounded distortion}There exists a constant $C\ge1$ such
that $\left|\varphi_{u}'(z)\right|\le C\left|\varphi_{u}'(w)\right|$
for all $u\in\Lambda^{*}$ and $z,w\in\Omega_{0}$.
\end{lem}

Combining the bounded distortion property with Lemma \ref{lem:lip prop for psi in F_epsilon}
we obtain the following.
\begin{lem}
\label{lem:Lip prop of varphi_u}There exists a constant $C>1$ such
that for all $u\in\Lambda^{*}$, $0<\epsilon\le1$, $\psi\in\mathcal{F}_{\epsilon}$,
and $z,w\in\Omega$
\[
C^{-1}\epsilon\left|\varphi_{u}'(0)\right||z-w|\le\left|\psi\circ\varphi_{u}(z)-\psi\circ\varphi_{u}(w)\right|\le C\epsilon^{-1}\left|\varphi_{u}'(0)\right||z-w|.
\]
\end{lem}

\begin{proof}
Let $C>1$ be a constant at least as large as the constants obtained
in Lemmas \ref{lem:lip prop for psi in F_epsilon} and \ref{lem:bounded distortion}.
Let $u\in\Lambda^{*}$, $0<\epsilon\le1$ and $\psi\in\mathcal{F}_{\epsilon}$
be given. By the chain rule and Lemma \ref{lem:bounded distortion},
$S_{\varphi_{u}'(0)}^{-1}\circ\psi\circ\varphi_{u}\in\mathcal{F}_{\epsilon/C}$.
Hence, by Lemma \ref{lem:lip prop for psi in F_epsilon}, for each
$z,w\in\Omega$
\[
C^{-2}\epsilon|z-w|\le\left|S_{\varphi_{u}'(0)}^{-1}\circ\psi\circ\varphi_{u}(z)-S_{\varphi_{u}'(0)}^{-1}\circ\psi\circ\varphi_{u}(w)\right|\le C^{2}\epsilon^{-1}|z-w|,
\]
which proves the lemma.
\end{proof}
Combining Cauchy's estimates with the above considerations, we can
also obtain the following estimate for the higher derivatives.
\begin{lem}
\label{lem:ub on high derivs}There exists a constant $C>1$ such
that for all $u\in\Lambda^{*}$, $0<\epsilon\le1$, $\psi\in\mathcal{F}_{\epsilon}$,
$k\in\mathbb{Z}_{>0}$, and $z\in\Omega$,
\[
\left|\left(\psi\circ\varphi_{u}\right)^{(k)}(z)\right|\le k!C^{k}\epsilon^{-1}\left|\varphi_{u}'(0)\right|,
\]
where $\left(\psi\circ\varphi_{u}\right)^{(k)}$ denotes the $k$-th
derivative of $\psi\circ\varphi_{u}$.
\end{lem}

\begin{proof}
Let $0<\delta<1$ be sufficiently small so that the closure of $\Omega^{(\delta)}$
is contained in $\Omega_{0}$. By the arguments given in the proofs
of Lemmas \ref{lem:lip prop for psi in F_epsilon} and \ref{lem:Lip prop of varphi_u},
there exists $C=C(\delta)>1$ such that for all $u\in\Lambda^{*}$,
$0<\epsilon\le1$ and $\psi\in\mathcal{F}_{\epsilon}$,
\begin{equation}
\left|\psi\circ\varphi_{u}(z)-\psi\circ\varphi_{u}(w)\right|\le C\epsilon^{-1}\left|\varphi_{u}'(0)\right||z-w|\;\text{ for all }z,w\in\Omega^{(\delta)}.\label{eq:lip prop on Omega^(eps)}
\end{equation}

Let $u\in\Lambda^{*}$, $0<\epsilon\le1$ and $\psi\in\mathcal{F}_{\epsilon}$
be given, and set $f:=\psi\circ\varphi_{u}-\psi\circ\varphi_{u}(0)$.
From (\ref{eq:lip prop on Omega^(eps)}) and by assuming $C$ is sufficiently
large, for $w\in\Omega^{(\delta)}$
\[
\left|f(w)\right|\le C\epsilon^{-1}\left|\varphi_{u}'(0)\right||w|\le C^{2}\epsilon^{-1}\left|\varphi_{u}'(0)\right|.
\]
Thus, by Cauchy's estimate, for $k\in\mathbb{Z}_{>0}$ and $z\in\Omega$
\[
\left|\left(\psi\circ\varphi_{u}\right)^{(k)}(z)\right|=\left|f^{(k)}(z)\right|\le k!\delta^{-k}\cdot C^{2}\epsilon^{-1}\left|\varphi_{u}'(0)\right|,
\]
which completes the proof of the lemma.
\end{proof}

\subsection{\label{subsec:Function-spaces}Function spaces and Taylor approximation}

Denote by $\mathcal{O}(\Omega)$ the vector space of holomorphic functions
from $\Omega$ to $\mathbb{C}$.

Given $k\ge1$, denote by $\mathcal{P}_{k}$ the vector space of polynomials
$q\in\mathbb{C}[X]$ with $\deg q\le k$. Let $\Vert\cdot\Vert_{2}$
be the norm on $\mathcal{P}_{k}$ defined by
\[
\Vert q\Vert_{2}^{2}:=\sum_{j=0}^{k}|a_{j}|^{2}\;\text{ for }q(X)=\sum_{j=0}^{k}a_{j}X^{j}\in\mathcal{P}_{k}.
\]
In what follows, all metric and topological concepts in $\mathcal{P}_{k}$
are considered with respect to $\Vert\cdot\Vert_{2}$.

Let $\nu$ be a finitely supported probability measure on $\mathcal{O}(\Omega)$
or an element of $\mathcal{M}\left(\mathcal{P}_{k}\right)$ for some
$k\ge1$. Given $\theta\in\mathcal{M}(\Omega)$, we write $\nu.\theta$
for the convolution of $\nu$ with $\theta$. That is, $\nu.\theta$
denotes the pushforward of $\nu\times\theta$ via the map sending
$(f,z)\in\mathcal{O}(\Omega)\times\Omega$ to $f(z)$. For $z\in\Omega$
we write $\nu.z$ in place of $\nu.\delta_{z}$, where $\delta_{z}$
is the Dirac mass at $z$.

For $k\ge1$ and $a\in\Omega$, let $P_{k,a}:\mathcal{O}(\Omega)\rightarrow\mathcal{P}_{k}$
be the linear map such that $P_{k,a}f$ equals the $k$-th order Taylor
polynomial of $f\in\mathcal{O}(\Omega)$ at the point $a$. That is,
\begin{equation}
P_{k,a}f(X)=\sum_{j=0}^{k}\frac{f^{(j)}(a)}{j!}(X-a)^{j}.\label{eq:exp form of P_k,a}
\end{equation}
The following Taylor estimate will be crucial for our arguments.
\begin{lem}
\label{lem:Taylor estimate}There exist $\delta>0$ and $C>1$ such
that for all $u\in\Lambda^{*}$, $0<\epsilon\le1$, $\psi\in\mathcal{F}_{\epsilon}$,
$k\ge1$, and $a,z\in\Omega$ with $|z-a|<\delta$,
\[
\left|\psi\circ\varphi_{u}(z)-P_{k,a}\left(\psi\circ\varphi_{u}\right)(z)\right|\le C^{k+1}\epsilon^{-1}\left|\varphi_{u}'(0)\right||z-a|^{k+1}.
\]
\end{lem}

\begin{proof}
Let $\delta>0$ be sufficiently small so that $\Omega^{(\delta)}\subset\Omega_{0}$.
Let $C>1$ be the constant obtained in Lemma \ref{lem:ub on high derivs},
and let $u\in\Lambda^{*}$, $0<\epsilon\le1$, $\psi\in\mathcal{F}_{\epsilon}$,
$k\ge1$, and $a,z\in\Omega$ with $0<|z-a|<\min\left\{ \delta,1/(2C)\right\} $.
Since $\psi\circ\varphi_{u}$ is analytic on $\Omega_{0}$ and $z\in D_{\delta}(a)\subset\Omega_{0}$,
\[
\psi\circ\varphi_{u}(z)=\sum_{j=0}^{\infty}\frac{\left(\psi\circ\varphi_{u}\right)^{(j)}(a)}{j!}(z-a)^{j}.
\]
Thus, by Lemma \ref{lem:ub on high derivs} and since $C|z-a|<1/2$,
\begin{eqnarray*}
\left|\psi\circ\varphi_{u}(z)-P_{k,a}\left(\psi\circ\varphi_{u}\right)(z)\right| & \le & \sum_{j=k+1}^{\infty}\frac{\left|\left(\psi\circ\varphi_{u}\right)^{(j)}(a)\right|}{j!}|z-a|^{j}\\
 & \le & 2C^{k+1}\epsilon^{-1}\left|\varphi_{u}'(0)\right||z-a|^{k+1},
\end{eqnarray*}
which completes the proof of the lemma.
\end{proof}

\subsection{Entropy}

Let $(X,\mathcal{B})$ be a measurable space. Given a probability
measure $\theta$ on $X$ and a countable partition $\mathcal{D}\subset\mathcal{B}$
of $X$, the entropy of $\theta$ with respect to $\mathcal{D}$ is
defined by
\[
H(\theta,\mathcal{D}):=-\sum_{D\in\mathcal{D}}\theta(D)\log\theta(D).
\]
If $\mathcal{E}\subset\mathcal{B}$ is another countable partition
of $X$, the conditional entropy given $\mathcal{E}$ is defined as
follows
\[
H(\theta,\mathcal{D}\mid\mathcal{E}):=\sum_{E\in\mathcal{E}}\theta(E)\cdot H(\theta_{E},\mathcal{D}).
\]

Throughout the paper, we repeatedly use basic properties of entropy
and conditional entropy. We frequently do so without explicit reference.
Readers are advised to consult \cite[Section 3.1]{Ho1} for details.

\subsection{\label{subsec:Dyadic-partitions}Dyadic partitions and component
measures}

For $d\ge1$ and $n\ge0$, denote by $\mathcal{D}_{n}^{\mathbb{C}^{d}}$
the level-$n$ dyadic partition of $\mathbb{C}^{d}$, where $\mathbb{C}^{d}$
is identified with $\mathbb{R}^{2d}$. For a real number $t\ge0$,
we write $\mathcal{D}_{t}^{\mathbb{C}^{d}}$ in place of $\mathcal{D}_{\left\lfloor t\right\rfloor }^{\mathbb{C}^{d}}$,
where $\left\lfloor t\right\rfloor $ denotes the integer part of
$t$. We omit the superscript $\mathbb{C}^{d}$ when it is clear from
the context.

Dyadic entropy, i.e. entropy measured with respect to the partitions
$\mathcal{D}_{n}^{\mathbb{C}}$, plays an important role in our arguments.
Note that by \cite{Yo}, when $\theta\in\mathcal{M}\left(\mathbb{C}\right)$
is exact dimensional,
\begin{equation}
\underset{n\rightarrow\infty}{\lim}\frac{1}{n}H\left(\theta,\mathcal{D}_{n}\right)=\dim\theta.\label{eq:ent dim =00003D exact dim}
\end{equation}

We shall also need the following two simple lemmas. They follow easily
from basic properties of entropy, and their proofs are therefore omitted.
\begin{lem}
\label{lem:dyad ent =000026 lip func}Let $d\ge1$ and $\theta\in\mathcal{M}\left(\mathbb{C}^{d}\right)$
be given. Let $C>1$, and let $f:\mathrm{supp}(\theta)\rightarrow\mathbb{C}^{d}$
be bi-Lipschitz with bi-Lipschitz constant $C$. That is, $C^{-1}|x-y|\le\left|f(x)-f(y)\right|\le C|x-y|$
for $x,y\in\mathrm{supp}(\theta)$. Then for each $n\ge0$,
\[
H\left(f\theta,\mathcal{D}_{n}\right)=H\left(\theta,\mathcal{D}_{n}\right)+O_{d}\left(1+\log C\right).
\]
\end{lem}

\begin{lem}
\label{lem:close func --> close ent}Let $(X,\mathcal{B},\theta)$
be a probability space, and let $f,g:X\rightarrow\mathbb{C}$ be measurable.
Let $n\ge0$, and suppose that $\left|f(x)-g(x)\right|\le2^{-n}$
for all $x\in X$. Then
\[
H\left(f\theta,\mathcal{D}_{n}\right)=H\left(g\theta,\mathcal{D}_{n}\right)+O(1).
\]
\end{lem}

We also need to consider dyadic partitions of vector spaces of polynomials.
For $k\ge1$ and $n\ge0$, we denote by $\mathcal{D}_{n}^{\mathcal{P}_{k}}$
the level-$n$ dyadic partition of $\mathcal{P}_{k}$, which is given
by 
\[
\mathcal{D}_{n}^{\mathcal{P}_{k}}:=\left\{ \left\{ \sum_{j=0}^{k}a_{j}X^{j}\::\:(a_{0},...,a_{k})\in D\right\} \::\:D\in\mathcal{D}_{n}^{\mathbb{C}^{k+1}}\right\} .
\]
Thus, $\mathcal{D}_{n}^{\mathcal{P}_{k}}$ is defined by identifying
$\mathcal{P}_{k}$ with $\mathbb{C}^{k+1}$ in a natural way. We omit
the superscript $\mathcal{P}_{k}$ when it is clear from the context.

Given $\theta\in\mathcal{M}\left(\mathbb{C}\right)$, $n\ge0$, and
$z\in\mathbb{C}$ with $\theta\left(\mathcal{D}_{n}(z)\right)>0$,
we write $\theta_{z,n}$ in place of the conditional measure $\theta_{\mathcal{D}_{n}(z)}$.
Similarly, for $\nu\in\mathcal{M}\left(\mathcal{P}_{k}\right)$ and
$q\in\mathcal{P}_{k}$ with $\nu\left(\mathcal{D}_{n}(q)\right)>0$,
we write $\nu_{q,n}$ in place of $\nu_{\mathcal{D}_{n}(q)}$. The
measures $\theta_{z,n}$ and $\nu_{q,n}$ are said to be level-$n$
components of $\theta$ and $\nu$ respectively.

Throughout the rest of the paper, we use the probabilistic notations
introduced in \cite[Section 2.2]{Ho1}; readers are encouraged to
consult this reference for further details. In particular, we often
consider $\theta_{z,n}$ and $\nu_{q,n}$ as random measures in a
natural way.

\subsection{\label{subsec:Symbolic-notations}Symbolic preliminaries}

Let $\Lambda^{\mathbb{N}}$ denote the set of one-sided infinite words
over $\Lambda$. We equip $\Lambda^{\mathbb{N}}$ with the product
topology, where each copy of $\Lambda$ is equipped with the discrete
topology. Let $\sigma:\Lambda^{\mathbb{N}}\rightarrow\Lambda^{\mathbb{N}}$
denote the left-shift map. That is, $\sigma(\omega)=(\omega_{n+1})_{n\ge0}$
for $(\omega_{n})_{n\ge0}=\omega\in\Lambda^{\mathbb{N}}$.

Given $n\ge1$ and $(\omega_{k})_{k\ge0}=\omega\in\Lambda^{\mathbb{N}}$,
write $\omega|_{n}:=\omega_{0}...\omega_{n-1}\in\Lambda^{n}$ for
the $n$th prefix of $\omega$. For a word $u\in\Lambda^{n}$, denote
by $[u]\subset\Lambda^{\mathbb{N}}$ the cylinder set determined by
$u$. That is, $[u]:=\left\{ \omega\in\Lambda^{\mathbb{N}}\::\:\omega|_{n}=u\right\} $.
Let $\mathcal{C}_{n}:=\left\{ [u]\right\} _{u\in\Lambda^{n}}$ denote
the partition of $\Lambda^{\mathbb{N}}$ into level-$n$ cylinder
sets. Given a set of words $\mathcal{U}\subset\Lambda^{*}$, we often
write $\left[\mathcal{U}\right]$ in place of $\cup_{u\in\mathcal{U}}[u]$.

Denote by $\beta:=p^{\mathbb{N}}$ the Bernoulli measure on $\Lambda^{\mathbb{N}}$
associated to $p$. That is, $\beta$ is the unique element in $\mathcal{M}\left(\Lambda^{\mathbb{N}}\right)$
such that $\beta\left([u]\right)=p_{u}$ for $u\in\Lambda^{*}$, where
$p_{u}:=p_{i_{1}}\cdot...\cdot p_{i_{n}}$ for $u=i_{1}...i_{n}\in\Lambda^{n}$.

Let $\Pi:\Lambda^{\mathbb{N}}\rightarrow\Omega$ be the coding map
associated to $\Phi$. That is,
\[
\Pi(\omega):=\underset{n\rightarrow\infty}{\lim}\varphi_{\omega|_{n}}(0)\text{ for }\omega\in\Lambda^{\mathbb{N}},
\]
where recall that we assumed that $0\in\Omega$. Note that $\mu=\Pi\beta$.
For $n\ge1$, let $\Pi_{n}:\Lambda^{\mathbb{N}}\rightarrow\mathcal{O}(\Omega)$
be defined by $\Pi_{n}(\omega)=\varphi_{\omega|_{n}}$ for $\omega\in\Lambda^{\mathbb{N}}$.

From the ergodicity of $\left(\Lambda^{\mathbb{N}},\sigma,\beta\right)$,
by bounded distortion, and since $\mu=\Pi\beta$, it follows easily
that 
\begin{equation}
\chi=-\underset{n\rightarrow\infty}{\lim}\frac{1}{n}\log\left|\varphi_{\omega|_{n}}'(0)\right|\text{ for }\beta\text{-a.e. }\omega.\label{eq:chi as a.s.  limit}
\end{equation}

Write $\{\beta_{\omega}\}_{\omega\in\Lambda^{\mathbb{N}}}\subset\mathcal{M}\left(\Lambda^{\mathbb{N}}\right)$
for the disintegration of $\beta$ with respect to $\Pi^{-1}\mathcal{B}_{\mathbb{C}}$,
where $\mathcal{B}_{\mathbb{C}}$ is the Borel $\sigma$-algebra of
$\mathbb{C}$. For the definition and basic properties of disintegrations,
we refer the reader to \cite[Theorem 5.14]{EiWa} and the discussion
following it.

For $u,v\in\Lambda^{*}$, we denote by $uv$ the concatenation of
$u$ with $v$.

Given integers $l,n\ge1$ and $0\le j<l$, let $\Psi\left(j,l;n\right)$
denote the set of words $u_{0}...u_{s}\in\Lambda^{*}$ such that $u_{0}\in\Lambda^{j}$,
$u_{i}\in\Lambda^{l}$ for $1\le i\le s$, $\left|\varphi_{u_{0}...u_{s}}'(0)\right|<2^{-n}$,
and $\left|\varphi_{u_{0}...u_{i}}'(0)\right|\ge2^{-n}$ for $0\le i<s$.
By bounded distortion, there exists a constant $C(l)>1$ such that
\begin{equation}
2^{-n}/C(l)\le\left|\varphi_{u}'(0)\right|<2^{-n}\text{ for all }u\in\Psi\left(j,l;n\right).\label{eq:der at 0 is comp to 2^-n for u in Psi_n}
\end{equation}
Additionally, note that $\Psi\left(j,l;n\right)$ is a minimal cut-set
for $\Lambda^{*}$, which means that for each $\omega\in\Lambda^{\mathbb{N}}$
there exists a unique $u\in\Psi\left(j,l;n\right)$ with $\omega\in[u]$.
From this, and the relation $\mu=\sum_{i\in\Lambda}p_{i}\cdot\varphi_{i}\mu$,
it follows that
\begin{equation}
\mu=\sum_{u\in\Psi\left(j,l;n\right)}p_{u}\cdot\varphi_{u}\mu.\label{eq:mu as conv comb with u in Psi_n}
\end{equation}
We shall write $\Psi_{n}$ in place of $\Psi\left(0,1;n\right)$.

It will sometimes be useful to choose words from $\Lambda^{n}$ and
$\Psi\left(j,l;n\right)$ at random. Let $\mathbf{U}_{n}$ and $\mathbf{I}(j,l;n)$
denote the random words with
\[
\mathbb{P}\left\{ \mathbf{U}_{n}=u\right\} =\begin{cases}
p_{u} & \text{ if }u\in\Lambda^{n}\\
0 & \text{ otherwise}
\end{cases}\:\text{ and }\:\mathbb{P}\left\{ \mathbf{I}(j,l;n)=u\right\} =\begin{cases}
p_{u} & \text{ if }u\in\Psi\left(j,l;n\right)\\
0 & \text{ otherwise}
\end{cases}.
\]

The following lemma will be needed in Section \ref{sec:non-saturation-of-components}
when we establish the non-saturation of components discussed in Section
\ref{subsec:An-overview-of-the-proof}. The lemma explains why $\Psi_{n}$
and $\mathbf{I}(0,1;n)$ are not sufficient, and why the more general
$\Psi\left(j,l;n\right)$ and $\mathbf{I}(j,l;n)$ are required.
\begin{lem}
\label{lem:ac of words}For every sufficiently large $l\ge1$ there
exists $M=M(l)\in\mathbb{Z}_{>0}$ such that for all $0\le j<l$ and
$n\ge1$ we have
\begin{equation}
\mathbb{E}_{1\le k\le n}\left(\delta_{\mathbf{U}_{j+lk}}\right)\ll\mathbb{E}_{1\le k\le nM}\left(\delta_{\mathbf{I}(j,l;k)}\right),\label{eq:ac of words}
\end{equation}
with Radon--Nikodym derivative bounded by $M$.
\end{lem}

\begin{rem*}
Note that $\mathbb{E}_{1\le k\le n}\left(\delta_{\mathbf{U}_{j+lk}}\right)$
denotes the probability measure $\frac{1}{n}\sum_{k=1}^{n}\mathbb{E}\left(\delta_{\mathbf{U}_{j+lk}}\right)$.
The measure $\mathbb{E}_{1\le k\le nM}\left(\delta_{\mathbf{I}(j,l;k)}\right)$
is defined analogously.
\end{rem*}
\begin{proof}
Let $C\ge1$ be the constant obtained in Lemma \ref{lem:bounded distortion},
set
\[
r_{\max}:=\sup\left\{ \left|\varphi_{i}'(z)\right|\::\:i\in\Lambda\text{ and }z\in\Omega\right\} \text{ and }r_{\min}:=\inf\left\{ \left|\varphi_{i}'(z)\right|\::\:i\in\Lambda\text{ and }z\in\Omega\right\} ,
\]
let $l\ge1$ be such that $Cr_{\max}^{l}<1/2$, and let $M\ge1$ be
with $r_{\min}^{2l}\ge2^{-M}$. For $u\in\Lambda^{*}$ and $v\in\Lambda^{l}$,
\begin{equation}
\left|\varphi_{uv}'(0)\right|=\left|\varphi_{u}'\left(\varphi_{v}(0)\right)\right|\left|\varphi_{v}'(0)\right|\le Cr_{\max}^{l}\left|\varphi_{u}'(0)\right|<\left|\varphi_{u}'(0)\right|/2.\label{eq:varphi_uv le varphi_u/2}
\end{equation}

Let $0\le j<l$, $n\ge1$, $1\le k\le n$, $u_{0}\in\Lambda^{j}$,
and $u_{1},...,u_{k}\in\Lambda^{l}$ be given. Let $m\ge1$ be the
unique integer with 
\[
2^{-m-1}\le\left|\varphi_{u_{0}...u_{k}}'(0)\right|<2^{-m}.
\]
From
\[
\left|\varphi_{u_{0}...u_{k}}'(0)\right|\ge r_{\min}^{l(n+1)}\ge2^{-Mn}
\]
it follows that $m\le Mn$. Moreover, by \ref{eq:varphi_uv le varphi_u/2},
\[
\left|\varphi_{u_{0}...u_{k}}'(0)\right|\le\left|\varphi_{u_{0}...u_{i}}'(0)\right|/2\text{ for all }0\le i<k,
\]
which implies that $u_{0}...u_{k}\in\Psi\left(j,l;m\right)$.

We have thus shown that
\[
\cup_{k=1}^{n}\Lambda^{j+lk}\subset\cup_{k=1}^{nM}\Psi\left(j,l;k\right).
\]
This clearly gives (\ref{eq:ac of words}) with Radon--Nikodym derivative
bounded by $M$, which completes the proof of the lemma.
\end{proof}

\section{\label{sec:-measure-of-real-analytic-curves}The $\mu$-measure of
real-analytic curves}

\subsection{Vanishing of $\mu$ on real-analytic curves}

The following proposition is the main result of this subsection.
\begin{prop}
\label{prop:mu(Gamma)=00003D0}For each regular real-analytic curve
$\Gamma\subset\Omega$ we have $\mu(\Gamma)=0$.
\end{prop}

The proof of the proposition requires some preparation.
\begin{lem}
\label{lem:exists word u s.t. varph_u mu(E) is large}Let $E\subset\Omega$
be a Borel set with $\mu(E)>0$. Then for each $\epsilon>0$ there
exists $u\in\Lambda^{*}$ such that $\varphi_{u}\mu(E)>1-\epsilon$
and $\mathrm{diam}\left(\varphi_{u}(\Omega)\right)\le\epsilon$.
\end{lem}

\begin{proof}
Since $\mu(E)>0$ and $K_{\Phi}=\mathrm{supp}(\mu)$, the density
theorem for Radon measures (see \cite[Corollary 2.14]{Ma}) implies
that there exists $z_{0}\in K_{\Phi}$ such that
\begin{equation}
\underset{r\downarrow0}{\lim}\frac{\mu\left(E\cap\overline{D}_{r}(z_{0})\right)}{\mu\left(\overline{D}_{r}(z_{0})\right)}=1.\label{eq:by density thm}
\end{equation}

Since $\mu$ is a probability measure, the set $\left\{ r>0\::\:\mu\left(C_{r}(z_{0})\right)>0\right\} $
is at most countable. Let $C>1$ be a large global constant, and let
$m\in\mathbb{Z}_{>0}$ and $\epsilon,r,\delta\in(0,1)$ be such that
$C\ll m\ll\epsilon^{-1}\ll r^{-1}\ll\delta^{-1}$ and $\mu\left(C_{r}(z_{0})\right)=0$.
Since $\epsilon^{-1}\ll r^{-1}$ and by (\ref{eq:by density thm}),
we may assume that
\[
\mu\left(E\cap\overline{D}_{r}(z_{0})\right)>(1-\epsilon)\mu\left(\overline{D}_{r}(z_{0})\right).
\]
Write $A$ for the closed annulus $\overline{D}_{r+\delta}(z_{0})\setminus D_{r-\delta}(z_{0})$.
Since $\mu\left(C_{r}(z_{0})\right)=0$ and $\epsilon^{-1},r^{-1}\ll\delta^{-1}$,
we may assume that $\mu\left(A\right)<\epsilon\mu\left(\overline{D}_{r}(z_{0})\right)$.

Setting $n:=\left\lceil -\log\delta\right\rceil $, let $\mathcal{U}_{1}$
be the set of all $u\in\Psi_{n+m}$ such that $\varphi_{u}\mu\left(\overline{D}_{r}(z_{0})\right)>0$
and let $\mathcal{U}_{2}$ be the set of all $u\in\Psi_{n+m}$ such
that $\varphi_{u}\left(K_{\Phi}\right)\subset\overline{D}_{r}(z_{0})$.
By Lemma \ref{lem:Lip prop of varphi_u}, from (\ref{eq:der at 0 is comp to 2^-n for u in Psi_n}),
and since $C\ll m$, for all $u\in\Psi_{n+m}$
\[
\mathrm{diam}\left(\varphi_{u}\left(\Omega\right)\right)\le C2^{-n-m}\le\delta.
\]
From this, since $K_{\Phi}\subset\Omega$, and by the definitions
and $\mathcal{U}_{1}$ and $\mathcal{U}_{2}$, it follows that $\varphi_{u}\left(K_{\Phi}\right)\subset A$
for each $u\in\mathcal{U}_{1}\setminus\mathcal{U}_{2}$. Hence, from
(\ref{eq:mu as conv comb with u in Psi_n}),
\[
\sum_{u\in\mathcal{U}_{1}\setminus\mathcal{U}_{2}}p_{u}\le\sum_{u\in\Psi_{n+m}}p_{u}\cdot\varphi_{u}\mu\left(A\right)=\mu(A)\le\epsilon\mu\left(\overline{D}_{r}(z_{0})\right),
\]
which gives
\begin{multline*}
(1-\epsilon)\mu\left(\overline{D}_{r}(z_{0})\right)<\mu\left(E\cap\overline{D}_{r}(z_{0})\right)\\
\le\sum_{u\in\mathcal{U}_{1}}p_{u}\cdot\varphi_{u}\mu\left(E\right)\le\epsilon\mu\left(\overline{D}_{r}(z_{0})\right)+\sum_{u\in\mathcal{U}_{2}}p_{u}\cdot\varphi_{u}\mu\left(E\right).
\end{multline*}
Thus,
\[
(1-2\epsilon)\sum_{u\in\mathcal{U}_{2}}p_{u}\le(1-2\epsilon)\mu\left(\overline{D}_{r}(z_{0})\right)<\sum_{u\in\mathcal{U}_{2}}p_{u}\cdot\varphi_{u}\mu\left(E\right),
\]
from which it follows that there exists $u\in\mathcal{U}_{2}$ such
that $\varphi_{u}\mu\left(E\right)>1-2\epsilon$. Since $\mathrm{diam}\left(\varphi_{u}\left(\Omega\right)\right)\le\delta<\epsilon$,
this completes the proof of the lemma.
\end{proof}
We omit the proof of the following lemma, which follows easily from
analyticity and the fact that harmonic functions are real-analytic.
\begin{lem}
\label{lem:harmonic =00003D0 on connected curve}Let $\Gamma\subset\Omega$
be a connected regular real-analytic curve, let $h:\Omega\rightarrow\mathbb{R}$
be harmonic, and suppose that the set $h^{-1}\{0\}\cap\Gamma$ has
a limit point in $\Gamma$. Then $h(z)=0$ for all $z\in\Gamma$.
\end{lem}

\begin{proof}[Proof of Proposition \ref{prop:mu(Gamma)=00003D0}]
Assume by contradiction that $\mu$ gives positive mass to some regular
real-analytic curve. From this assumption and by Definition \ref{def:reg real-anal curve},
it follows easily that there exist a compact interval $J\subset\mathbb{R}$,
a domain $V\subset\mathbb{C}$ with $J\subset V$, and a holomorphic
injection $f:V\rightarrow\Omega$, such that $\mu\left(f(J)\right)>0$.
Writing $h:=\mathrm{Im}\circ f^{-1}$, it holds that $h:f(V)\rightarrow\mathbb{R}$
is harmonic and regular (i.e. $\nabla h(z)\ne0$ for all $z\in f(V)$,
where $\nabla h$ denotes the gradient of $h$). Moreover, $f(J)\subset h^{-1}\{0\}$.

Let $W\subset\mathbb{C}$ be a bounded domain with $f(J)\subset W\subset\overline{W}\subset f(V)$.
By Lemma \ref{lem:exists word u s.t. varph_u mu(E) is large}, there
exists a sequence $\left\{ u_{n}\right\} _{n\ge1}\subset\Lambda^{*}$
such that $\varphi_{u_{n}}\mu\left(f(J)\right)>1-\frac{1}{n}$ and
$\varphi_{u_{n}}(\Omega)\subset W$ for each $n\ge1$. For $n\ge1$
set $r_{n}:=\left|\varphi_{u_{n}}'(0)\right|$ and $h_{n}:=S_{r_{n}}^{-1}\circ h\circ\varphi_{u_{n}}$,
where recall that $S_{r_{n}}z:=r_{n}z$ for $z\in\mathbb{C}$. Since
$h$ is harmonic and $\varphi_{u_{n}}$ is holomorphic, it follows
that $h_{n}:\Omega\rightarrow\mathbb{R}$ is harmonic.

It is clear that $h$ is Lipschitz on $\overline{W}$. From this and
by Lemma \ref{lem:Lip prop of varphi_u}, it follows that there exists
$M\ge1$ such that $h_{n}$ is $M$-Lipschitz on $\Omega$ for each
$n\ge1$. Given $n\ge1$,
\[
\varphi_{u_{n}}\mu\left(h^{-1}\{0\}\right)\ge\varphi_{u_{n}}\mu\left(f(J)\right)>0,
\]
and so there exists $z_{n}\in\Omega$ such that $h_{n}(z_{n})=0$.
Thus, for each $z\in\Omega$,
\[
\left|h_{n}(z)\right|=\left|h_{n}(z)-h_{n}(z_{n})\right|\le M\cdot\mathrm{diam}(\Omega),
\]
which shows that $\left\{ h_{n}\right\} _{n\ge1}$ is uniformly bounded
on $\Omega$. Hence, by \cite[Theorems 1.23 and 2.6]{MR1805196} and
by moving to a subsequence without changing notation, there exists
a harmonic function $h_{0}:\Omega\rightarrow\mathbb{R}$ such that
$h_{n}\overset{n}{\rightarrow}h_{0}$ and $\nabla h_{n}\overset{n}{\rightarrow}\nabla h_{0}$
uniformly on compact subsets of $\Omega$.

Since $\overline{W}$ is compact and $h$ is regular,
\[
\min\left\{ \left|\nabla h(z)\right|\::\:z\in\overline{W}\right\} >0.
\]
Hence, by bounded distortion and the chain rule,
\[
\inf\left\{ \left|\nabla h_{n}(z)\right|\::\:n\ge1\text{ and }z\in\Omega\right\} >0.
\]
From this, and since $\nabla h_{n}\overset{n}{\rightarrow}\nabla h_{0}$
uniformly on compact subsets of $\Omega$, it follows that $h_{0}$
is regular. Thus, since $h_{0}$ is real-analytic (being harmonic)
and by the real-analytic implicit function theorem (see \cite[Theorem 2.3.5]{MR1916029}),
we obtain that $h_{0}^{-1}\{0\}$ is a regular real-analytic curve.

Let us show that $\mu\left(h_{0}^{-1}\{0\}\right)=1$. Let $\epsilon>0$
be given. For each $n\ge1$,
\[
\mu\left(h_{n}^{-1}\{0\}\right)=\varphi_{u_{n}}\mu\left(h^{-1}\{0\}\right)\ge\varphi_{u_{n}}\mu\left(f(J)\right)>1-\frac{1}{n}.
\]
From this and since $h_{n}\overset{n}{\rightarrow}h_{0}$ uniformly
on compact subsets of $\Omega$, it follows that there exists $n\ge1$
such that $\mu\left(h_{n}^{-1}\{0\}\right)>1-\epsilon$ and $\Vert h_{0}-h_{n}\Vert_{K_{\Phi}}<\epsilon$,
where $\Vert\cdot\Vert_{K_{\Phi}}$ denotes the supremum norm on $K_{\Phi}$.
Given $z\in K_{\Phi}\cap h_{n}^{-1}\{0\}$,
\[
\left|h_{0}(z)\right|=\left|h_{0}(z)-h_{n}(z)\right|<\epsilon,
\]
from which it follows that
\[
\mu\left(h_{0}^{-1}(-\epsilon,\epsilon)\right)\ge\mu\left(K_{\Phi}\cap h_{n}^{-1}\{0\}\right)>1-\epsilon.
\]
Since this holds for all $\epsilon>0$ and $h_{0}^{-1}\{0\}=\cap_{\epsilon>0}h_{0}^{-1}(-\epsilon,\epsilon)$,
we obtain that $\mu\left(h_{0}^{-1}\{0\}\right)=1$.

Since $h_{0}^{-1}\{0\}$ is a closed subset of $\Omega$ and $K_{\Phi}$
is the support of $\mu$, it follows that $K_{\Phi}\subset h_{0}^{-1}\{0\}$.
Let $\Gamma$ denote the union of all connected components of $h_{0}^{-1}\{0\}$
whose intersection with $K_{\Phi}$ is nonempty. Note that $\Gamma$
is a regular real-analytic curve.

Since $h_{0}^{-1}\{0\}$ is a manifold it is locally connected, and
so its connected components are open in $h_{0}^{-1}\{0\}$. Thus,
$\Gamma$ is closed in $h_{0}^{-1}\{0\}$, from which it follows that
$\Gamma$ is also closed in $\Omega$. We shall arrive at the desired
contradiction by showing that $\Gamma$ is $\Phi$-invariant.

Let $\Gamma_{0}$ be a connected component of $h_{0}^{-1}\{0\}$ such
that $\Gamma_{0}\subset\Gamma$, and fix $z_{0}\in\Gamma_{0}\cap K_{\Phi}$.
Since $\Gamma_{0}$ is open in $h_{0}^{-1}\{0\}$, there exists $\epsilon>0$
such that
\begin{equation}
D_{\epsilon}(z_{0})\cap h_{0}^{-1}\{0\}\subset\Gamma_{0}.\label{eq:cont in Gamma}
\end{equation}
Let $u\in\Lambda^{*}$ be such that $z_{0}\in\varphi_{u}(K_{\Phi})$
and $\mathrm{diam}\left(\varphi_{u}(K_{\Phi})\right)<\epsilon$. From
these properties, since $\varphi_{u}(K_{\Phi})\subset K_{\Phi}\subset h_{0}^{-1}\{0\}$,
and by (\ref{eq:cont in Gamma}), we obtain that $\varphi_{u}(K_{\Phi})\subset\Gamma_{0}$.

Given $i\in\Lambda$ we have $\varphi_{i}\left(\varphi_{u}(K_{\Phi})\right)\subset K_{\Phi}\subset h_{0}^{-1}\{0\}$,
and so
\begin{equation}
\varphi_{u}(K_{\Phi})\subset\left(h_{0}\circ\varphi_{i}\right)^{-1}\{0\}\cap\Gamma_{0}.\label{eq:varph_u(K) con in 0 set}
\end{equation}
Note that, since the maps in $\Phi$ do not have a common fixed point,
the set $K_{\Phi}$ is necessarily uncountable, and so $\varphi_{u}(K_{\Phi})$
is also uncountable. From this and (\ref{eq:varph_u(K) con in 0 set}),
it follows that $\left(h_{0}\circ\varphi_{i}\right)^{-1}\{0\}\cap\Gamma_{0}$
has a limit point in $\Gamma_{0}$. Thus, since $h_{0}\circ\varphi_{i}$
is harmonic and by Lemma \ref{lem:harmonic =00003D0 on connected curve},
$h_{0}\left(\varphi_{i}(z)\right)=0$ for all $z\in\Gamma_{0}$, which
implies that $\varphi_{i}\left(\Gamma_{0}\right)\subset h_{0}^{-1}\{0\}$.

Since $\varphi_{i}\left(\Gamma_{0}\right)$ is connected, $\varphi_{i}\left(\Gamma_{0}\right)\subset\Gamma_{1}$
for some connected component $\Gamma_{1}$ of $h_{0}^{-1}\{0\}$.
Since $\varphi_{i}\left(\varphi_{u}(K_{\Phi})\right)\subset\varphi_{i}\left(\Gamma_{0}\right)\subset\Gamma_{1}$
and $\varphi_{i}\left(\varphi_{u}(K_{\Phi})\right)\subset K_{\Phi}$,
we obtain that $\Gamma_{1}\subset\Gamma$, and so $\varphi_{i}\left(\Gamma_{0}\right)\subset\Gamma$.
We have thus shown that $\Gamma$ is a $\Phi$-invariant regular real-analytic
curve which is closed in $\Omega$. This contradicts our standing
assumptions and completes the proof of the proposition.
\end{proof}

\subsection{The $\mu$-measure of neighborhoods of dyadic cubes}

The purpose of this subsection is to prove the following proposition.
Recall that for a nonempty set $E\subset\mathbb{C}$ and $r>0$, we
write $E^{(r)}$ for the $r$-thickening of $E$. Also recall the
notation $\mathcal{F}_{\epsilon}$ introduced in Section \ref{subsec:A-bi-Lipchitz-estimate}.
\begin{prop}
\label{prop:mu-mass of neigh of dyd cubes}For each $\epsilon>0$
there exists $\delta>0$ such that,
\[
\psi\mu\left(\cup_{D\in\mathcal{D}_{n}^{\mathbb{C}}}(\partial D)^{(\delta2^{-n})}\right)<\epsilon\text{ for all }\psi\in\mathcal{F}_{\epsilon}\text{ and }n\ge1,
\]
where $\partial D$ denotes the boundary of $D$.
\end{prop}

We shall need the following lemma. Its proof can be found, for example,
in \cite[Chapter 8, Proposition 3.5]{MR1976398}.
\begin{lem}
\label{lem:injective or constant}Let $\Omega_{1}\subset\mathbb{C}$
be a domain, and let $f,f_{1},f_{2},...:\Omega_{1}\rightarrow\mathbb{C}$
be holomorphic. Suppose that $f_{n}\overset{n}{\rightarrow}f$ uniformly
on compact subsets of $\Omega_{1}$, and that $f_{n}$ is injective
for each $n\ge1$. Then $f$ is either injective or constant.
\end{lem}

Most of the proof of Proposition \ref{prop:mu-mass of neigh of dyd cubes}
is contained in the proof of the following lemma.
\begin{lem}
\label{lem:pre mu-mass of neigh of dyd cubes}For each $\epsilon>0$
there exists $\delta>0$ such that $\psi\varphi_{u}\mu\left(\left(z+\ell\right)^{(\delta2^{-n})}\right)<\epsilon$
for all $n\ge1$, $u\in\Psi_{n}$, $\psi\in\mathcal{F}_{\epsilon}$,
$\ell\in\mathbb{RP}^{1}$ and $z\in\mathbb{C}$.
\end{lem}

\begin{proof}
Let $C>1$ be a large global constant, and let $\epsilon>0$ be given.
Assume by contradiction that for every $k\ge1$ there exists $n_{k}\ge1$,
$u_{k}\in\Psi_{n_{k}}$, $\psi_{k}\in\mathcal{F}_{\epsilon}$, $\ell_{k}\in\mathbb{RP}^{1}$
and $z_{k}\in\mathbb{C}$ such that
\begin{equation}
\psi_{k}\varphi_{u_{k}}\mu\left(\left(z_{k}+\ell_{k}\right)^{(C^{-1}k^{-1}2^{-n_{k}})}\right)\ge\epsilon.\label{eq:lem for thickend lines =0000231}
\end{equation}

For $k\ge1$ set $r_{k}:=\left|\varphi_{u_{k}}'(0)\right|$, $w_{k}:=S_{r_{k}}^{-1}\circ\psi_{k}\circ\varphi_{u_{k}}(0)$,
and $f_{k}:=T_{w_{k}}^{-1}S_{r_{k}}^{-1}\psi_{k}\varphi_{u_{k}}$.
By (\ref{eq:der at 0 is comp to 2^-n for u in Psi_n}) and (\ref{eq:lem for thickend lines =0000231}),
\begin{equation}
f_{k}\mu\left(\left(\pi_{\ell_{k}^{\perp}}\left(r_{k}^{-1}z_{k}-w_{k}\right)+\ell_{k}\right)^{(1/k)}\right)\ge\epsilon.\label{eq:lem for thickend lines =0000232}
\end{equation}
Note that $f_{k}(0)=0$. Thus, by Lemma \ref{lem:Lip prop of varphi_u}
and since $\Omega$ is bounded, we may assume that
\begin{equation}
\left|f_{k}(z)\right|\le C/\epsilon\text{ for all }z\in\Omega.\label{eq:lem for thickend lines =0000233}
\end{equation}
From this and by (\ref{eq:lem for thickend lines =0000232}),
\begin{equation}
\left|\pi_{\ell_{k}^{\perp}}\left(r_{k}^{-1}z_{k}-w_{k}\right)\right|\le1+C/\epsilon.\label{eq:lem for thickend lines =0000234}
\end{equation}

From (\ref{eq:lem for thickend lines =0000233}) and (\ref{eq:lem for thickend lines =0000234}),
by Montel's theorem, and by moving to a subsequence without changing
the notation, there exist $w\in\mathbb{C}$, $\ell\in\mathbb{RP}^{1}$,
and a holomorphic $f:\Omega\rightarrow\mathbb{C}$ such that $\pi_{\ell_{k}^{\perp}}\left(r_{k}^{-1}z_{k}-w_{k}\right)\overset{k}{\rightarrow}w$,
$d\left(\ell,\ell_{k}\right)\overset{k}{\rightarrow}0$, and $f_{k}\overset{k}{\rightarrow}f$
uniformly on compact subsets of $\Omega$. From this and by (\ref{eq:lem for thickend lines =0000232})
it follows easily that $f\mu\left(w+\ell\right)\ge\epsilon$.

It also holds that $f_{k}'(0)\overset{k}{\rightarrow}f'(0)$ (see
\cite[page 54]{MR1976398}). Thus, since $\left|f_{k}'(0)\right|\ge\epsilon$
for each $k\ge1$, it follows that $\left|f'(0)\right|\ge\epsilon$.
In particular, $f$ is nonconstant. Hence, since the maps $f_{k}$
are injective, Lemma \ref{lem:injective or constant} implies that
$f$ is injective as well. Since $f$ is also holomorphic, it follows
that $f^{-1}\left(w+\ell\right)$ is a regular real-analytic curve.
But $\mu\left(f^{-1}\left(w+\ell\right)\right)\ge\epsilon$, which
contradicts Proposition \ref{prop:mu(Gamma)=00003D0} and completes
the proof of the lemma.
\end{proof}
\begin{proof}[Proof of Proposition \ref{prop:mu-mass of neigh of dyd cubes}]
Let $0<\epsilon,\eta,\delta<1$ be with $\epsilon^{-1}\ll\eta^{-1}\ll\delta^{-1}$,
let $\psi\in\mathcal{F}_{\epsilon}$ and $n\ge1$ be given, and set
\[
B:=\cup_{D\in\mathcal{D}_{n}^{\mathbb{C}}}(\partial D)^{(\delta2^{-n})}.
\]

For $u\in\Psi_{n}$, it follows from Lemma \ref{lem:Lip prop of varphi_u}
and (\ref{eq:der at 0 is comp to 2^-n for u in Psi_n}) that
\[
\mathrm{diam}\left(\mathrm{supp}\left(\psi\varphi_{u}\mu\right)\right)=O\left(\epsilon^{-1}2^{-n}\right).
\]
Hence,
\[
\#\left\{ D\in\mathcal{D}_{n}^{\mathbb{C}}\::\:\psi\varphi_{u}\mu\left((\partial D)^{(\delta2^{-n})}\right)>0\right\} =O_{\epsilon}(1).
\]
Moreover, by Lemma \ref{lem:pre mu-mass of neigh of dyd cubes} and
since $\epsilon^{-1},\eta^{-1}\ll\delta^{-1}$,
\[
\psi\varphi_{u}\mu\left((\partial D)^{(\delta2^{-n})}\right)<\eta\text{ for all }D\in\mathcal{D}_{n}^{\mathbb{C}}.
\]
Combining these facts, we obtain
\[
\psi\varphi_{u}\mu\left(B\right)=O_{\epsilon}(\eta)\text{ for all }u\in\Psi_{n}.
\]
From this, together with $\psi\mu=\sum_{u\in\Psi_{n}}p_{u}\cdot\psi\varphi_{u}\mu$
and $\epsilon^{-1}\ll\eta^{-1}$, we get that $\psi\mu(B)<\epsilon$,
which completes the proof of the proposition.
\end{proof}

\section{\label{sec:non-saturation-of-components}Non-saturation of components}

The following proposition is the main result of this section. Its
proof is a modification of the proof of \cite[Proposition 1.5]{Rap_Ren_Fur_on_CP1},
but it contains significant differences and requires several adaptations
to our holomorphic setting.
\begin{prop}
\label{prop:lb on ent of proj of comp of nu}Suppose that $\dim\mu<2$.
Then there exists $\gamma>0$ such that for every $\epsilon>0$, $m\ge M(\epsilon)\ge1$,
$n\ge1$, and $\psi\in\mathcal{F}_{\epsilon}$, if we set $\theta:=\psi\mu$,
then
\[
\mathbb{P}\left\{ \underset{w\mathbb{R}\in\mathbb{RP}^{1}}{\inf}\frac{1}{m}H\left(\pi_{w\mathbb{R}}\theta_{z,n},\mathcal{D}_{n+m}\right)>\dim\mu-1+\gamma\right\} >1-\epsilon.
\]
\end{prop}

\begin{rem*}
Let $\theta$ be as in the statement of the proposition. Since $\theta$
has uniform entropy dimension (see Proposition \ref{prop:uni ent dim}
below), Proposition \ref{prop:lb on ent of proj of comp of nu} implies
that, in the terminology introduced by Hochman (see \cite[Definition 2.3]{Ho}),
all but an arbitrarily small proportion of the components of $\theta$
are non-saturated in any direction.
\end{rem*}
Throughout this section, given $u\in\Lambda^{*}$ we shall write $S_{u}$
in place of $S_{\left|\varphi_{u}'(0)\right|}^{-1}$. That is, $S_{u}z=\left|\varphi_{u}'(0)\right|^{-1}z$
for $z\in\mathbb{C}$. Most of the proof of Proposition \ref{prop:lb on ent of proj of comp of nu}
is devoted to establishing the following statement.
\begin{prop}
\label{prop:lb on ent of proj of cylinders of nu}Suppose that $\dim\mu<2$.
Then there exists $0<\gamma<1$ such that for every $\epsilon>0$,
$n\ge N(\epsilon)\ge1$, $z\mathbb{R}\in\mathbb{RP}^{1}$, $u\in\Lambda^{*}$,
and $\psi\in\mathcal{F}_{\epsilon}$,
\[
\frac{1}{n}H\left(S_{u}\pi_{z\mathbb{R}}\psi\varphi_{u}\mu,\mathcal{D}_{n}\right)\ge\dim\mu-1+\gamma.
\]
\end{prop}

The proof of Proposition \ref{prop:lb on ent of proj of cylinders of nu}
involves bounding from below entropies of the form
\begin{equation}
\frac{1}{m}H\left(S_{uv}\pi_{z\mathbb{R}}\psi\varphi_{uv}\mu,\mathcal{D}_{m}\right),\label{eq:entropies needed for prop}
\end{equation}
where $m\ge1$ is large, $u\in\Lambda^{*}$ is fixed, and $v\in\cup_{j=1}^{n}\Lambda^{j}$,
with $n\ge1$ large. Appropriate lower bounds for these entropies,
combined with a multiscale decomposition, the concavity of entropy,
and the stationarity of $\mu$, will yield the proposition.

In Section \ref{subsec:The-trivial-lower bd}, we show that the entropies
in (\ref{eq:entropies needed for prop}) are bounded from below by
$\dim\mu-1$ up to an arbitrarily small error. Section \ref{subsec:The-direction-cocycle}
is devoted to the study of the direction cocycle $\alpha_{n}:\Lambda^{\mathbb{N}}\rightarrow\mathbb{RP}^{1}$
(defined in that section). We prove that it is not a coboundary, which
yields an important non-concentration corollary (Corollary \ref{cor:from equidist prop}).
In Section \ref{subsec:The-nontrivial-lower lb}, we use this corollary
in order to show that, when $v$ is chosen randomly according to $\mathbb{E}_{1\le i\le n}\left(\delta_{\mathbf{U}_{i}}\right)$,
the entropies in (\ref{eq:entropies needed for prop}) are, with non-negligible
probability, bounded from below by $\frac{1}{2}\dim\mu$ up to an
arbitrarily small error. Note that since $\dim\mu<2$, we have $\frac{1}{2}\dim\mu>\dim\mu-1$.
Finally, in Section \ref{subsec:Proof-of-Propositions}, we complete
the proofs of Propositions \ref{prop:lb on ent of proj of comp of nu}
and \ref{prop:lb on ent of proj of cylinders of nu}.

\subsection{\label{subsec:The-trivial-lower bd}The trivial lower bound}

The following lemma will be used several times below.
\begin{lem}
\label{lem:ent wrt mod cont dyad part}For every $\epsilon>0$, $m\ge M(\epsilon)\ge1$,
$u\in\Lambda^{*}$, $\psi\in\mathcal{F}_{\epsilon}$, and $z\mathbb{R},w\mathbb{R}\in\mathbb{RP}^{1}$
with $d\left(z\mathbb{R},w\mathbb{R}\right)\ge\epsilon$, we have
\[
\left|\frac{1}{m}H\left(S_{u}\psi\varphi_{u}\mu,\pi_{z\mathbb{R}}^{-1}\mathcal{D}_{m}\vee\pi_{w\mathbb{R}}^{-1}\mathcal{D}_{m}\right)-\dim\mu\right|<\epsilon.
\]
\end{lem}

\begin{proof}
Let $0<\epsilon<1$ and $m\ge1$ be with $\epsilon^{-1}\ll m$, let
$u\in\Lambda^{*}$ and $\psi\in\mathcal{F}_{\epsilon}$ be given,
let $z\mathbb{R},w\mathbb{R}\in\mathbb{RP}^{1}$ be with $d\left(z\mathbb{R},w\mathbb{R}\right)\ge\epsilon$,
and set
\[
\mathcal{E}:=\pi_{z\mathbb{R}}^{-1}\mathcal{D}_{m}\vee\pi_{w\mathbb{R}}^{-1}\mathcal{D}_{m}.
\]
From $d\left(z\mathbb{R},w\mathbb{R}\right)\ge\epsilon$ it follows
easily that the partitions $\mathcal{E}$ and $\mathcal{D}_{m}^{\mathbb{C}}$
are $O_{\epsilon}(1)$-commensurable. That is, for each $E\in\mathcal{E}$
and $D\in\mathcal{D}_{m}^{\mathbb{C}}$
\[
\#\left\{ D'\in\mathcal{D}_{m}^{\mathbb{C}}\::\:D'\cap E\ne\emptyset\right\} ,\#\left\{ E'\in\mathcal{E}\::\:E'\cap D\ne\emptyset\right\} =O_{\epsilon}(1).
\]
Hence, by \cite[Lemma 3.2]{Ho1} and since $\epsilon^{-1}\ll m$,
\[
\left|\frac{1}{m}H\left(S_{u}\psi\varphi_{u}\mu,\mathcal{E}\right)-\frac{1}{m}H\left(S_{u}\psi\varphi_{u}\mu,\mathcal{D}_{m}\right)\right|<\epsilon.
\]

By Lemma \ref{lem:Lip prop of varphi_u}, $S_{u}\circ\psi\circ\varphi_{u}$
is bi-Lipschitz with bi-Lipschitz constant $O\left(1/\epsilon\right)$.
Thus, by Lemma \ref{lem:dyad ent =000026 lip func},
\[
\left|\frac{1}{m}H\left(S_{u}\psi\varphi_{u}\mu,\mathcal{D}_{m}\right)-\frac{1}{m}H\left(\mu,\mathcal{D}_{m}\right)\right|<\epsilon.
\]
Moreover, by (\ref{eq:ent dim =00003D exact dim}),
\[
\left|\frac{1}{m}H\left(\mu,\mathcal{D}_{m}\right)-\dim\mu\right|<\epsilon.
\]
All of this completes the proof of the lemma.
\end{proof}
We can now obtain the following lower bound.
\begin{lem}
\label{lem:triv lb}For every $\epsilon>0$, $m\ge M(\epsilon)\ge1$,
$z\mathbb{R}\in\mathbb{RP}^{1}$, $u\in\Lambda^{*}$, and $\psi\in\mathcal{F}_{\epsilon}$,
\[
\frac{1}{m}H\left(S_{u}\pi_{z\mathbb{R}}\psi\varphi_{u}\mu,\mathcal{D}_{m}\right)>\dim\mu-1-\epsilon.
\]
\end{lem}

\begin{proof}
Let $0<\epsilon<1$ and $m\ge1$ be with $\epsilon^{-1}\ll m$, let
$z\mathbb{R}\in\mathbb{RP}^{1}$, $u\in\Lambda^{*}$, and $\psi\in\mathcal{F}_{\epsilon}$
be given, and set $\theta:=S_{u}\psi\varphi_{u}\mu$. Recall that
$\left(z\mathbb{R}\right)^{\perp}\in\mathbb{RP}^{1}$ denotes the
line perpendicular to $z\mathbb{R}$, and set
\[
\mathcal{E}:=\pi_{z\mathbb{R}}^{-1}\mathcal{D}_{m}\vee\pi_{\left(z\mathbb{R}\right)^{\perp}}^{-1}\mathcal{D}_{m}.
\]
By Lemma \ref{lem:ent wrt mod cont dyad part},
\[
\left|\frac{1}{m}H\left(\theta,\mathcal{E}\right)-\dim\mu\right|<\epsilon.
\]

By Lemma \ref{lem:Lip prop of varphi_u},
\[
\mathrm{diam}\left(\mathrm{supp}\left(\theta\right)\right)=O\left(1/\epsilon\right).
\]
Hence, since $\epsilon^{-1}\ll m$,
\[
\frac{1}{m}H\left(\pi_{\left(z\mathbb{R}\right)^{\perp}}\theta,\mathcal{D}_{m}\right)\le1+\epsilon.
\]
Additionally, by the conditional entropy formula,
\[
\frac{1}{m}H\left(\theta,\mathcal{E}\right)\le\frac{1}{m}H\left(\pi_{z\mathbb{R}}\theta,\mathcal{D}_{m}\right)+\frac{1}{m}H\left(\pi_{\left(z\mathbb{R}\right)^{\perp}}\theta,\mathcal{D}_{m}\right).
\]

All of this gives
\[
\frac{1}{m}H\left(\pi_{z\mathbb{R}}\theta,\mathcal{D}_{m}\right)\ge\dim\mu-1-2\epsilon.
\]
Since $S_{u}\circ\pi_{z\mathbb{R}}=\pi_{z\mathbb{R}}\circ S_{u}$,
this completes the proof of the lemma.
\end{proof}

\subsection{\label{subsec:The-direction-cocycle}The direction cocycle}

Let $\alpha:\Lambda^{\mathbb{N}}\rightarrow\mathbb{RP}^{1}$ be such
that
\[
\alpha(\omega)=\varphi_{\omega_{0}}'\left(\Pi(\sigma\omega)\right)\mathbb{R}\text{ for }\omega\in\Lambda^{\mathbb{N}}.
\]
Define a cocycle, which we call the direction cocycle, by setting
for each $n\ge0$ and $\omega\in\Lambda^{\mathbb{N}}$
\[
\alpha_{n}(\omega):=\prod_{i=0}^{n-1}\alpha\left(\sigma^{i}\omega\right),
\]
where recall from Section \ref{subsec:Basic-notations} that $\mathbb{RP}^{1}$
is considered as a multiplicative group. Note that $\alpha_{0}(\omega)=\mathbb{R}$,
and that by the chain rule
\begin{equation}
\alpha_{n}(\omega)=\varphi_{\omega|_{n}}'\left(\Pi\left(\sigma^{n}\omega\right)\right)\mathbb{R}.\label{eq:alpha_n equals}
\end{equation}

The following proposition is the main result of this subsection.
\begin{prop}
\label{prop:noncoboundary}There does not exist a Borel measurable
map $f:\Lambda^{\mathbb{N}}\rightarrow\mathbb{RP}^{1}$ such that
$\alpha(\omega)=f(\omega)^{-1}f\left(\sigma\omega\right)$ for $\beta$-a.e.
$\omega$.
\end{prop}

\begin{rem*}
In the terminology of measurable cohomology (see \cite{MR578731}),
Proposition \ref{prop:noncoboundary} asserts that the cocycle $\alpha_{n}$
is not a coboundary.
\end{rem*}
The proof of the proposition requires some preparation. We omit the
proof of the following lemma, which follows easily from the real-analytic
implicit function theorem (see \cite[Theorem 2.3.5]{MR1916029}) and
the real-analyticity of harmonic functions.
\begin{lem}
\label{lem:level sets of nonconst harmonic}Let $\Omega_{1}\subset\mathbb{C}$
be a domain, and let $h:\Omega_{1}\rightarrow\mathbb{R}$ be a nonconstant
harmonic function. Then for each $c\in\mathbb{R}$ there exists a
(possibly empty) countable set $A\subset\Omega_{1}$ and a (possibly
empty) regular real-analytic curve $\Gamma\subset\Omega_{1}$ such
that $h^{-1}\{c\}=A\cup\Gamma$.
\end{lem}

Using this lemma, we can establish the following statement.
\begin{lem}
\label{lem:lem bef no cobd prop}Let $\psi:\Omega\rightarrow\mathbb{C}$
be holomorphic and injective, and let $j\in\Lambda$ be given. Set
$g:=\psi\circ\varphi_{j}\circ\psi^{-1}$, and suppose that
\begin{equation}
g'\left(\psi(z)\right)\mathbb{R}=\mathbb{R}\text{ for all }z\in K_{\Phi}.\label{eq:cond in lem bef no cobd prop}
\end{equation}
Then $g$ is homothetic.
\end{lem}

\begin{proof}
Set $U:=\mathbb{C}\setminus i\mathbb{R}_{\ge0}$ and
\[
\Omega_{1}:=\left\{ z\in\Omega\::\:g'\left(\psi(z)\right)\in U\right\} ,
\]
so that $\Omega_{1}$ is an open subset of $\Omega$. From (\ref{eq:cond in lem bef no cobd prop})
and since $\mathrm{supp}(\mu)=K_{\Phi}$, it follows that $\mu(\Omega_{1})=1$.
Let $\log_{U}$ denote the natural branch of the logarithm corresponding
to the domain $U$, so that
\[
\mathrm{Im}\left(\log_{U}(z)\right)\in\left(-\frac{3\pi}{2},\frac{\pi}{2}\right)\text{ for all }z\in U.
\]
Setting
\[
h:=\mathrm{Im}\circ\log_{U}\circ g'\circ\psi,
\]
it holds that $h:\Omega_{1}\rightarrow\mathbb{R}$ is harmonic.

Let $\Omega_{2}$ be a connected component of $\Omega_{1}$ with $\mu(\Omega_{2})>0$.
Assume by contradiction that $h$ is nonconstant on $\Omega_{2}$,
and set
\[
L:=\Omega_{2}\cap h^{-1}\{0,-\pi\}.
\]
By Lemma \ref{lem:level sets of nonconst harmonic}, there exists
a (possibly empty) countable set $A\subset\Omega_{2}$ and a (possibly
empty) regular real-analytic curve $\Gamma\subset\Omega_{2}$ such
that $L=A\cup\Gamma$. On the other hand, by (\ref{eq:cond in lem bef no cobd prop})
we have $\mu\left(h^{-1}\{0,-\pi\}\right)=1$, and so $\mu(L)=\mu(\Omega_{2})>0$.
Since the maps in $\Phi$ do not have a common fixed point, $\mu$
is clearly nonatomic. Hence $\mu(A)=0$, which implies that $\mu(\Gamma)=\mu(L)>0$.
But this contradicts Proposition \ref{prop:mu(Gamma)=00003D0}, and
so $h$ must be constant on $\Omega_{2}$.

Since $h$ is constant on $\Omega_{2}$, there exists $0\ne z_{0}\in\mathbb{C}$
such that $g'\left(\psi(\Omega_{2})\right)\subset z_{0}\mathbb{R}_{>0}$.
Thus, since $\psi(\Omega_{2})$ is open and $z_{0}\mathbb{R}_{>0}$
has empty interior, the open mapping theorem implies that $g'$ is
constant on $\psi(\Omega)$. From this and (\ref{eq:cond in lem bef no cobd prop}),
it follows that there exists $0\ne r\in\mathbb{R}$ such that $g'(z)=r$
for all $z\in\psi(\Omega)$, which shows that $g$ is homothetic.
\end{proof}
\begin{proof}[Proof of Proposition \ref{prop:noncoboundary}]
Assume by contradiction that there exists a Borel measurable $f:\Lambda^{\mathbb{N}}\rightarrow\mathbb{RP}^{1}$
such that $\alpha(\omega)=f(\omega)^{-1}f\left(\sigma\omega\right)$
for $\beta$-a.e. $\omega$. Then by (\ref{eq:alpha_n equals}),
\[
\varphi_{\omega|_{n}}'\left(\Pi\left(\sigma^{n}\omega\right)\right)\mathbb{R}=f(\omega)^{-1}f\left(\sigma^{n}\omega\right)\text{ for all }n\ge0\text{ and }\beta\text{-a.e. }\omega,
\]
which implies that
\begin{equation}
f(\omega)=\varphi_{u}'\left(\Pi\left(\omega\right)\right)\mathbb{R}f(u\omega)\text{ for all }u\in\Lambda^{*}\text{ and }\beta\text{-a.e. }\omega.\label{eq:first eq in comp arg}
\end{equation}

By Lusin's theorem, there exists a Borel subset $E$ of $\Lambda^{\mathbb{N}}$
such that $\beta(E)>0$ and $f|_{E}$ is continuous, where $f|_{E}$
denotes the restriction of $f$ to $E$. By the martingale theorem,
there exists $\eta\in E$ such that $\underset{n\rightarrow\infty}{\lim}\beta_{[\eta|_{n}]}(E)=1$.
Note that, for each $n\ge1$, the measure $\beta_{[\eta|_{n}]}$ equals
the pushforward of $\beta$ via the map $\omega\mapsto\eta|_{n}\omega$.
Thus,
\[
\underset{n\rightarrow\infty}{\lim}\beta\left\{ \omega\::\:\eta|_{n}\omega\in E\right\} =1.
\]
Given $\epsilon>0$, this implies that
\[
\underset{n\rightarrow\infty}{\lim}\beta\left\{ \omega\::\:d\left(f\left(\eta|_{n}\omega\right),f(\eta)\right)>\epsilon\right\} =0.
\]
In other words, the sequence $\left\{ \omega\mapsto f\left(\eta|_{n}\omega\right)\right\} _{n\ge1}$
converges in probability to $f(\eta)$. Hence, there exists a strictly
increasing sequence $\left(n_{k}\right)_{k\ge1}\subset\mathbb{Z}_{>0}$
such that
\begin{equation}
\underset{k\rightarrow\infty}{\lim}f\left(\eta|_{n_{k}}\omega\right)=f(\eta)\text{ for }\beta\text{-a.e }\omega.\label{eq:conv to f(eta)}
\end{equation}

For $k\ge1$, write
\[
\psi_{k}:=S_{\eta|_{n_{k}}}\circ T_{\varphi_{\eta|_{n_{k}}}(0)}^{-1}\circ\varphi_{\eta|_{n_{k}}}.
\]
By Lemma \ref{lem:Lip prop of varphi_u} and since $\psi_{k}(0)=0$
for $k\ge1$, the sequence $\left\{ \psi_{k}\right\} _{k\ge1}$ is
uniformly bounded on $\Omega$. Thus, by Montel's theorem and by moving
to a subsequence without changing notation, there exists a holomorphic
$\psi:\Omega\rightarrow\mathbb{C}$ such that $\psi_{k}\overset{k}{\rightarrow}\psi$
uniformly on compact subsets of $\Omega$.

It also holds that $\psi_{k}'\overset{k}{\rightarrow}\psi'$ uniformly
on compact subsets of $\Omega$ (see \cite[page 54]{MR1976398}).
Moreover, by bounded distortion,
\[
\inf\left\{ \left|\psi_{k}'(z)\right|\::\:k\ge1\text{ and }z\in\Omega\right\} >0,
\]
and so $\psi'(z)\ne0$ for all $z\in\Omega.$ Hence, for each $z\in\Omega$
\[
\underset{k\rightarrow\infty}{\lim}\varphi_{\eta|_{n_{k}}}'(z)\mathbb{R}=\underset{k\rightarrow\infty}{\lim}\psi_{k}'(z)\mathbb{R}=\psi'(z)\mathbb{R}.
\]
From this, (\ref{eq:first eq in comp arg}), and (\ref{eq:conv to f(eta)}),
it follows that for $\beta$-a.e. $\omega$
\begin{equation}
f(\omega)=\underset{k\rightarrow\infty}{\lim}\varphi_{\eta|_{n_{k}}}'\left(\Pi\left(\omega\right)\right)\mathbb{R}f\left(\eta|_{n_{k}}\omega\right)=\psi'\left(\Pi\left(\omega\right)\right)\mathbb{R}f(\eta).\label{eq:f(omega)=00003Dpsi'(Pi(omega))R}
\end{equation}

Recall that the maps in $\Phi$ are all assumed to be injective on
$\Omega$, from which it follows that $\psi_{k}$ is injective on
$\Omega$ for each $k\ge1$. Moreover, since $\psi$ has nonvanishing
derivative it is certainly nonconstant. Hence, by Lemma \ref{lem:injective or constant},
we obtain that $\psi$ is injective.

Fix $j\in\Lambda$ and set $g:=\psi\circ\varphi_{j}\circ\psi^{-1}$.
By (\ref{eq:first eq in comp arg})
\[
f(j\omega)\varphi_{j}'\left(\Pi\left(\omega\right)\right)\mathbb{R}f(\omega)^{-1}=\mathbb{R}\text{ for }\beta\text{-a.e. }\omega.
\]
Thus, by (\ref{eq:f(omega)=00003Dpsi'(Pi(omega))R})
\[
\psi'\left(\Pi\left(j\omega\right)\right)\varphi_{j}'\left(\Pi\left(\omega\right)\right)\psi'\left(\Pi\left(\omega\right)\right)^{-1}\mathbb{R}=\mathbb{R}\text{ for }\beta\text{-a.e. }\omega.
\]
Note that $\Pi\left(j\omega\right)=\varphi_{j}\left(\Pi(\omega)\right)$
for $\omega\in\Lambda^{\mathbb{N}}$, and $\psi'(z)^{-1}=\left(\psi^{-1}\right)'(\psi(z))$
for $z\in\Omega$. Hence, for $\beta$-a.e. $\omega$
\begin{equation}
g'\left(\psi\circ\Pi(\omega)\right)\mathbb{R}=\psi'\left(\varphi_{j}\left(\Pi(\omega)\right)\right)\varphi_{j}'\left(\Pi\left(\omega\right)\right)\left(\psi^{-1}\right)'\left(\psi\circ\Pi(\omega)\right)\mathbb{R}=\mathbb{R}.\label{eq:for beta-a.e omega =00003DR}
\end{equation}
Since $\mu=\Pi\beta$ and $\mathrm{supp}(\mu)=K_{\Phi}$, this clearly
implies that $g'\left(\psi(z)\right)\mathbb{R}=\mathbb{R}$ for all
$z\in K_{\Phi}$. Thus, by Lemma \ref{lem:lem bef no cobd prop},
$g$ is homothetic.

We have thus shown that for each $j\in\Lambda$ the map $\psi\circ\varphi_{j}\circ\psi^{-1}:\psi(\Omega)\rightarrow\mathbb{C}$
is homothetic, and so $\Phi$ is holomorphically conjugate to a homothetic
IFS, which contradicts our standing assumption. Hence, there does
not exist a Borel measurable $f:\Lambda^{\mathbb{N}}\rightarrow\mathbb{RP}^{1}$
such that $\alpha(\omega)=f(\omega)^{-1}f\left(\sigma\omega\right)$
for $\beta$-a.e. $\omega$, which completes the proof of the proposition.
\end{proof}
We next state an important implication of Proposition \ref{prop:noncoboundary}.
First, we need the following definition.
\begin{defn}
Given $\delta>0$, we say that $\theta\in\mathcal{M}\left(\mathbb{RP}^{1}\right)$
is $\delta$-concentrated if there exists $z\mathbb{R}\in\mathbb{RP}^{1}$
such that $\theta\left(B\left(z\mathbb{R},\delta\right)\right)>1-\delta$.
\end{defn}

The following proposition is a consequence of Proposition \ref{prop:noncoboundary}.
The derivation is identical to that of \cite[Proposition 4.8]{Rap_Ren_Fur_on_CP1}
from \cite[Proposition 4.6]{Rap_Ren_Fur_on_CP1} and is therefore
omitted.
\begin{prop}
\label{prop:equidist prop}There exists $\delta>0$ such that for
every continuous $h:\Lambda^{\mathbb{N}}\rightarrow\mathbb{RP}^{1}$
and for $\beta$-a.e. $\omega$, the sequence $\left(\alpha_{n}(\omega)h\left(\sigma^{n}\omega\right)\right)_{n\ge0}$
is equidistributed with respect to some $\theta\in\mathcal{M}\left(\mathbb{RP}^{1}\right)$
that is not $\delta$-concentrated.
\end{prop}

The following corollary follows immediately from Proposition \ref{prop:equidist prop}.
Recall that $\lambda_{n}$ denotes the uniform probability measure
on $\mathcal{N}_{n}:=\left\{ 1,...,n\right\} $.
\begin{cor}
\label{cor:from equidist prop}There exists $0<\delta<1$ such that
for every continuous $h:\Lambda^{\mathbb{N}}\rightarrow\mathbb{RP}^{1}$
and for $\beta$-a.e. $\omega$, there exists $N_{h,\omega}\ge1$
so that for every $n\ge N_{h,\omega}$,
\[
\lambda_{n}\left\{ i\in\mathcal{N}_{n}\::\:d\left(\alpha_{i}(\omega)h\left(\sigma^{i}\omega\right),z\mathbb{R}\right)>\delta\right\} >\delta\text{ for all }z\mathbb{R}\in\mathbb{RP}^{1}.
\]
\end{cor}

\subsection{\label{subsec:The-nontrivial-lower lb}The nontrivial lower bound}

The purpose of this subsection is to prove the following proposition.
Recall the random words $\mathbf{U}_{i}$ introduced in Section \ref{subsec:Symbolic-notations}.
\begin{prop}
\label{prop:nontriv portion > dim(mu)/2}There exists $0<\delta<1$
such that for every $\epsilon>0$, $m\ge M(\epsilon)\ge1$, $n\ge N(\epsilon,m)\ge1$,
$z\mathbb{R}\in\mathbb{RP}^{1}$, $u\in\Lambda^{*}$, and $\psi\in\mathcal{F}_{\epsilon}$,
\[
\mathbb{P}_{1\le i\le n}\left\{ \frac{1}{m}H\left(S_{u\mathbf{U}_{i}}\pi_{z\mathbb{R}}\psi\varphi_{u\mathbf{U}_{i}}\mu,\mathcal{D}_{m}\right)>\frac{1}{2}\dim\mu-\epsilon\right\} >\delta.
\]
\end{prop}

The proof of the proposition requires some preparation.
\begin{lem}
\label{lem:ent > 1/2 dim mu outside a small interval}For every $\epsilon>0$,
$m\ge M(\epsilon)\ge1$, and $u\in\Lambda^{*}$, there exists $z\mathbb{R}\in\mathbb{RP}^{1}$
such that
\[
\frac{1}{m}H\left(S_{u}\pi_{w\mathbb{R}}\varphi_{u}\mu,\mathcal{D}_{m}\right)>\frac{1}{2}\dim\mu-\epsilon\text{ for all }w\mathbb{R}\in\mathbb{RP}^{1}\setminus B\left(z\mathbb{R},\epsilon\right).
\]
\end{lem}

\begin{proof}
Let $0<\epsilon<1$ and $m\ge1$ be with $\epsilon^{-1}\ll m$, and
let $u\in\Lambda^{*}$ be given. For each $w\mathbb{R}\in\mathbb{RP}^{1}$
set
\[
H\left(w\mathbb{R}\right):=\frac{1}{m}H\left(S_{u}\pi_{w\mathbb{R}}\varphi_{u}\mu,\mathcal{D}_{m}\right),
\]
and let $z\mathbb{R}\in\mathbb{RP}^{1}$ be such that
\[
H\left(z\mathbb{R}\right)\le\underset{w\mathbb{R}\in\mathbb{RP}^{1}}{\inf}H\left(w\mathbb{R}\right)+\epsilon.
\]
From Lemma \ref{lem:ent wrt mod cont dyad part}, by the conditional
entropy formula, and by the last inequality, it follows that for each
$w\mathbb{R}\in\mathbb{RP}^{1}\setminus B\left(z\mathbb{R},\epsilon\right)$
\[
\dim\mu-\epsilon<H\left(z\mathbb{R}\right)+H\left(w\mathbb{R}\right)\le2H\left(w\mathbb{R}\right)+\epsilon,
\]
which proves the lemma.
\end{proof}
The proof of Proposition \ref{prop:nontriv portion > dim(mu)/2} relies
on a linearization argument which is contained in the proof of the
following lemma.
\begin{lem}
\label{lem:linearization in non sat argument}For every $\epsilon>0$,
$m\ge M(\epsilon)\ge1$, $k\ge K(\epsilon,m)\ge1$, $u\in\Lambda^{*}$,
$\psi\in\mathcal{F}_{\epsilon}$, $i\in\mathbb{Z}_{>0}$, $\omega\in\Lambda^{\mathbb{N}}$,
and $z\mathbb{R}\in\mathbb{RP}^{1}$, we have
\[
\frac{1}{m}H\left(S_{u\omega|_{i+k}}\pi_{z\mathbb{R}}\psi\varphi_{u\omega|_{i+k}}\mu,\mathcal{D}_{m}\right)\ge\frac{1}{m}H\left(S_{(\sigma^{i}\omega)|_{k}}\pi_{\ell}\varphi_{(\sigma^{i}\omega)|_{k}}\mu,\mathcal{D}_{m}\right)-\epsilon,
\]
where
\begin{equation}
\ell:=\frac{z}{\left(\psi\circ\varphi_{u}\right)'\left(\Pi\omega\right)}\mathbb{R}\alpha_{i}(\omega)^{-1}\in\mathbb{RP}^{1}.\label{eq:def of ell}
\end{equation}
\end{lem}

\begin{proof}
Let $C>1$ be a large global constant, and let $\epsilon\in(0,1)$
and $m,k\in\mathbb{Z}_{>0}$ be such that $C,\epsilon^{-1}\ll m\ll k$.
Let $u\in\Lambda^{*}$, $\psi\in\mathcal{F}_{\epsilon}$, $i\in\mathbb{Z}_{>0}$,
$\omega\in\Lambda^{\mathbb{N}}$, and $z\mathbb{R}\in\mathbb{RP}^{1}$
be given, and set
\[
H:=\frac{1}{m}H\left(S_{u\omega|_{i+k}}\pi_{z\mathbb{R}}\psi\varphi_{u\omega|_{i+k}}\mu,\mathcal{D}_{m}\right).
\]

Write
\[
r_{\max}:=\sup\left\{ \left|\varphi_{j}'(w)\right|\::\:j\in\Lambda\text{ and }w\in\Omega\right\} .
\]
By Lemma \ref{lem:Lip prop of varphi_u},
\[
\mathrm{diam}\left(\varphi_{(\sigma^{i}\omega)|_{k}}(\Omega)\right)\le C\left|\varphi_{(\sigma^{i}\omega)|_{k}}'(0)\right|\le Cr_{\max}^{k}.
\]
Hence, setting $a:=\Pi\left(\sigma^{i}\omega\right)$ and $q:=P_{1,a}\left(\psi\circ\varphi_{u\omega|_{i}}\right)$,
it follows from Lemma \ref{lem:Taylor estimate} and since $a\in\varphi_{(\sigma^{i}\omega)|_{k}}(\Omega)$
that for all $w\in\varphi_{(\sigma^{i}\omega)|_{k}}(\Omega)$
\begin{eqnarray*}
\left|\psi\circ\varphi_{u\omega|_{i}}(w)-q(w)\right| & \le & C^{3}\epsilon^{-1}\left|\varphi_{u\omega|_{i}}'(0)\right|\left|\varphi_{(\sigma^{i}\omega)|_{k}}'(0)\right|^{2}\\
 & \le & C^{3}\epsilon^{-1}r_{\max}^{k}\left|\varphi_{u\omega|_{i}}'(0)\right|\left|\varphi_{(\sigma^{i}\omega)|_{k}}'(0)\right|.
\end{eqnarray*}
From this, by bounded distortion, by Lemma \ref{lem:close func --> close ent},
and since $C,\epsilon^{-1}\ll m\ll k$,
\[
H+\epsilon\ge\frac{1}{m}H\left(S_{u\omega|_{i+k}}\pi_{z\mathbb{R}}q\varphi_{(\sigma^{i}\omega)|_{k}}\mu,\mathcal{D}_{m}\right).
\]

Set $b:=\frac{\left(\psi\circ\varphi_{u\omega|_{i}}\right)'(a)}{\left|\left(\psi\circ\varphi_{u\omega|_{i}}\right)'(a)\right|}$,
and recall that $S_{b}w=bw$ for $w\in\mathbb{C}$. Since $S_{b}^{-1}$
is an isometry of $\mathbb{C}$,
\begin{equation}
H+2\epsilon\ge\frac{1}{m}H\left(S_{u\omega|_{i+k}}S_{b}^{-1}\pi_{z\mathbb{R}}q\varphi_{(\sigma^{i}\omega)|_{k}}\mu,\mathcal{D}_{m}\right).\label{eq:H+2epsilon>=00003D}
\end{equation}

By the definition of the orthogonal projections and from (\ref{eq:alpha_n equals}),
\[
S_{b}^{-1}\circ\pi_{z\mathbb{R}}\circ S_{b}=\pi_{b^{-1}z\mathbb{R}}=\pi_{\ell},
\]
where $\ell\in\mathbb{RP}^{1}$ is defined in (\ref{eq:def of ell}).
Thus, for $w\in\mathbb{C}$,
\begin{eqnarray*}
S_{b}^{-1}\circ\pi_{z\mathbb{R}}\circ q(w) & = & S_{b}^{-1}\circ\pi_{z\mathbb{R}}\circ q(0)+S_{b}^{-1}\circ\pi_{z\mathbb{R}}\left(\left(\psi\circ\varphi_{u\omega|_{i}}\right)'(a)\cdot w\right)\\
 & = & S_{b}^{-1}\circ\pi_{z\mathbb{R}}\circ q(0)+\left|\left(\psi\circ\varphi_{u\omega|_{i}}\right)'(a)\right|\cdot\pi_{\ell}(w).
\end{eqnarray*}
From this, from (\ref{eq:H+2epsilon>=00003D}), by bounded distortion,
since $\psi\in\mathcal{F}_{\epsilon}$, and since $\epsilon^{-1}\ll m$,
\[
H+3\epsilon\ge\frac{1}{m}H\left(S_{(\sigma^{i}\omega)|_{k}}\pi_{\ell}\varphi_{(\sigma^{i}\omega)|_{k}}\mu,\mathcal{D}_{m}\right),
\]
which completes the proof of the lemma.
\end{proof}
\begin{proof}[Proof of Proposition \ref{prop:nontriv portion > dim(mu)/2}]
Let $0<\delta<1$ be as obtained in Corollary \ref{cor:from equidist prop},
let $0<\epsilon<\delta/3$ and $m,k,n\in\mathbb{Z}_{>0}$ be with
$\epsilon^{-1}\ll m\ll k\ll n$, and fix $z\mathbb{R}\in\mathbb{RP}^{1}$,
$u\in\Lambda^{*}$, and $\psi\in\mathcal{F}_{\epsilon}$. Set
\[
H(v):=\frac{1}{m}H\left(S_{uv}\pi_{z\mathbb{R}}\psi\varphi_{uv}\mu,\mathcal{D}_{m}\right)\text{ for }v\in\Lambda^{*},
\]
and let
\[
P:=\mathbb{P}_{1\le i\le n}\left\{ H\left(\mathbf{U}_{i}\right)>\frac{1}{2}\dim\mu-\epsilon\right\} .
\]

Recalling that $\lambda_{n}$ denotes the uniform probability measure
on $\mathcal{N}_{n}:=\left\{ 1,...,n\right\} $,
\begin{eqnarray*}
P & = & \int\beta\left\{ \omega:H\left(\omega|_{i}\right)>\frac{1}{2}\dim\mu-\epsilon\right\} d\lambda_{n}(i)\\
 & = & \int\lambda_{n}\left\{ i\in\mathcal{N}_{n}:H\left(\omega|_{i}\right)>\frac{1}{2}\dim\mu-\epsilon\right\} d\beta(\omega).
\end{eqnarray*}
Hence, since $\epsilon^{-1},k\ll n$,
\begin{equation}
P\ge\int\lambda_{n}\left\{ i\in\mathcal{N}_{n}\::\:H\left(\omega|_{i+k}\right)>\frac{1}{2}\dim\mu-\epsilon\right\} \:d\beta(\omega)-\epsilon.\label{eq:first lb on P}
\end{equation}

Be Lemma \ref{lem:linearization in non sat argument}, for all $(i,\omega)\in\mathcal{N}_{n}\times\Lambda^{\mathbb{N}}$
we have
\begin{equation}
H\left(\omega|_{i+k}\right)\ge\frac{1}{m}H\left(S_{(\sigma^{i}\omega)|_{k}}\pi_{\ell(i,\omega)}\varphi_{(\sigma^{i}\omega)|_{k}}\mu,\mathcal{D}_{m}\right)-\epsilon/2,\label{eq:by linearization non sat}
\end{equation}
where
\[
\ell(i,\omega):=\frac{z}{\left(\psi\circ\varphi_{u}\right)'\left(\Pi\omega\right)}\mathbb{R}\alpha_{i}(\omega)^{-1}\in\mathbb{RP}^{1}.
\]

By Lemma \ref{lem:ent > 1/2 dim mu outside a small interval}, for
each $v\in\Lambda^{k}$ there exists $w_{v}\mathbb{R}\in\mathbb{RP}^{1}$
such that
\begin{equation}
\frac{1}{m}H\left(S_{v}\pi_{w\mathbb{R}}\varphi_{v}\mu,\mathcal{D}_{m}\right)>\frac{1}{2}\dim\mu-\epsilon/2\text{ for all }w\mathbb{R}\in\mathbb{RP}^{1}\setminus B\left(w_{v}\mathbb{R},\epsilon\right).\label{eq:>Delta-=00005Cepsilon/2 ouside of interval}
\end{equation}

Let $h:\Lambda^{\mathbb{N}}\rightarrow\mathbb{RP}^{1}$ be defined
by $h(\omega)=w_{\omega|_{k}}\mathbb{R}$ for $\omega\in\Lambda^{\mathbb{N}}$.
Note that since $\epsilon^{-1},m,k\ll n$, we may assume that $n$
is large with respect to $h$. From (\ref{eq:by linearization non sat})
and (\ref{eq:>Delta-=00005Cepsilon/2 ouside of interval}),
\[
H\left(\omega|_{i+k}\right)>\frac{1}{2}\dim\mu-\epsilon\text{ for all }(i,\omega)\in\mathcal{N}_{n}\times\Lambda^{\mathbb{N}}\text{ such that }\ell(i,\omega)\notin B\left(h(\sigma^{i}\omega),\epsilon\right).
\]
Hence, by (\ref{eq:first lb on P}) and the definition of the lines
$\ell(i,\omega)$,
\begin{eqnarray*}
P & \ge & \int\lambda_{n}\left\{ i\in\mathcal{N}_{n}\::\:\ell(i,\omega)\notin B\left(h(\sigma^{i}\omega),\epsilon\right)\right\} \:d\beta(\omega)-\epsilon\\
 & = & \int\lambda_{n}\left\{ i\in\mathcal{N}_{n}\::\:d\left(\frac{z}{\left(\psi\circ\varphi_{u}\right)'\left(\Pi\omega\right)}\mathbb{R},\alpha_{i}(\omega)h(\sigma^{i}\omega)\right)>\epsilon\right\} \:d\beta(\omega)-\epsilon.
\end{eqnarray*}
From this, by Corollary \ref{cor:from equidist prop}, since $n$
is large with respect to $h$, and since $\epsilon<\delta/3$, it
follows that $P\ge\delta/2$, which completes the proof of the proposition.
\end{proof}

\subsection{\label{subsec:Proof-of-Propositions}Proof of Propositions \ref{prop:lb on ent of proj of comp of nu}
and \ref{prop:lb on ent of proj of cylinders of nu}}

First, we prove Proposition \ref{prop:lb on ent of proj of cylinders of nu},
which we restate here for convenience.
\begin{prop*}
Suppose that $\dim\mu<2$. Then there exists $0<\gamma<1$ such that
for every $\epsilon>0$, $n\ge N(\epsilon)\ge1$, $z\mathbb{R}\in\mathbb{RP}^{1}$,
$u\in\Lambda^{*}$, and $\psi\in\mathcal{F}_{\epsilon}$,
\[
\frac{1}{n}H\left(S_{u}\pi_{z\mathbb{R}}\psi\varphi_{u}\mu,\mathcal{D}_{n}\right)\ge\dim\mu-1+\gamma.
\]
\end{prop*}
\begin{proof}
Let $0<\delta<1$ be as obtained in Proposition \ref{prop:nontriv portion > dim(mu)/2},
and let $0<\epsilon<1$ and $l,m,n\in\mathbb{Z}_{>0}$ be with
\[
\delta^{-1}\ll l\ll\epsilon^{-1}\ll m\ll n.
\]
Let $M=M(l)\in\mathbb{Z}_{>0}$ be as obtained in Lemma \ref{lem:ac of words}.
Since $l\ll\epsilon^{-1}$, we may assume that $M\ll\epsilon^{-1}$.
Fix $z\mathbb{R}\in\mathbb{RP}^{1}$, $u\in\Lambda^{*}$ and $\psi\in\mathcal{F}_{\epsilon}$,
and set
\[
H:=\frac{1}{n}H\left(S_{u}\pi_{z\mathbb{R}}\psi\varphi_{u}\mu,\mathcal{D}_{n}\right).
\]

Set $n':=\left\lfloor n/M\right\rfloor $, and let $\mathcal{V}$
denote the set of $v\in\Lambda^{*}$ such that
\[
\frac{1}{m}H\left(S_{uv}\pi_{z\mathbb{R}}\psi\varphi_{uv}\mu,\mathcal{D}_{m}\right)>\frac{1}{2}\dim\mu-\epsilon.
\]
By Proposition \ref{prop:nontriv portion > dim(mu)/2} and since $\delta^{-1},\epsilon^{-1},l,M,m\ll n$,
\[
\mathbb{P}_{l\le i\le n'l+l-1}\left\{ \mathbf{U}_{i}\in\mathcal{V}\right\} >\delta/2.
\]
Hence, there exists $0\le j<l$ such that
\begin{equation}
\mathbb{P}_{1\le i\le n'}\left\{ \mathbf{U}_{j+li}\in\mathcal{V}\right\} >\delta/2.\label{eq:P=00007BU_=00007Bj+li=00007D in U_1=00007D>delta/2}
\end{equation}

Setting
\[
\rho:=\mathbb{P}_{1\le i\le n}\left\{ \mathbf{I}\left(j,l;i\right)\in\mathcal{V}\right\} ,
\]
it follows from (\ref{eq:P=00007BU_=00007Bj+li=00007D in U_1=00007D>delta/2}),
by Lemma \ref{lem:ac of words}, and since $M\ll n$, that $\rho>\delta/(4M)$.

By Lemma \ref{lem:Lip prop of varphi_u},
\begin{equation}
\mathrm{diam}\left(\mathrm{supp}\left(S_{v}\pi_{z\mathbb{R}}\psi\varphi_{v}\mu\right)\right)=O\left(1/\epsilon\right)\text{ for all }v\in\Lambda^{*}.\label{eq:control of diam of supp all v}
\end{equation}
From this, by \cite[Lemma 3.4]{Ho1}, and since $\epsilon^{-1},m\ll n$,
\[
H\ge\mathbb{E}_{1\le i\le n}\left(\frac{1}{m}H\left(S_{u}\pi_{z\mathbb{R}}\psi\varphi_{u}\mu,\mathcal{D}_{i+m}\mid\mathcal{D}_{i}\right)\right)-\epsilon.
\]
By (\ref{eq:mu as conv comb with u in Psi_n}), we have $\varphi_{u}\mu=\mathbb{E}\left(\varphi_{u\mathbf{I}\left(j,l;i\right)}\mu\right)$
for each $i\ge1$. Hence, from the last formula, by the concavity
of conditional entropy, from (\ref{eq:der at 0 is comp to 2^-n for u in Psi_n}),
and since $l,\epsilon^{-1}\ll m$, 
\[
H\ge\mathbb{E}_{1\le i\le n}\left(\frac{1}{m}H\left(S_{u}S_{\mathbf{I}\left(j,l;i\right)}\pi_{z\mathbb{R}}\psi\varphi_{u\mathbf{I}\left(j,l;i\right)}\mu,\mathcal{D}_{m}\mid\mathcal{D}_{0}\right)\right)-2\epsilon.
\]
 From this, by bounded distortion, from (\ref{eq:control of diam of supp all v}),
and since $\epsilon^{-1}\ll m$, 
\[
H\ge\mathbb{E}_{1\le i\le n}\left(\frac{1}{m}H\left(S_{u\mathbf{I}\left(j,l;i\right)}\pi_{z\mathbb{R}}\psi\varphi_{u\mathbf{I}\left(j,l;i\right)}\mu,\mathcal{D}_{m}\right)\right)-3\epsilon.
\]

Now, from the last formula, by Lemma \ref{lem:triv lb}, and since
$\rho>\delta/(4M)$,
\begin{eqnarray*}
H & \ge & \rho\left(\frac{1}{2}\dim\mu-\epsilon\right)+\left(1-\rho\right)\left(\dim\mu-1-\epsilon\right)-3\epsilon\\
 & \ge & \dim\mu-1+\frac{\delta}{4M}\left(1-\frac{1}{2}\dim\mu\right)-4\epsilon.
\end{eqnarray*}
Since $\dim\mu<2$ and $\delta^{-1},M\ll\epsilon^{-1}$, this completes
the proof of the proposition.
\end{proof}
We can now prove Proposition \ref{prop:lb on ent of proj of comp of nu},
which is the following statement.
\begin{prop*}
Suppose that $\dim\mu<2$. Then there exists $\gamma>0$ such that
for every $\epsilon>0$, $m\ge M(\epsilon)\ge1$, $n\ge1$, and $\psi\in\mathcal{F}_{\epsilon}$,
if we set $\theta:=\psi\mu$, then
\[
\mathbb{P}\left\{ \underset{w\mathbb{R}\in\mathbb{RP}^{1}}{\inf}\frac{1}{m}H\left(\pi_{w\mathbb{R}}\theta_{z,n},\mathcal{D}_{n+m}\right)>\dim\mu-1+\gamma\right\} >1-\epsilon.
\]
\end{prop*}
\begin{proof}
Let $0<\gamma<1$ be as obtained in Proposition \ref{prop:lb on ent of proj of cylinders of nu},
let $0<\epsilon,\delta<1$ and $k,m,n\in\mathbb{Z}_{>0}$ be with
\[
\gamma^{-1}\ll\epsilon^{-1}\ll\delta^{-1}\ll k\ll m,
\]
and fix $\psi\in\mathcal{F}_{\epsilon}$.

Set $B:=\cup_{D\in\mathcal{D}_{n}^{\mathbb{C}}}\left(\partial D\right)^{\left(2^{-n}\delta\right)}$.
By Proposition \ref{prop:mu-mass of neigh of dyd cubes} and since
$\epsilon^{-1}\ll\delta^{-1}$, we may assume that $\psi\mu(B)<\epsilon^{2}$.
Let $\mathcal{U}$ be the set of words $u\in\Psi_{n+k}$ such that
$\psi\circ\varphi_{u}(\Omega)\subset D$ for some $D\in\mathcal{D}_{n}^{\mathbb{C}}$.
By Lemma \ref{lem:Lip prop of varphi_u} and (\ref{eq:der at 0 is comp to 2^-n for u in Psi_n}),
\[
\mathrm{diam}\left(\psi\circ\varphi_{u}(\Omega)\right)=O\left(\epsilon^{-1}2^{-n-k}\right)\text{ for }u\in\Psi_{n+k}.
\]
From this and since $\epsilon^{-1},\delta^{-1}\ll k$, it follows
that $\psi\circ\varphi_{u}(\Omega)\subset B$ for $u\in\Psi_{n+k}\setminus\mathcal{U}$.
Thus, by (\ref{eq:mu as conv comb with u in Psi_n}),
\[
\sum_{u\in\Psi_{n+k}\setminus\mathcal{U}}p_{u}\le\sum_{u\in\Psi_{n+k}}p_{u}\cdot\psi\varphi_{u}\mu(B)=\psi\mu(B)<\epsilon^{2},
\]
which implies,
\begin{equation}
\epsilon^{2}>\sum_{D\in\mathcal{D}_{n}^{\mathbb{C}}}\psi\mu(D)\sum_{u\in\Psi_{n+k}\setminus\mathcal{U}}p_{u}\cdot\frac{\psi\varphi_{u}\mu(D)}{\psi\mu(D)}.\label{eq:eta^2>}
\end{equation}

Let $\mathcal{E}$ be the set of $D\in\mathcal{D}_{n}^{\mathbb{C}}$
such that $\psi\mu(D)>0$ and
\[
\frac{1}{\psi\mu(D)}\sum_{u\in\Psi_{n+k}\setminus\mathcal{U}}p_{u}\cdot\psi\varphi_{u}\mu(D)<\epsilon.
\]
By (\ref{eq:eta^2>}), we have $\psi\mu\left(\cup\mathcal{E}\right)>1-\epsilon$.

Fix $D\in\mathcal{E}$ and let $\mathcal{U}_{D}$ be the set of $u\in\Psi_{n+k}$
such that $\psi\circ\varphi_{u}(\Omega)\subset D$. By (\ref{eq:mu as conv comb with u in Psi_n}),
\begin{equation}
\left(\psi\mu\right)_{D}=\frac{1}{\psi\mu(D)}\sum_{u\in\Psi_{n+k}}p_{u}\cdot\psi\varphi_{u}\mu(D)\cdot\left(\psi\varphi_{u}\mu\right)_{D}.\label{eq:(psi mu)_D =00003D}
\end{equation}
Additionally, from $D\in\mathcal{E}$ and since $\psi\varphi_{u}\mu(D)=0$
for $u\in\mathcal{U}\setminus\mathcal{U}_{D}$,
\begin{equation}
\frac{1}{\psi\mu(D)}\sum_{u\in\mathcal{U}_{D}}p_{u}\cdot\psi\varphi_{u}\mu(D)>1-\epsilon.\label{eq:weights in U_D >}
\end{equation}

Fix $w\mathbb{R}\in\mathbb{RP}^{1}$. Given $u\in\mathcal{U}_{D}$,
we have $\left(\psi\varphi_{u}\mu\right)_{D}=\psi\varphi_{u}\mu$.
Hence, from $u\in\Psi_{n+k}$, (\ref{eq:der at 0 is comp to 2^-n for u in Psi_n}),
and $\gamma^{-1},k\ll m$,
\[
\frac{1}{m}H\left(\pi_{w\mathbb{R}}\left(\psi\varphi_{u}\mu\right)_{D},\mathcal{D}_{n+m}\right)\ge\frac{1}{m}H\left(S_{u}\pi_{w\mathbb{R}}\psi\varphi_{u}\mu,\mathcal{D}_{m}\right)-\gamma/2.
\]
From this and by Proposition \ref{prop:lb on ent of proj of cylinders of nu},
\[
\frac{1}{m}H\left(\pi_{w\mathbb{R}}\left(\psi\varphi_{u}\mu\right)_{D},\mathcal{D}_{n+m}\right)\ge\dim\mu-1+\gamma/2\;\text{ for all }u\in\mathcal{U}_{D}.
\]
Thus, by concavity of entropy, (\ref{eq:(psi mu)_D =00003D}), (\ref{eq:weights in U_D >}),
and $\gamma^{-1}\ll\epsilon^{-1}$,
\begin{eqnarray*}
\frac{1}{m}H\left(\pi_{w\mathbb{R}}\left(\psi\mu\right)_{D},\mathcal{D}_{n+m}\right) & \ge & (1-\epsilon)\left(\dim\mu-1+\gamma/2\right)\\
 & > & \dim\mu-1+\gamma/4,
\end{eqnarray*}
which holds for all $D\in\mathcal{E}$ and $w\mathbb{R}\in\mathbb{RP}^{1}$.
Since $\psi\mu\left(\cup\mathcal{E}\right)>1-\epsilon$, this completes
the proof of the proposition.
\end{proof}

\section{\label{sec:An-entropy-increase-result}An entropy increase result}

The purpose of this section is to prove the following theorem. The
proof resembles that of \cite[Theorem 3.1]{Rap_analytic_on_R}, except
that here we must take into account the phenomenon of saturation,
which does not occur on the real line and was addressed in the previous
section.
\begin{thm}
\label{thm:ent inc with poly}Suppose that $\dim\mu<2$. Then for
each $k\in\mathbb{Z}_{>0}$ and $0<\epsilon<1$ there exists $\rho=\rho(k,\epsilon)>0$
such that the following holds for all $n\ge N(k,\epsilon)\ge1$. Let
$\nu\in\mathcal{M}(\mathcal{P}_{k})$ be such that $\Vert q\Vert_{2}\le\epsilon^{-1}$
and $\left|q'(z)\right|\ge\epsilon$ for all $q\in\mathrm{supp}(\nu)$
and $z\in\mathbb{D}$. Suppose also that $\frac{1}{n}H\left(\nu,\mathcal{D}_{n}\right)\ge\epsilon$,
and let $\psi\in\mathcal{F}_{\epsilon}$ be with $\psi\left(\Omega\right)\subset\mathbb{D}$.
Then $\frac{1}{n}H\left(\nu.\left(\psi\mu\right),\mathcal{D}_{n}\right)\ge\dim\mu+\rho$.
\end{thm}

\subsection{\label{subsec:Entropy-growth-under conv in C}Entropy growth under
convolution in $\mathbb{C}$}

The proof of Theorem \ref{thm:ent inc with poly} relies on the following
theorem, which follows directly from \cite[Theorem 2.8]{Ho}; a complete
derivation is given in \cite[Section 5.1]{Rap_Ren_Fur_on_CP1}. For
$\theta,\xi\in\mathcal{M}(\mathbb{C})$, we write $\theta*\xi$ for
their convolution, that is, the push-forward of $\theta\times\xi$
under the map $(x,y)\mapsto x+y$.
\begin{thm}
\label{thm:ent inc in C}For every $0<\epsilon<1$, $m\ge1$ and $0<\eta<\eta(\epsilon)$,
there exists $\delta=\delta(\epsilon,m,\eta)>0$, such that for all
$n\ge N(\epsilon,m,\eta)\ge1$ the following holds. Let $i\in\mathbb{Z}_{>0}$
and $\xi,\theta\in\mathcal{M}\left(\mathbb{C}\right)$ be such that
\[
\mathrm{diam}(\mathrm{supp}(\xi)),\mathrm{diam}(\mathrm{supp}(\theta))\le\epsilon^{-1}2^{-i},
\]
\[
\mathbb{P}_{i\le j\le i+n}\left\{ \frac{1}{m}H\left(\theta_{z,j},\mathcal{D}_{j+m}\right)<2-\epsilon\right\} >1-\eta,
\]
\[
\mathbb{P}_{i\le j\le i+n}\left\{ \underset{w\mathbb{R}\in\mathbb{RP}^{1}}{\inf}\frac{1}{m}H\left(\pi_{w\mathbb{R}}\theta_{z,j},\mathcal{D}_{j+m}\right)>\frac{1}{m}H\left(\theta_{z,j},\mathcal{D}_{j+m}\right)-1+\epsilon\right\} >1-\eta,
\]
and
\[
\frac{1}{n}H\left(\xi,\mathcal{D}_{i+n}\right)>\epsilon.
\]
Then,
\[
\frac{1}{n}H\left(\xi*\theta,\mathcal{D}_{i+n}\right)\ge\frac{1}{n}H\left(\theta,\mathcal{D}_{i+n}\right)+\delta.
\]
\end{thm}

\subsection{Uniform entropy dimension}

By Lemma \ref{lem:ent wrt mod cont dyad part} and (\ref{eq:der at 0 is comp to 2^-n for u in Psi_n}),
for every $\epsilon>0$, $m\ge M(\epsilon)\ge1$, $n\ge1$, $u\in\Psi_{n}$,
and $\psi\in\mathcal{F}_{\epsilon}$, we have
\[
\frac{1}{m}H\left(\psi\varphi_{u}\mu,\mathcal{D}_{n+m}\right)\ge\dim\mu-\epsilon.
\]
The following proposition, whose proof we omit, follows from this
and from Proposition \ref{prop:mu-mass of neigh of dyd cubes} by
an argument similar to that in the proof of \cite[Proposition 3.6]{Rap_analytic_on_R}.
\begin{prop}
\label{prop:bef uni ent dim prop}For each $0<\epsilon<1$, $m\ge M(\epsilon)\ge1$,
$n\ge1$, and $\psi\in\mathcal{F}_{\epsilon}$, if we set $\theta:=\psi\mu$,
then
\[
\mathbb{P}\left\{ \frac{1}{m}H\left(\theta_{z,n},\mathcal{D}_{n+m}\right)>\dim\mu-\epsilon\right\} >1-\epsilon.
\]
\end{prop}

Roughly speaking, and in the terminology of \cite{Ho1}, the following
proposition states that holomorphic images of $\mu$ have uniform
entropy dimension $\dim\mu$. It follows easily from Proposition \ref{prop:bef uni ent dim prop}
by an argument identical to the one given in the proof of \cite[Proposition 3.3]{Rap_analytic_on_R}.
We therefore omit the proof.
\begin{prop}
\label{prop:uni ent dim}For each $0<\epsilon<1$, $m\ge M(\epsilon)\ge1$,
$n\ge N(\epsilon,m)$, and $\psi\in\mathcal{F}_{\epsilon}$, if we
set $\theta:=\psi\mu$, then
\[
\mathbb{P}_{1\le i\le n}\left\{ \left|\frac{1}{m}H\left(\theta_{z,i},\mathcal{D}_{i+m}\right)-\dim\mu\right|<\epsilon\right\} >1-\epsilon.
\]
\end{prop}

\subsection{Additional preparations for the proof of Theorem \ref{thm:ent inc with poly}}

The proof of the following lemma is similar to that of \cite[Lemma 6.9]{HR}
and is therefore omitted.
\begin{lem}
\label{lem:step1 in ent inc pf}Let $k\in\mathbb{Z}_{>0}$, $R>1$,
$\nu\in\mathcal{M}(\mathcal{P}_{k})$ and $\theta\in\mathcal{M}(\mathbb{D})$
be given. Suppose that $\Vert q\Vert_{2}\le R$ for $q\in\mathrm{supp}(\nu)$.
Then for every $n\ge m\ge1,$
\[
\frac{1}{n}H\left(\nu.\theta,\mathcal{D}_{n}\right)\ge\mathbb{E}_{1\le i\le n}\left(\frac{1}{m}H\left(\nu_{q,i}.\theta_{z,i},\mathcal{D}_{i+m}\right)\right)-O_{k,R}\left(\frac{1}{m}+\frac{m}{n}\right).
\]
\end{lem}

Writing $F$ for the map taking $(q,z)\in\mathcal{P}_{k}\times\mathbb{D}$
to $q(z)$, the (complex) differential of $F$ at a point $(q,z)$
is given by $dF_{(q,z)}(r,w)=r(z)+q'(z)w$. Using this, the proof
of the following linearization lemma is almost identical to that of
\cite[Lemma 4.2]{BHR}, so we omit it.
\begin{lem}
\label{lem:linearization}For every $\epsilon>0$, $k\in\mathbb{Z}_{>0}$,
$R>1$, $m\ge M(\epsilon)\ge1$, and $0<\delta<\delta(\epsilon,k,R,m)$,
the following holds. Let $q\in\mathcal{P}_{k}$, $z\in\mathbb{D}$,
$\nu\in\mathcal{M}(\mathcal{P}_{k})$ and $\theta\in\mathcal{M}(\mathbb{D})$
be such that $\Vert q\Vert_{2}\le R$, $\Vert r-q\Vert_{2}\le\delta$
for $r\in\mathrm{supp}(\nu)$, and $|w-z|\le\delta$ for $w\in\mathrm{supp}(\theta)$.
Then,
\[
\left|\frac{1}{m}H\left(\nu.\theta,\mathcal{D}_{m-\log\delta}\right)-\frac{1}{m}H\left(\left(\nu.z\right)*\left(S_{q'(z)}\theta\right),\mathcal{D}_{m-\log\delta}\right)\right|<\epsilon.
\]
\end{lem}

Finally, the following lemma is obtained by an argument similar to
that used in the proof of \cite[Lemma 3.10]{Rap_analytic_on_R}, and
we again omit the details.
\begin{lem}
\label{lem:ent of nu imp ent of push of comp of nu}For every $0<\epsilon<1$
and $k\in\mathbb{Z}_{>0}$ there exists $\epsilon_{0}=\epsilon_{0}(\epsilon,k)>0$
such that for all $m\ge M(\epsilon,k)\ge1$ and $n\ge N(\epsilon,k,m)$
the following holds. Let $\nu\in\mathcal{M}(\mathcal{P}_{k})$ be
such that $\Vert q\Vert_{2}\le\epsilon^{-1}$ for $q\in\mathrm{supp}(\nu)$
and $\frac{1}{n}H\left(\nu,\mathcal{D}_{n}\right)\ge\epsilon$. Let
$\psi\in\mathcal{F}_{\epsilon}$ be with $\psi\left(\Omega\right)\subset\mathbb{D}$,
and set $\theta:=\psi\mu$. Then,
\begin{equation}
\int\mathbb{P}_{1\le i\le n}\left\{ \frac{1}{m}H\left(\left(\nu_{q,i}\right).z,\mathcal{D}_{i+m}\right)>\epsilon_{0}\right\} \:d\theta(z)>\epsilon_{0}.\label{eq:ent of nu.x nontrivial}
\end{equation}
\end{lem}

\subsection{\label{subsec:Proof-of-the ent in result}Proof of the entropy increase
result}
\begin{proof}[Proof of Theorem \ref{thm:ent inc with poly}]
Suppose that $\dim\mu<2$, let $0<\gamma<1$ be as obtained in Proposition
\ref{prop:lb on ent of proj of comp of nu}, and let $k,l,m,n\in\mathbb{Z}_{>0}$
and $0<\epsilon,\epsilon_{0},\eta,\rho,\delta<1$ be with,
\begin{equation}
\left(2-\dim\mu\right)^{-1},\gamma,k,\epsilon^{-1}\ll\epsilon_{0}^{-1}\ll\eta^{-1}\ll l\ll\rho^{-1}\ll\delta^{-1}\ll m\ll n.\label{eq:rel between params in ent inc result}
\end{equation}
Let $\nu\in\mathcal{M}(\mathcal{P}_{k})$ be such that $\Vert q\Vert_{2}\le\epsilon^{-1}$
and $\left|q'(z)\right|\ge\epsilon$ for $q\in\mathrm{supp}(\nu)$
and $z\in\mathbb{D}$, and such that $\frac{1}{n}H\left(\nu,\mathcal{D}_{n}\right)\ge\epsilon$.
Let $\psi\in\mathcal{F}_{\epsilon}$ be with $\psi(\Omega)\subset\mathbb{D}$,
and set $\theta:=\psi\mu$.

By Lemma \ref{lem:step1 in ent inc pf},
\[
\frac{1}{n}H\left(\nu.\theta,\mathcal{D}_{n}\right)+\delta\ge\mathbb{E}_{1\le i\le n}\left(\frac{1}{m}H\left(\nu_{q,i}.\theta_{z,i},\mathcal{D}_{i+m}\right)\right).
\]
Hence, by Lemma \ref{lem:linearization},
\[
\frac{1}{n}H\left(\nu.\theta,\mathcal{D}_{n}\right)+2\delta\ge\mathbb{E}_{1\le i\le n}\left(\frac{1}{m}H\left(\left(\nu_{q,i}.z\right)*\left(S_{q'(z)}\theta_{z,i}\right),\mathcal{D}_{i+m}\right)\right).
\]
Thus, since $\left|q'(z)\right|\ge\epsilon$ for $q\in\mathrm{supp}(\nu)$
and $z\in\mathbb{D}$,
\begin{equation}
\frac{1}{n}H\left(\nu.\theta,\mathcal{D}_{n}\right)+3\delta\ge\mathbb{E}_{1\le i\le n}\left(\frac{1}{m}H\left(\left(S_{q'(z)}^{-1}\nu_{q,i}\right).z*\theta_{z,i},\mathcal{D}_{i+m}\right)\right).\label{eq:by linearization ect}
\end{equation}

Write $\Gamma:=\lambda_{n}\times\theta\times\nu$, where $\lambda_{n}$
is defined in Section \ref{subsec:Basic-notations}. Let $E_{1}$
be the set of all $(i,z,q)\in\mathcal{N}_{n}\times\mathbb{D}\times\mathrm{supp}(\nu)$
such that $\frac{1}{m}H\left(\theta_{z,i},\mathcal{D}_{i+m}\right)\ge\dim\mu-\delta$.
By Proposition \ref{prop:bef uni ent dim prop}, we may assume that
$\Gamma(E_{1})>1-\delta$. Also, by \cite[Corollary 4.10]{Ho1},
\begin{equation}
\frac{1}{m}H\left(\left(S_{q'(z)}^{-1}\nu_{q,i}\right).z*\theta_{z,i},\mathcal{D}_{i+m}\right)\ge\dim\mu-2\delta\text{ for }(i,z,q)\in E_{1}.\label{eq:lb on E_1}
\end{equation}

Let $E_{2}$ be the set of all $(i,z,q)\in E_{1}$ such that
\begin{equation}
\mathbb{P}_{i\le j\le i+m}\left\{ \frac{1}{l}H\left(\left(\theta_{z,i}\right)_{w,j},\mathcal{D}_{j+l}\right)<1+\frac{1}{2}\dim\mu\right\} >1-\eta,\label{eq:first prop of E_2}
\end{equation}
\begin{equation}
\mathbb{P}_{i\le j\le i+m}\left\{ \begin{array}{c}
\underset{u\mathbb{R}\in\mathbb{RP}^{1}}{\inf}\frac{1}{l}H\left(\pi_{u\mathbb{R}}\left(\theta_{z,i}\right)_{w,j},\mathcal{D}_{j+l}\right)\\
>\frac{1}{l}H\left(\left(\theta_{z,i}\right)_{w,j},\mathcal{D}_{j+l}\right)-1+\gamma/2
\end{array}\right\} >1-\eta,\label{eq:proj prop of E_2}
\end{equation}
and
\begin{equation}
\frac{1}{m}H\left(\left(\nu_{q,i}\right).z,\mathcal{D}_{i+m}\right)>\epsilon_{0}.\label{eq:second prop of E_2}
\end{equation}
By Propositions \ref{prop:lb on ent of proj of comp of nu} and \ref{prop:uni ent dim},
Lemma \ref{lem:ent of nu imp ent of push of comp of nu}, \cite[Lemma 2.7]{Ho},
and since $\dim\mu<2$, we may assume that $\Gamma(E_{2})>\epsilon_{0}$.

Let $(i,z,q)\in E_{2}$ and set $\xi:=\left(S_{q'(z)}^{-1}\nu_{q,i}\right).z$.
We next want to apply Theorem \ref{thm:ent inc in C} in order to
obtain entropy increase for the convolution $\xi*\theta_{z,i}$. Let
$q_{1},q_{2}\in\mathrm{supp}(\nu_{q,i})$ be given. Since $z\in\mathbb{D}$,
$\left|q'(z)\right|\ge\epsilon$, and $q_{1}$ and $q_{2}$ belong
to the same atom of $\mathcal{D}_{i}^{\mathcal{P}_{k}}$,
\[
\left|q'(z)\right|^{-1}\cdot\left|q_{1}(z)-q_{2}(z)\right|\le\epsilon^{-1}(k+1)2^{1-i}.
\]
From this and since $\theta_{z,i}$ is supported on a single atom
of $\mathcal{D}_{i}^{\mathbb{C}}$,
\begin{equation}
\mathrm{diam}\left(\mathrm{supp}\left(\xi\right)\right),\mathrm{diam}\left(\mathrm{supp}\left(\theta_{z,i}\right)\right)=O_{k,\epsilon}\left(2^{-i}\right).\label{eq:supports are big O}
\end{equation}
We have $\Vert q\Vert_{2}\le\epsilon^{-1}$, which implies $\left|q'(z)\right|=O_{k,\epsilon}(1)$.
Thus, from (\ref{eq:second prop of E_2}) and since $k,\epsilon^{-1},\epsilon_{0}^{-1}\ll m$,
\[
\frac{1}{m}H\left(\xi,\mathcal{D}_{i+m}\right)>\epsilon_{0}/2.
\]
From this, (\ref{eq:rel between params in ent inc result}), (\ref{eq:first prop of E_2}),
(\ref{eq:proj prop of E_2}), (\ref{eq:supports are big O}), and
Theorem \ref{thm:ent inc in C},
\[
\frac{1}{m}H\left(\xi*\theta_{z,i},\mathcal{D}_{i+m}\right)\ge\frac{1}{m}H\left(\theta_{z,i},\mathcal{D}_{i+m}\right)+\rho.
\]
Hence, since $E_{2}\subset E_{1}$, we have thus proven that
\begin{equation}
\frac{1}{m}H\left(\left(S_{q'(z)}^{-1}\nu_{q,i}\right).z*\theta_{z,i},\mathcal{D}_{i+m}\right)\ge\dim\mu+\rho-\delta\text{ for }(i,z,q)\in E_{2}.\label{eq:lb on E_2}
\end{equation}

Now, from (\ref{eq:by linearization ect}), (\ref{eq:lb on E_1})
and (\ref{eq:lb on E_2}),
\[
\frac{1}{n}H\left(\nu.\theta,\mathcal{D}_{n}\right)+3\delta\ge\Gamma\left(E_{1}\setminus E_{2}\right)\left(\dim\mu-2\delta\right)+\Gamma\left(E_{2}\right)\left(\dim\mu+\rho-\delta\right).
\]
Thus, since $\Gamma(E_{1})>1-\delta$ and $\Gamma(E_{2})>\epsilon_{0}$,
\[
\frac{1}{n}H\left(\nu.\theta,\mathcal{D}_{n}\right)\ge\dim\mu+\epsilon_{0}\rho-O(\delta).
\]
By (\ref{eq:rel between params in ent inc result}), this completes
the proof of the theorem.
\end{proof}

\section{\label{sec:Proof-of-the-main-result}Proof of the main result}

In this section we prove our main result, Theorem \ref{thm:main thm}.
The proof is an adaptation of the proof of the main result of \cite{Rap_analytic_on_R}
and largely follows it. Nevertheless, for the reader\textquoteright s
convenience, we shall provide full details of the main argument.

The proof of the following preliminary lemma is almost identical to
that of \cite[Lemma 4.1]{Rap_analytic_on_R} and is therefore omitted.
Recall from Section \ref{subsec:Symbolic-notations} that $\{\beta_{\omega}\}_{\omega\in\Lambda^{\mathbb{N}}}\subset\mathcal{M}\left(\Lambda^{\mathbb{N}}\right)$
denotes the disintegration of $\beta$ with respect to $\Pi^{-1}\mathcal{B}_{\mathbb{C}}$.
Also, recall that for each $n\ge1$, we defined $\Pi_{n}:\Lambda^{\mathbb{N}}\rightarrow\mathcal{O}(\Omega)$
by $\Pi_{n}(\omega)=\varphi_{\omega|_{n}}$ for $\omega\in\Lambda^{\mathbb{N}}$.
\begin{lem}
\label{lem:0-ent lemma}For $\beta$-a.e. $\omega$,
\[
\underset{n\rightarrow\infty}{\lim}\frac{1}{n}H\left(\left(\Pi_{n}\beta_{\omega}\right).\mu,\mathcal{D}_{\chi n}\right)=0.
\]
\end{lem}

We shall also need the following lemma.
\begin{lem}
\label{lem:estimate on r_psi}There exists $C>1$ such that for all
$u\in\Lambda^{*}$ the following holds. Let $\psi:\mathbb{C}\rightarrow\mathbb{C}$
be the unique homothetic map such that $\psi'(0)>0$, $\psi\left(\varphi_{u}(0)\right)=0$,
and
\[
\sup\left\{ \left|\psi\left(\varphi_{u}(z)\right)\right|\::\:z\in\Omega\right\} =1.
\]
Let $r_{\psi}>0$ and $b_{\psi}\in\mathbb{C}$ be with $\psi(z)=r_{\psi}z+b_{\psi}$
for $z\in\mathbb{C}$. Then, $C^{-1}\le r_{\psi}\left|\varphi_{u}'(0)\right|\le C$.
\end{lem}

\begin{proof}
Let $C>1$ be a large global constant depending only on $\Phi$ and
$\Omega$. In particular, since the bounded domain $\Omega$ contains
$0$, we may assume that $C^{-1}\le|z|\le C$ for all $z\in\partial\Omega$.

Let $u\in\Lambda^{*}$, and let $\psi:\mathbb{C}\rightarrow\mathbb{C}$
be as in the statement of the lemma. Since $\overline{\Omega}$ is
compact and by the maximum modulus principle, there exists $w\in\partial\Omega$
such that $\left|\psi\left(\varphi_{u}(w)\right)\right|=1$. We have,
\[
1=\left|\psi\left(\varphi_{u}(w)\right)-\psi\left(\varphi_{u}(0)\right)\right|=r_{\psi}\left|\varphi_{u}(w)-\varphi_{u}(0)\right|.
\]
Moreover, by Lemma \ref{lem:Lip prop of varphi_u},
\[
C^{-1}\left|\varphi_{u}'(0)\right||w|\le\left|\varphi_{u}(w)-\varphi_{u}(0)\right|\le C\left|\varphi_{u}'(0)\right||w|.
\]
Combining all of this, we obtain $C^{-2}\le r_{\psi}\left|\varphi_{u}'(0)\right|\le C^{2}$,
which completes the proof of the lemma.
\end{proof}
We can now begin the proof of our main result.
\begin{proof}[Proof of Theorem \ref{thm:main thm}]
Suppose that the assumptions made in the theorem are all satisfied,
and assume by contradiction that $\dim\mu<\min\left\{ 2,H(p)/\chi\right\} $.
Recall that for $l\ge1$, we denote by $\mathcal{C}_{l}$ the partition
of $\Lambda^{\mathbb{N}}$ into level-$l$ cylinder sets. Set $\Delta':=H\left(\beta,\mathcal{C}_{1}\mid\Pi^{-1}\mathcal{B}_{\mathbb{C}}\right)$,
where the right-hand side denotes the conditional entropy of $\mathcal{C}_{1}$
given $\Pi^{-1}\mathcal{B}_{\mathbb{C}}$ with respect to $\beta$.
By \cite[Theorem 2.8]{FH-dimension}, we have $\dim\mu=\left(H(p)-\Delta'\right)/\chi$.
Thus, from $\dim\mu<H(p)/\chi$, it follows that $\Delta'>0$. Moreover,
by \cite[Proposition 4.10]{FH-dimension},
\begin{equation}
\underset{l\rightarrow\infty}{\lim}\frac{1}{l}H\left(\beta_{\omega},\mathcal{C}_{l}\right)=\Delta'\text{ for }\beta\text{-a.e. }\omega.\label{eq:ent of slices}
\end{equation}

Since $\Phi$ is exponentially separated, there exists $0<c<1$ as
in Definition \ref{def:exp sep}. Set $\Delta:=\min\left\{ \Delta',1\right\} /2$,
and let $C>1$ be a large global constant depending only on $\Phi$,
$\Omega$ and $\Omega_{0}$. Let $0<\rho,\delta<1$ and $M,k,n\in\mathbb{Z}_{>0}$
be such that
\begin{equation}
\Delta^{-1},c^{-1},C\ll M\ll k\ll\rho^{-1}\ll\delta^{-1}\ll n.\label{eq:rel between params in main pf}
\end{equation}
Set $n'=\left\lfloor n\Delta/\left(2\log|\Lambda|\right)\right\rfloor $.
By exponential separation and the choice of $c$, we may clearly assume
that
\begin{equation}
\Vert\varphi_{u_{1}}-\varphi_{u_{2}}\Vert_{\Omega}\ge c^{n+n'}\text{ for all distinct }u_{1},u_{2}\in\Lambda^{n+n'}.\label{eq:by exp sep}
\end{equation}
By (\ref{eq:ent dim =00003D exact dim}) and $\delta^{-1}\ll n$,
we may also assume that
\begin{equation}
\dim\mu+\delta\ge\frac{1}{Mn}H\left(\mu,\mathcal{D}_{Mn+\chi(n+n')}\mid\mathcal{D}_{\chi(n+n')}\right).\label{eq:lb by exact dim}
\end{equation}

By (\ref{eq:def rel of mu}) and since $\beta=\int\beta_{\omega}\:d\beta(\omega)$,
\[
\mu=\left(\Pi_{n+n'}\beta\right).\mu=\int\left(\Pi_{n+n'}\beta_{\omega}\right).\mu\:d\beta(\omega).
\]
Hence, from (\ref{eq:lb by exact dim}) and by the concavity of conditional
entropy,
\[
\dim\mu+\delta\ge\int\frac{1}{Mn}H\left(\left(\Pi_{n+n'}\beta_{\omega}\right).\mu,\mathcal{D}_{Mn+\chi(n+n')}\mid\mathcal{D}_{\chi(n+n')}\right)\:d\beta(\omega).
\]
Together with Lemma \ref{lem:0-ent lemma}, this gives
\[
\dim\mu+2\delta\ge\int\frac{1}{Mn}H\left(\left(\Pi_{n+n'}\beta_{\omega}\right).\mu,\mathcal{D}_{Mn+\chi(n+n')}\right)\:d\beta(\omega).
\]
Thus, setting
\[
g(\omega):=\frac{1}{Mn}H\left(\left(\Pi_{n+n'}\beta_{\omega}\right).\mu,\mathcal{D}_{Mn+\chi(n+n')}\right)\text{ for }\omega\in\Lambda^{\mathbb{N}},
\]
we have
\begin{equation}
\dim\mu+2\delta\ge\int g(\omega)\:d\beta(\omega).\label{eq:lb on dim(mu) by g}
\end{equation}

Recall that for $\mathcal{U}\subset\Lambda^{*}$, we write $\left[\mathcal{U}\right]$
in place of $\cup_{u\in\mathcal{U}}[u]$. For $l\in\mathbb{Z}_{>0}$,
let $\mathcal{U}_{l}$ be the set of words $u\in\Lambda^{l}$ such
that $2^{-l(\chi+\delta)}\le\left|\varphi_{u}'(0)\right|\le2^{-l(\chi-\delta)}$.
Let $E$ be the set of $\omega\in\Lambda^{\mathbb{N}}$ such that
$\mathrm{supp}(\beta_{\omega})\subset\Pi^{-1}\left(\Pi\omega\right)$,
$\frac{1}{n}H\left(\beta_{\omega},\mathcal{C}_{n}\right)>\Delta$,
$\beta_{\omega}\left(\left[\mathcal{U}_{n}\right]\right)>1-\delta$,
and $\sigma^{n}\beta_{\omega}\left(\left[\mathcal{U}_{n'}\right]\right)>1-\delta$.
By (\ref{eq:ent of slices}) and (\ref{eq:chi as a.s.  limit}), by
basic properties of disintegrations, and since $\Delta^{-1},\delta^{-1}\ll n$,
we may assume that $\beta(E)>1-\delta$. In what follows, fix $\omega\in E$.

By the definition of $\mathcal{U}_{n}$, there exists a partition
$\left\{ \mathcal{U}_{n,j}\right\} _{j\in J}$ of $\mathcal{U}_{n}$
such that $|J|\le3\delta n$ and,
\begin{equation}
\left|\varphi_{u_{1}}'(0)\right|\le2\left|\varphi_{u_{2}}'(0)\right|\text{ for all }j\in J\text{ and }u_{1},u_{2}\in\mathcal{U}_{n,j}.\label{eq:def prop of U_n,j}
\end{equation}
For $j\in J$ and $u\in\mathcal{U}_{n'}$, set $F_{j,u}:=\left[\mathcal{U}_{n,j}\right]\cap\sigma^{-n}\left[u\right]$
and
\[
g(\omega,j,u):=\frac{1}{Mn}H\left(\left(\Pi_{n+n'}\left(\beta_{\omega}\right)_{F_{j,u}}\right).\mu,\mathcal{D}_{Mn+\chi(n+n')}\right).
\]

Since $\omega\in E$, there exist $0\le\delta'<2\delta$ and $\theta\in\mathcal{M}\left(\Lambda^{\mathbb{N}}\right)$
such that
\[
\beta_{\omega}=\sum_{(j,u)\in J\times\mathcal{U}_{n'}}\beta_{\omega}\left(F_{j,u}\right)\left(\beta_{\omega}\right)_{F_{j,u}}+\delta'\theta.
\]
Thus, by concavity,
\begin{equation}
g(\omega)\ge\sum_{(j,u)\in J\times\mathcal{U}_{n'}}\beta_{\omega}\left(F_{j,u}\right)g(\omega,j,u).\label{eq:lb on g(omega) by g(omega,j,u)}
\end{equation}
Moreover, from $\frac{1}{n}H\left(\beta_{\omega},\mathcal{C}_{n}\right)>\Delta$,
by the convexity bound (see \cite[Lemma 3.1]{Ho1}), and since the
cardinality of $J\times\mathcal{U}_{n'}$ is at most $3\delta n|\Lambda|^{n'}$,
\begin{multline*}
\Delta<\sum_{(j,u)\in J\times\mathcal{U}_{n'}}\beta_{\omega}\left(F_{j,u}\right)\frac{1}{n}H\left(\left(\beta_{\omega}\right)_{F_{j,u}},\mathcal{C}_{n}\right)\\
+\delta'\frac{1}{n}H\left(\theta,\mathcal{C}_{n}\right)+\frac{\log\left(3\delta n\right)}{n}+\frac{n'}{n}\log|\Lambda|.
\end{multline*}
Note that,
\begin{equation}
\frac{1}{n}H\left(\theta',\mathcal{C}_{n}\right)\le\log|\Lambda|\text{ for all }\theta'\in\mathcal{M}\left(\Lambda^{\mathbb{N}}\right).\label{eq:ub ent all measures}
\end{equation}
Hence, from the previous formula and the definition of $n'$,
\begin{equation}
\Delta/3<\sum_{(j,u)\in J\times\mathcal{U}_{n'}}\beta_{\omega}\left(F_{j,u}\right)\frac{1}{n}H\left(\left(\beta_{\omega}\right)_{F_{j,u}},\mathcal{C}_{n}\right).\label{eq:Delta/3<}
\end{equation}

Let $\mathcal{Q}$ be the set of all $(j,u)\in J\times\mathcal{U}_{n'}$
such that $\frac{1}{n}H\left(\left(\beta_{\omega}\right)_{F_{j,u}},\mathcal{C}_{n}\right)\ge\Delta/6$.
From (\ref{eq:ub ent all measures}) and (\ref{eq:Delta/3<}),
\begin{equation}
\frac{\Delta}{6\log|\Lambda|}<\sum_{(j,u)\in\mathcal{Q}}\beta_{\omega}\left(F_{j,u}\right).\label{eq:lb on mass of (j,u ) in Q}
\end{equation}

Recall that $C>1$ is a large global constant. For $v\in\mathcal{U}_{n}$
and $w\in\mathcal{U}_{n'}$, it follows by bounded distortion, the
chain rule, and the definition of the sets $\mathcal{U}_{l}$, that
\[
C^{-1}2^{-\delta(n+n')}\le2^{\chi(n+n')}\left|\varphi_{vw}'(0)\right|\le C2^{\delta(n+n')}.
\]
Thus, by Lemma \ref{lem:Lip prop of varphi_u}, $S_{2^{\chi(n+n')}}\circ\varphi_{vw}$
is bi-Lipschitz with bi-Lipschitz constant at most $C^{2}2^{\delta(n+n')}$.
Moreover, given $(j,u)\in J\times\mathcal{U}_{n'}$ and setting $\xi:=\left(\Pi_{n+n'}\left(\beta_{\omega}\right)_{F_{j,u}}\right)$,
we have
\[
\mathrm{supp}(\xi)\subset\left\{ \varphi_{vw}\::\:v\in\mathcal{U}_{n}\text{ and }w\in\mathcal{U}_{n'}\right\} .
\]
From these facts, by concavity, from Lemma \ref{lem:dyad ent =000026 lip func},
by (\ref{eq:ent dim =00003D exact dim}), and since $C,\delta^{-1}\ll n$,
\[
g(\omega,j,u)\ge\int\frac{1}{Mn}H\left(S_{2^{\chi(n+n')}}\varphi\mu,\mathcal{D}_{Mn}\right)\:d\xi(\varphi)-O(1/n)\ge\dim\mu-O(\delta).
\]
We have thus shown that,
\begin{equation}
g(\omega,j,u)\ge\dim\mu-O(\delta)\text{ for }(j,u)\in J\times\mathcal{U}_{n'}.\label{eq:simple lb by conc}
\end{equation}

Fix $(j,u)\in\mathcal{Q}$, and set $a:=\varphi_{u}(0)$. Recall from
Section \ref{subsec:Function-spaces} that the linear operator $P_{k,a}:\mathcal{O}(\Omega)\rightarrow\mathcal{P}_{k}$
is defined by sending $f\in\mathcal{O}(\Omega)$ to its $k$-th order
Taylor polynomial at the point $a$. Let $v\in\mathcal{U}_{n,j}$
and $z\in\Omega$ be given. From Lemma \ref{lem:Lip prop of varphi_u}
and since $u\in\mathcal{U}_{n'}$,
\[
\left|\varphi_{u}(z)-a\right|\le C\left|\varphi_{u}'(0)\right|\le C2^{n'(\delta-\chi)}.
\]
From this, since we can assume that $\delta<\chi/2$, by Lemma \ref{lem:Taylor estimate},
and since $v\in\mathcal{U}_{n}$,
\begin{eqnarray*}
\left|\left(\varphi_{v}-P_{k,a}\varphi_{v}\right)\left(\varphi_{u}(z)\right)\right| & \le & C^{k+1}\left|\varphi_{v}'(0)\right|\left|\varphi_{u}(z)-a\right|^{k+1}\\
 & \le & C^{2k+2}2^{(k+1)n'(\delta-\chi)}2^{n(\delta-\chi)}.
\end{eqnarray*}
Hence, from $\Delta^{-1},C,M\ll k\ll n$, by the definition of $n'$,
and since $\delta<\chi/2$,
\begin{equation}
\left|\left(\varphi_{v}-P_{k,a}\varphi_{v}\right)\left(\varphi_{u}(z)\right)\right|\le2^{-Mn-\chi(n+n')}\text{ for }v\in\mathcal{U}_{n,j}\text{ and }z\in\Omega.\label{eq:bd on remainder}
\end{equation}

Note that
\[
\left(\Pi_{n+n'}\left(\beta_{\omega}\right)_{F_{j,u}}\right).\mu=\left(\Pi_{n}\left(\beta_{\omega}\right)_{F_{j,u}}\right).\varphi_{u}\mu,
\]
and also
\begin{equation}
\mathrm{supp}\left(\Pi_{n}\left(\beta_{\omega}\right)_{F_{j,u}}\right)=\left\{ \varphi_{v}\::\:v\in\mathcal{U}_{n,j}\text{ and }\beta_{\omega}\left(\left[vu\right]\right)>0\right\} .\label{eq:form of supp of push by Pi_n}
\end{equation}
Together with (\ref{eq:bd on remainder}) and Lemma \ref{lem:close func --> close ent},
this implies that
\begin{equation}
g(\omega,j,u)\ge\frac{1}{Mn}H\left(\left(P_{k,a}\Pi_{n}\left(\beta_{\omega}\right)_{F_{j,u}}\right).\varphi_{u}\mu,\mathcal{D}_{Mn+\chi(n+n')}\right)-\delta.\label{eq:lb on g by taylor}
\end{equation}

Let $\psi:\mathbb{C}\rightarrow\mathbb{C}$ be the unique homothetic
map such that $\psi'(0)>0$, $\psi(a)=0$, and
\[
\sup\left\{ \left|\psi\left(\varphi_{u}(z)\right)\right|\::\:z\in\Omega\right\} =1.
\]
Let $r_{\psi}>0$ and $b_{\psi}\in\mathbb{C}$ be with $\psi(z)=r_{\psi}z+b_{\psi}$
for $z\in\mathbb{C}$. By Lemma \ref{lem:estimate on r_psi} and since
$u\in\mathcal{U}_{n'}$,
\begin{equation}
C^{-1}2^{n'(\chi-\delta)}\le C^{-1}\left|\varphi_{u}'(0)\right|^{-1}\le r_{\psi}\le C\left|\varphi_{u}'(0)\right|^{-1}\le C2^{n'(\chi+\delta)}.\label{eq:bounds on r_psi}
\end{equation}
Let $R_{\psi^{-1}}:\mathcal{P}_{k}\rightarrow\mathcal{P}_{k}$ be
with $R_{\psi^{-1}}(q):=q\circ\psi^{-1}$ for $q\in\mathcal{P}_{k}$.

By (\ref{eq:def prop of U_n,j}) and $\mathcal{U}_{n,j}\subset\mathcal{U}_{n}$,
there exists $2^{n(\chi-\delta)}\le\alpha\le2^{n(\chi+\delta)}$ such
that
\begin{equation}
\frac{1}{2}\le\alpha\left|\varphi_{v}'(0)\right|\le2\text{ for }v\in\mathcal{U}_{n,j}.\label{eq:bounds on alpha}
\end{equation}
Set
\[
\nu:=S_{\alpha r_{\psi}}T_{-\Pi(\omega)}R_{\psi^{-1}}P_{k,a}\Pi_{n}\left(\beta_{\omega}\right)_{F_{j,u}},
\]
where recall from Section \ref{subsec:Basic-notations} that for $q(X)\in\mathcal{P}_{k}$
\[
T_{-\Pi(\omega)}\left(q(X)\right)=q(X)-\Pi(\omega)\:\text{ and }\:S_{\alpha r_{\psi}}\left(q(X)\right)=\alpha r_{\psi}\cdot q(X).
\]
Note that,
\[
\nu.\left(\psi\varphi_{u}\mu\right)=S_{\alpha r_{\psi}}T_{-\Pi(\omega)}\left(\left(P_{k,a}\Pi_{n}\left(\beta_{\omega}\right)_{F_{j,u}}\right).\varphi_{u}\mu\right).
\]
Together with (\ref{eq:lb on g by taylor}), $r_{\psi}\le C2^{n'(\chi+\delta)}$
and $\alpha\le2^{n(\chi+\delta)}$, this gives
\begin{equation}
g(\omega,j,u)\ge\frac{1}{Mn}H\left(\nu.\left(\psi\varphi_{u}\mu\right),\mathcal{D}_{Mn}\right)-O(\delta).\label{eq:lb on g(omega,j,u) bef ent enc}
\end{equation}

Next, we aim to apply Theorem \ref{thm:ent inc with poly} to the
entropy appearing on the right-hand side of the last inequality. We
proceed to verify the conditions of the theorem. Let $v\in\mathcal{U}_{n,j}$
be with $\beta_{\omega}\left(\left[vu\right]\right)>0$, and set $q:=S_{\alpha r_{\psi}}T_{-\Pi(\omega)}R_{\psi^{-1}}P_{k,a}\varphi_{v}$.
Note that by (\ref{eq:form of supp of push by Pi_n}), each polynomial
in $\mathrm{supp}(\nu)$ is of this form. We have
\[
\psi^{-1}(X)-a=r_{\psi}^{-1}X-b_{\psi}r_{\psi}^{-1}-a=r_{\psi}^{-1}\left(X-\psi(a)\right)=r_{\psi}^{-1}X,
\]
which, together with (\ref{eq:exp form of P_k,a}), gives
\begin{eqnarray}
q(X)+\alpha r_{\psi}\Pi(\omega) & = & \alpha r_{\psi}\sum_{l=0}^{k}\frac{\varphi_{v}^{(l)}\left(a\right)}{l!}\left(\psi^{-1}(X)-a\right)^{l}\nonumber \\
 & = & \alpha r_{\psi}\sum_{l=0}^{k}\frac{\varphi_{v}^{(l)}\left(a\right)}{l!}r_{\psi}^{-l}X^{l}.\label{eq:dev of p}
\end{eqnarray}
Thus,
\begin{equation}
\Vert q(X)\Vert_{2}^{2}=\alpha^{2}r_{\psi}^{2}\left|\varphi_{v}\left(a\right)-\Pi(\omega)\right|^{2}+\sum_{l=1}^{k}\left|\alpha r_{\psi}\frac{\varphi_{v}^{(l)}\left(a\right)}{l!}r_{\psi}^{-l}\right|^{2}.\label{eq:norm of p equals}
\end{equation}

By Lemma \ref{lem:ub on high derivs}, since $v\in\mathcal{U}_{n,j}$,
and from (\ref{eq:bounds on alpha}),
\begin{equation}
\left|\varphi_{v}^{(l)}\left(a\right)\right|\le l!C^{l}\left|\varphi_{v}'(0)\right|\le l!C^{l}2\alpha^{-1}\text{ for }1\le l\le k.\label{eq:ub on derivatives}
\end{equation}
Since $\omega\in E$, we have $\mathrm{supp}(\beta_{\omega})\subset\Pi^{-1}\left(\Pi\omega\right)$.
From this and $\beta_{\omega}\left(\left[vu\right]\right)>0$, it
follows that there exists $\omega'\in\left[vu\right]$ such that $\Pi(\omega)=\Pi(\omega')=\varphi_{vu}\Pi(\sigma^{n+n'}(\omega'))$.
Additionally, from $a=\varphi_{u}(0)$, we get $\varphi_{v}\left(a\right)=\varphi_{vu}\left(0\right)$.
Hence, by Lemma \ref{lem:Lip prop of varphi_u}, by bounded distortion,
since $\Pi(\sigma^{n+n'}(\omega'))\in\Omega$, and from (\ref{eq:bounds on r_psi})
and (\ref{eq:bounds on alpha}),
\begin{multline*}
\left|\varphi_{v}\left(a\right)-\Pi(\omega)\right|=\left|\varphi_{vu}\left(0\right)-\varphi_{vu}\Pi(\sigma^{n+n'}(\omega'))\right|\\
\le C\left|\varphi_{v}'(0)\right|\left|\varphi_{u}'(0)\right|\le2C^{2}\alpha^{-1}r_{\psi}^{-1}.
\end{multline*}
From this, (\ref{eq:norm of p equals}) and (\ref{eq:ub on derivatives}),
\[
\Vert q(X)\Vert_{2}^{2}\le4C^{4}+\sum_{l=1}^{k}4C^{2l}r_{\psi}^{2-2l}\le8kC^{2k}.
\]

Let $z\in\mathbb{D}$ be given. From (\ref{eq:dev of p}),
\[
q'(z)=\alpha r_{\psi}\sum_{l=1}^{k}\frac{\varphi_{v}^{(l)}\left(a\right)}{(l-1)!}r_{\psi}^{-l}z^{l-1}.
\]
From (\ref{eq:bounds on r_psi}) and (\ref{eq:ub on derivatives}),
it follows that for $2\le l\le k$
\[
\left|\alpha\frac{\varphi_{v}^{(l)}\left(a\right)}{(l-1)!}r_{\psi}^{1-l}z^{l-1}\right|\le2kC^{2k-1}\cdot2^{-n'(\chi-\delta)}.
\]
Moreover, by (\ref{eq:bounds on alpha}) and bounded distortion, we
get $\left|\alpha\varphi_{v}'\left(a\right)\right|\ge\frac{1}{2}C^{-1}$.
Since $C,k\ll n'$ and we can assume that $\delta<\chi/2$, all of
this implies that $\left|q'(z)\right|\ge\frac{1}{4}C^{-1}$. We have
thus shown that,
\begin{equation}
\Vert q\Vert_{2}\le\left(8kC^{2k}\right)^{1/2}\text{ and }\left|q'(z)\right|\ge\frac{1}{4}C^{-1}\text{ for }q\in\mathrm{supp}(\nu)\text{ and }z\in\mathbb{D}.\label{eq:ub on norm =000026 lb on der}
\end{equation}

Let $v_{1},v_{2}\in\mathcal{U}_{n,j}$ be distinct, and for $i=1,2$
set $q_{i}:=S_{\alpha r_{\psi}}T_{-\Pi(\omega)}R_{\psi^{-1}}P_{k,a}\varphi_{v_{i}}$.
Assume by contradiction that $\mathcal{D}_{Mn}^{\mathcal{P}_{k}}(q_{1})=\mathcal{D}_{Mn}^{\mathcal{P}_{k}}(q_{2})$.
By the definition of $\mathcal{D}_{Mn}^{\mathcal{P}_{k}}$ (see Section
\ref{subsec:Dyadic-partitions}), this implies that
\begin{equation}
\left|q_{1}(z)-q_{2}(z)\right|\le(k+1)2^{1-Mn}\text{ for }z\in\mathbb{D}.\label{eq:ub on dist of poly}
\end{equation}
On the other hand, by (\ref{eq:by exp sep}), we have $\Vert\varphi_{v_{1}u}-\varphi_{v_{2}u}\Vert_{\Omega}\ge c^{n+n'}$.
Thus, since $\psi\circ\varphi_{u}(\Omega)\subset\mathbb{D}$, there
exists $z_{0}\in\mathbb{D}$ such that $\psi^{-1}(z_{0})\in\varphi_{u}(\Omega)$
and
\[
\left|\varphi_{v_{1}}\left(\psi^{-1}(z_{0})\right)-\varphi_{v_{2}}\left(\psi^{-1}(z_{0})\right)\right|\ge c^{n+n'}.
\]
Additionally, from $\psi^{-1}(z_{0})\in\varphi_{u}(\Omega)$ and (\ref{eq:bd on remainder}),
\[
\left|\left(\varphi_{v_{i}}-P_{k,a}\varphi_{v_{i}}\right)\left(\psi^{-1}(z_{0})\right)\right|\le2^{-Mn-\chi(n+n')}\text{ for }i=1,2.
\]
From the last two inequalities and since $c^{-1}\ll M$,
\[
\left|P_{k,a}\varphi_{v_{1}}\left(\psi^{-1}(z_{0})\right)-P_{k,a}\varphi_{v_{2}}\left(\psi^{-1}(z_{0})\right)\right|\ge c^{n+n'}/2,
\]
which gives $\left|q_{1}(z_{0})-q_{2}(z_{0})\right|\ge c^{n+n'}/2$.
Since $c^{-1}\ll M$ and $k\ll n$, this contradicts (\ref{eq:ub on dist of poly}).
Hence we must have $\mathcal{D}_{Mn}^{\mathcal{P}_{k}}(q_{1})\ne\mathcal{D}_{Mn}^{\mathcal{P}_{k}}(q_{2})$,
which implies
\begin{equation}
\frac{1}{Mn}H\left(\nu,\mathcal{D}_{Mn}^{\mathcal{P}_{k}}\right)=\frac{1}{Mn}H\left(\left(\beta_{\omega}\right)_{F_{j,u}},\mathcal{C}_{n}\right)\ge\frac{\Delta}{6M},\label{eq:lb on ent of nu}
\end{equation}
where the last inequality follows from $(j,u)\in\mathcal{Q}$.

For $z\in\Omega_{0}$ we have $(\psi\circ\varphi_{u})'(z)=r_{\psi}\varphi_{u}'(z)$.
Thus, by Lemma \ref{lem:bounded distortion} and (\ref{eq:bounds on r_psi}),
\[
C^{-2}\le\left|(\psi\circ\varphi_{u})'(z)\right|\le C^{2},
\]
which implies that $\psi\circ\varphi_{u}\in\mathcal{F}_{1/C^{2}}$.
Also note that $\psi\circ\varphi_{u}(\Omega)\subset\mathbb{D}$. From
these facts, from (\ref{eq:ub on norm =000026 lb on der}) and (\ref{eq:lb on ent of nu}),
since $\dim\mu<2$, by the relations (\ref{eq:rel between params in main pf}),
and by Theorem \ref{thm:ent inc with poly},
\[
\frac{1}{Mn}H\left(\nu.\left(\psi\varphi_{u}\mu\right),\mathcal{D}_{Mn}\right)\ge\dim\mu+\rho.
\]
Hence, by (\ref{eq:lb on g(omega,j,u) bef ent enc}),
\begin{equation}
g(\omega,j,u)\ge\dim\mu+\rho-O(\delta)\text{ for all }(j,u)\in\mathcal{Q}.\label{eq:lb for (j,u) in Q}
\end{equation}

We can now complete the proof. From (\ref{eq:lb on g(omega) by g(omega,j,u)}),
(\ref{eq:simple lb by conc}), and (\ref{eq:lb for (j,u) in Q}),
\begin{eqnarray*}
g(\omega) & \ge & \sum_{(j,u)\in\mathcal{Q}}\beta_{\omega}\left(F_{j,u}\right)\left(\dim\mu+\rho-O(\delta)\right)\\
 & + & \sum_{(j,u)\in(J\times\mathcal{U}_{n'})\setminus\mathcal{Q}}\beta_{\omega}\left(F_{j,u}\right)\left(\dim\mu-O(\delta)\right).
\end{eqnarray*}
Thus, by (\ref{eq:lb on mass of (j,u ) in Q}) and since $\cup_{(j,u)\in J\times\mathcal{U}_{n'}}F_{j,u}=\left[\mathcal{U}_{n}\right]\cap\sigma^{-n}\left[\mathcal{U}_{n'}\right]$,
\[
g(\omega)\ge\frac{\rho\Delta}{6\log|\Lambda|}+\beta_{\omega}\left(\left[\mathcal{U}_{n}\right]\cap\sigma^{-n}\left[\mathcal{U}_{n'}\right]\right)\left(\dim\mu-O(\delta)\right),
\]
which holds for all $\omega\in E$. From this and by the definition
of $E$,
\[
g(\omega)\ge\frac{\rho\Delta}{6\log|\Lambda|}+\dim\mu-O(\delta)\text{ for all }\omega\in E.
\]
Hence, from (\ref{eq:lb on dim(mu) by g}) and since $\beta(E)>1-\delta$,
\[
\dim\mu+2\delta\ge(1-\delta)\left(\frac{\rho\Delta}{6\log|\Lambda|}+\dim\mu-O(\delta)\right).
\]
But since $\Delta^{-1},\rho^{-1}\ll\delta^{-1}$, this yields the
desired contradiction, completing the proof of the theorem.
\end{proof}

\section{\label{sec:Proof-of-main Corollary for sets}Proof of Corollary \ref{cor:main cor for sets}}

In this section we prove our result concerning the dimension of self-conformal
sets, Corollary \ref{cor:main cor for sets}. The proof requires some
preparation.
\begin{lem}
\label{lem:psi_1' almost like psi_2'}Let $n\ge1$ and let $\psi_{1},\psi_{2}:\Omega\rightarrow\mathbb{C}$
be holomorphic and injective. Suppose that $\psi_{j}\circ\varphi_{u}\circ\psi_{j}^{-1}:\psi_{j}(\Omega)\rightarrow\mathbb{C}$
is homothetic for each $j=1,2$ and $u\in\Lambda^{n}$. Then there
exists $0\ne w\in\mathbb{C}$ such that $\psi_{2}'(z)\mathbb{R}=\psi_{1}'(z)w\mathbb{R}$
for all $z\in K_{\Phi}$.
\end{lem}

\begin{proof}
For $j=1,2$ and $u\in\Lambda^{n}$ set $h_{j,u}:=\psi_{j}\circ\varphi_{u}\circ\psi_{j}^{-1}$.
Let $z\in K_{\Phi}$ and $\omega\in\Lambda^{\mathbb{N}}$ be with
$\Pi\omega=z$, and for $k\ge0$ set $u_{k}:=(\sigma^{kn}\omega)|_{n}$.
For $j=1,2$ and $m\ge1$, 
\[
\psi_{j}^{-1}\circ h_{j,u_{0}}\circ...\circ h_{j,u_{m-1}}\circ\psi_{j}=\varphi_{u_{0}}\circ...\circ\varphi_{u_{m-1}}=\varphi_{\omega|_{mn}}.
\]
Thus, since the maps $h_{j,u}$ are homothetic,
\begin{multline*}
\left(\psi_{j}^{-1}\circ h_{j,u_{0}}\circ...\circ h_{j,u_{m-1}}\circ\psi_{j}\right)'(0)\mathbb{R}=\left(\psi_{j}^{-1}\right)'\left(h_{j,u_{0}}\circ...\circ h_{j,u_{m-1}}\circ\psi_{j}(0)\right)\psi_{j}'(0)\mathbb{R}\\
=\left(\psi_{j}^{-1}\right)'\left(\psi_{j}\circ\varphi_{\omega|_{mn}}(0)\right)\psi_{j}'(0)\mathbb{R}=\left(\psi_{j}'\left(\varphi_{\omega|_{mn}}(0)\right)\right)^{-1}\psi_{j}'(0)\mathbb{R}.
\end{multline*}

From the last two equalities,
\[
\left(\psi_{1}'\left(\varphi_{\omega|_{mn}}(0)\right)\right)^{-1}\psi_{1}'(0)\mathbb{R}=\left(\psi_{2}'\left(\varphi_{\omega|_{mn}}(0)\right)\right)^{-1}\psi_{2}'(0)\mathbb{R}.
\]
Hence, by letting $m$ tend to $\infty$,
\[
\left(\psi_{1}'\left(z\right)\right)^{-1}\psi_{1}'(0)\mathbb{R}=\left(\psi_{2}'\left(z\right)\right)^{-1}\psi_{2}'(0)\mathbb{R},
\]
which completes the proof of the lemma with $w=\psi_{2}'(0)/\psi_{1}'(0)$.
\end{proof}
\begin{lem}
\label{lem:Phi_n or Phi_=00007Bn+1=00007D not conj to homo}Let $n\ge1$
be given. Then at least one of the IFSs $\left\{ \varphi_{u}\right\} _{u\in\Lambda^{n}}$
and $\left\{ \varphi_{u}\right\} _{u\in\Lambda^{n+1}}$ is not holomorphically
conjugate to a homothetic IFS.
\end{lem}

\begin{proof}
Assume by contradiction that the lemma is false. Then there exist
holomorphic and injective functions $\psi_{1},\psi_{2}:\Omega\rightarrow\mathbb{C}$
such that $\psi_{1}\circ\varphi_{u}\circ\psi_{1}^{-1}$ is homothetic
for each $u\in\Lambda^{n}$ and $\psi_{2}\circ\varphi_{u}\circ\psi_{2}^{-1}$
is homothetic for each $u\in\Lambda^{n+1}$. It is clear that $\psi_{j}\circ\varphi_{u}\circ\psi_{j}^{-1}$
is homothetic for each $j=1,2$ and $u\in\Lambda^{n(n+1)}$. Hence,
by Lemma \ref{lem:psi_1' almost like psi_2'}, there exists $0\ne w\in\mathbb{C}$
such that $\psi_{2}'(z)\mathbb{R}=\psi_{1}'(z)w\mathbb{R}$ for all
$z\in K_{\Phi}$. This also implies that for each $z\in K_{\Phi}$
\[
\left(\psi_{2}^{-1}\right)'(\psi_{2}(z))\mathbb{R}=\frac{1}{\psi_{2}'(z)}\mathbb{R}=\frac{1}{\psi_{1}'(z)w}\mathbb{R}=w^{-1}\left(\psi_{1}^{-1}\right)'(\psi_{1}(z))\mathbb{R}.
\]

Let $i\in\Lambda$ and set $g:=\psi_{2}\circ\varphi_{i}\circ\psi_{2}^{-1}$.
Given $u\in\Lambda^{n}$, it holds that $\psi_{2}\circ\varphi_{ui}\circ\psi_{2}^{-1}$
is homothetic. Hence, for $z\in K_{\Phi}$
\begin{multline*}
\mathbb{R}=\left(\psi_{2}\circ\varphi_{ui}\circ\psi_{2}^{-1}\right)'(\psi_{2}z)\mathbb{R}=\left(\psi_{2}\circ\varphi_{u}\circ\psi_{2}^{-1}\circ\psi_{2}\circ\varphi_{i}\circ\psi_{2}^{-1}\right)'(\psi_{2}z)\mathbb{R}\\
=\left(\psi_{2}'\left(\varphi_{ui}(z)\right)\right)\cdot\left(\varphi_{u}'\left(\varphi_{i}(z)\right)\right)\cdot\left(\left(\psi_{2}^{-1}\right)'\left(\psi_{2}\circ\varphi_{i}(z)\right)\right)\cdot\left(g'(\psi_{2}z)\right)\mathbb{R}.
\end{multline*}
Thus, since
\[
\psi_{2}'\left(\varphi_{ui}(z)\right)\mathbb{R}=\psi_{1}'\left(\varphi_{ui}(z)\right)w\mathbb{R}
\]
 and
\[
\left(\psi_{2}^{-1}\right)'\left(\psi_{2}\circ\varphi_{i}(z)\right)\mathbb{R}=w^{-1}\left(\psi_{1}^{-1}\right)'\left(\psi_{1}\circ\varphi_{i}(z)\right)\mathbb{R},
\]
we obtain that
\begin{multline*}
\mathbb{R}=\left(\psi_{1}'\left(\varphi_{ui}(z)\right)\right)\cdot\left(\varphi_{u}'\left(\varphi_{i}(z)\right)\right)\cdot\left(\left(\psi_{1}^{-1}\right)'\left(\psi_{1}\circ\varphi_{i}(z)\right)\right)\cdot\left(g'(\psi_{2}z)\right)\mathbb{R}\\
=\left(\left(\psi_{1}\circ\varphi_{u}\circ\psi^{-1}\right)'(\psi_{1}\circ\varphi_{i}(z))\right)\cdot\left(g'(\psi_{2}z)\right)\mathbb{R}=\left(g'(\psi_{2}z)\right)\mathbb{R},
\end{multline*}
where the last equality holds since $\psi_{1}\circ\varphi_{u}\circ\psi_{1}^{-1}$
is homothetic.

We have thus shown that $\left(\psi_{2}\circ\varphi_{i}\circ\psi_{2}^{-1}\right)'(\psi_{2}z)\mathbb{R}=\mathbb{R}$
for all $z\in K_{\Phi}$. Hence, by Lemma \ref{lem:lem bef no cobd prop},
it follows that $\psi_{2}\circ\varphi_{i}\circ\psi_{2}^{-1}$ is homothetic.
Since this holds for all $i\in\Lambda$, we obtain that $\Phi$ is
holomorphically conjugate to a homothetic IFS, which contradicts our
standing assumption and completes the proof of the lemma.
\end{proof}
\begin{lem}
\label{lem:Gamma_n not Phi_n invariant}Let $n\ge1$ be given, set
$\Phi_{n}:=\left\{ \varphi_{u}\right\} _{u\in\Lambda^{n}}$, and let
$\Gamma\subset\Omega$ be a regular real-analytic curve which is closed
in $\Omega$. Then $\Gamma$ is not $\Phi_{n}$-invariant.
\end{lem}

\begin{proof}
Assume by contradiction that $\Gamma$ is $\Phi_{n}$-invariant. It
is clear that $K_{\Phi}$ is also the self-conformal set associated
to $\Phi_{n}$. From this, and since $\Gamma$ is $\Phi_{n}$-invariant
and closed in $\Omega$, it follows easily that $K_{\Phi}\subset\Gamma$.
But since $\mu$ is supported on $K_{\Phi}$, this contradicts Proposition
\ref{prop:mu(Gamma)=00003D0}, which completes the proof of the lemma.
\end{proof}
\begin{proof}[Proof of Corollary \ref{cor:main cor for sets}]
Using \cite[Theorem 1.22]{MR2423393}, it can be shown that there
exists a unique $\sigma$-invariant $\nu\in\mathcal{M}\left(\Lambda^{\mathbb{N}}\right)$
such that $s(\Phi)=h(\nu)/\chi(\Phi,\nu)$. Here $h(\nu)$ is the
entropy of $\nu$, and $\chi(\Phi,\nu)$ is the Lyapunov exponent
associated to $\Phi$ and $\nu$, defined by
\[
\chi(\Phi,\nu):=-\int\log\left|\varphi_{\omega_{0}}'\left(\Pi(\sigma\omega)\right)\right|d\nu(\omega).
\]
Moreover, by the construction of $\nu$ (again see \cite{MR2423393}),
we have $\nu\left([u]\right)>0$ for $u\in\Lambda^{n}$.

Given $n\ge1$, set $\Phi_{n}:=\left\{ \varphi_{u}\right\} _{u\in\Lambda^{n}}$
and $p_{n}:=\left(\nu\left([u]\right)\right)_{u\in\Lambda^{n}}$,
and denote by $\mu_{n}$ the self-conformal measure corresponding
to $\Phi_{n}$ and $p_{n}$. The Lyapunov exponent $\chi(\Phi_{n},p_{n})$
associated to $\Phi_{n}$ and $p_{n}$ is defined as in (\ref{eq:def of Lyap expo}).
It is easy to verify that $\frac{1}{n}H(p_{n})\overset{n}{\rightarrow}h(\nu)$
and $\frac{1}{n}\chi(\Phi_{n},p_{n})\overset{n}{\rightarrow}\chi(\Phi,\nu)$,
and so
\begin{equation}
\underset{n\rightarrow\infty}{\lim}\left(H(p_{n})/\chi(\Phi_{n},p_{n})\right)=s(\Phi).\label{eq:conv of ratio to s(Phi)}
\end{equation}

Since $\Phi$ is exponentially separated and its maps do not have
a common fixed point, it is easy to verify that these two properties
also hold for $\Phi_{n}$ for every $n\ge1$. Thus, by Lemmas \ref{lem:Phi_n or Phi_=00007Bn+1=00007D not conj to homo}
and \ref{lem:Gamma_n not Phi_n invariant}, there exist arbitrarily
large $n\ge1$ for which the conditions of Theorem \ref{thm:main thm}
are satisfied with $\Phi_{n}$ in place of $\Phi$. Hence, by that
theorem and since $\mu_{n}$ is supported on $K_{\Phi}$ for all $n\ge1$,
\[
\min\left\{ 2,H(p_{n})/\chi(\Phi_{n},p_{n})\right\} =\dim\mu_{n}\le\dim_{H}K_{\Phi}
\]
for arbitrarily large $n\ge1$. This together with (\ref{eq:conv of ratio to s(Phi)})
shows that $\min\left\{ 2,s(\Phi)\right\} \le\dim_{H}K_{\Phi}$. Since
the reverse inequality always holds, this completes the proof of the
corollary.
\end{proof}
\bibliographystyle{plain}
\bibliography{../../bibfile.bib}

$\newline$\textsc{Zhou Feng, Department of Mathematics, Technion, Haifa, Israel}$\newline$\textit{E-mail: }
\texttt{zfeng@campus.technion.ac.il}

$\newline$\textsc{Ariel Rapaport, Department of Mathematics, Technion, Haifa, Israel}$\newline$\textit{E-mail: }
\texttt{arapaport@technion.ac.il}
\end{document}